\documentclass[reqno]{amsart}
\usepackage{amsmath, amsfonts, amsthm, amssymb, setspace, textcomp,
	fullpage, bbm}
\newtheorem{theorem}{Theorem}[section]
\newtheorem{lemma}[theorem]{Lemma}

\theoremstyle{definition}

\theoremstyle{remark}
\newtheorem{remark}[theorem]{Remark}

\numberwithin{equation}{section}

\newcommand{\spt}[1]{\mbox{\normalfont spt}\Parans{#1}}

\newcommand{\Parans}[1]{\left(#1\right)}
\newcommand{\CBrackets}[1]{\left\{#1\right\}}
\newcommand{\SBrackets}[1]{\left[#1\right]}

\newcommand{\aqprod}[3]{\Parans{#1;#2}_{#3}}
\newcommand{\jacprod}[2]{\SBrackets{#1;#2}_{\infty}}
\newcommand{\Jac}[2]{\left(\frac{#1}{#2}\right)}
\newcommand{\GEta}[3]{\eta_{#1,#2}\Parans{#3}}

\newcommand{\SSeries}[3]{S\Parans{#1,#2,#3}}

\author{CHRIS JENNINGS-SHAFFER}
\address{Department of Mathematics, University of Florida\\
Gainesville, Florida 32611, USA
\endgraf cjenningsshaffer@ufl.edu}



\title{Exotic Bailey-Slater SPT-Functions I: Group A}

\allowdisplaybreaks
\begin{document}

\allowdisplaybreaks

\begin{abstract}
We introduce several spt-type functions that arise from Bailey pairs.
We prove simple Ramanujan type congruences for these functions 
which can be explained by a spt-crank-type function. 
The spt-crank-type functions are constructed by adding an
extra variable $z$ into the generating
functions. We find dissections when $z$ is a certain root of unity, 
as has been done for many rank and crank
difference formulas of various partition type objects. Our formulas require
an identity of Chan \cite{Chan} on generalized Lambert series. 
\end{abstract}

\maketitle

\section{Introduction and statement of Results}
\allowdisplaybreaks

We recall a partition of a positive integer $n$ is a non-increasing sequence 
of positive integers that sum to $n$, we let $p(n)$ denote the number of 
partitions of $n$. For example $p(3)=3$ since the partitions of $3$ are
just $3$, $2+1$, and $1+1+1$. In \cite{Andrews} Andrews introduced
a weighted count on the partitions of $n$ given by counting each partition of
$n$ by the number of times the smallest part occurs. We call this weighted 
count $\spt{n}$ and note $\spt{3}=5$.

In \cite{AGL} Andrews, Garvan, and Liang gave combinatorial refinements of congruences for 
$\spt{n}$ by considering $S(z,q)$, a two variable generalization of the generating 
function of $\spt{n}$. They then applied Bailey's Lemma to recognize $S(z,q)$ as the
difference between the generating functions for the rank and crank of 
ordinary partitions. Based on information on the rank and crank of a partition,
they were able to deduce results for $\spt{n}$.

This method was used again by Garvan and the author in \cite{GarvanJennings} to
give combinatorial refinements and prove new congruences for the number of 
smallest parts in the overpartitions of $n$, the number of smallest parts in 
the overpartitions of $n$ with smallest part even, the number of smallest parts in 
the overpartitions of $n$ with smallest part odd, and the number of smallest parts in 
the partitions of $n$ with smallest part even and distinct odd parts.
The process has two key steps. The first is to use a Bailey pair and Bailey's 
lemma to see the difference of a rank and crank function, the second is to
use dissection formulas for the rank and crank function to deduce congruences.
In the cases of \cite{GarvanJennings}, the rank functions had been previously 
considered by Lovejoy and Osburn in \cite{LO1} and \cite{LO2}.

Here we first look to Bailey pairs to get the two variable generating
function and from that deduce what are the partition functions, rather than
starting with a partition function and trying to find a related Bailey pair.
We apply Bailey's Lemma to the two variable generalizations of each partition function
to get the difference of a series we can dissect and the generating
function for the crank of ordinary partitions. We then dissect the series
at roots of unity using methods similar to those used by Atkin and 
Swinnerton-Dyer in \cite{AS} for the rank of partitions. This method has also
been used by Ekin in studying the crank of partitions \cite{Ekin},
Lewis and Santa-Gadea in studying th rank and crank of partitions 
\cite{LewisSantaGadea}, and
by Lovejoy and Osburn in studying the rank of overpartitions \cite{LO1},
the $M_2$-rank of overpartitions \cite{LO3}, and the $M_2$-rank of partitions
without repeated odd parts \cite{LO2}.

We use the product notation
\begin{align*}
	&\aqprod{z}{q}{n} = \prod_{j=0}^{n-1} (1-zq^j),
	\hspace{20pt}
	\aqprod{z}{q}{\infty} = \prod_{j=0}^\infty (1-zq^j),
	\hspace{20pt}
	\jacprod{z}{q} = \aqprod{z,q/z}{q}{\infty},	
	\\
	&\aqprod{z_1,\dots,z_k}{q}{n} = \aqprod{z_1}{q}{n}\dots\aqprod{z_k}{q}{n},
	\hspace{20pt}
	\aqprod{z_1,\dots,z_k}{q}{\infty} = \aqprod{z_1}{q}{\infty}\dots\aqprod{z_k}{q}{\infty},
	\\
	&\jacprod{z_1,\dots,z_k}{q} = \jacprod{z_1}{q}\dots\jacprod{z_k}{q}
.
\end{align*}

We recall a pair of sequences $(\alpha,\beta)$ is a Bailey pair relative to 
$(a,q)$ if
\begin{align*}
	\beta_n = \sum_{k=0}^n \frac{\alpha_k}{\aqprod{q}{q}{n-k}\aqprod{aq}{q}{n+k}}
.
\end{align*}
A limiting case of Bailey's Lemma gives that if $(\alpha,\beta)$ is a Bailey 
pair relative to $(a,q)$ then
\begin{align*}
	\sum_{n=0}^\infty \aqprod{\rho_1,\rho_2}{q}{n} 
		\left(\frac{aq}{\rho_1\rho_2} \right)^n \beta_n
	&=
	\frac{\aqprod{aq/\rho_1,aq/\rho_2}{q}{\infty}}{\aqprod{aq,aq/\rho_1\rho_2}{q}{\infty}}
	\sum_{n=0}^\infty \frac{
		\aqprod{\rho_1,\rho_2}{q}{n} 
		\left(\frac{aq}{\rho_1\rho_2} \right)^n \alpha_n
		}{\aqprod{aq/\rho_1,aq/\rho_2}{q}{n}}	
.  
\end{align*}

We start with four Bailey pairs. Each $(\beta^i,\alpha^i)$ is a Bailey pairs relative to $(1,q)$,
these are out of group A of \cite{Slater}. For each $i$,
$\beta^i_0 = \alpha^i_0 = 1$ and for $n\ge 1$ the $\alpha^i_n$
and $\beta^i_n$ are defined by
$$\renewcommand{\arraystretch}{2}
\begin{array}{lll}
	\beta^1_n = \frac{1}{\aqprod{q}{q}{2n}},	
	\hspace{20pt}
	&\alpha^1_{3n} =  q^{6n^2-n} + q^{6n^2+n},
	\hspace{20pt}
	&\alpha^1_{3n\pm 1} = -q^{6n^2\pm 5n + 1}
	,\\
	\beta^2_n = \frac{q^n}{\aqprod{q}{q}{2n}},	
	&\alpha^2_{3n} =  q^{6n^2-2n} + q^{6n^2+2n},
	&\alpha^2_{3n\pm 1} = -q^{6n^2\pm 2n}
	,\\
	\beta^3_n = \frac{q^{n^2}}{\aqprod{q}{q}{2n}},	
	&\alpha^3_{3n} =  q^{3n^2-n} + q^{3n^2+n},
	&\alpha^3_{3n\pm 1} = -q^{3n^2\pm n}
	,\\
	\beta^4_n = \frac{q^{n^2-n}}{\aqprod{q}{q}{2n}},	
	&\alpha^4_{3n} =  q^{3n^2-2n} + q^{3n^2+2n},
	&\alpha^4_{3n\pm 1} = -q^{3n^2\pm 4n+1}
.
\end{array}
$$

For each $\beta^i$ we define a corresponding series,
\begin{align*}
	PP_1 (z,q)
	&=
	\frac{\aqprod{q}{q}{\infty}}{\aqprod{z,z^{-1}}{q}{\infty}}
	\sum_{n=0}^\infty \aqprod{z,z^{-1}}{q}{n} \beta_n^1 q^n
	-
	\frac{\aqprod{q}{q}{\infty}}{\aqprod{z,z^{-1}}{q}{\infty}}
	=
	\sum_{n=1}^\infty \frac{q^n\aqprod{q^{2n+1}}{q}{\infty}}
	{\aqprod{zq^n,z^{-1}q^n}{q}{\infty}}
	,\\			
	PP_2 (z,q)
	&=
	\frac{\aqprod{q}{q}{\infty}}{\aqprod{z,z^{-1}}{q}{\infty}}
	\sum_{n=0}^\infty \aqprod{z,z^{-1}}{q}{n} \beta_n^2 q^n
	-
	\frac{\aqprod{q}{q}{\infty}}{\aqprod{z,z^{-1}}{q}{\infty}}
	=
	\sum_{n=1}^\infty \frac{q^{2n}\aqprod{q^{2n+1}}{q}{\infty}}
	{\aqprod{zq^n,z^{-1}q^n}{q}{\infty}}
	,\\			
	PP_3 (z,q)
	&=
	\frac{\aqprod{q}{q}{\infty}}{\aqprod{z,z^{-1}}{q}{\infty}}
	\sum_{n=0}^\infty \aqprod{z,z^{-1}}{q}{n} \beta_n^3 q^n
	-
	\frac{\aqprod{q}{q}{\infty}}{\aqprod{z,z^{-1}}{q}{\infty}}
	=
	\sum_{n=1}^\infty \frac{q^{n^2+n}\aqprod{q^{2n+1}}{q}{\infty}}
	{\aqprod{zq^n,z^{-1}q^n}{q}{\infty}}
	,\\			
	PP_4 (z,q)
	&=
	\frac{\aqprod{q}{q}{\infty}}{\aqprod{z,z^{-1}}{q}{\infty}}
	\sum_{n=0}^\infty \aqprod{z,z^{-1}}{q}{n} \beta_n^4 q^n
	-
	\frac{\aqprod{q}{q}{\infty}}{\aqprod{z,z^{-1}}{q}{\infty}}
	=
	\sum_{n=1}^\infty \frac{q^{n^2}\aqprod{q^{2n+1}}{q}{\infty}}
	{\aqprod{zq^n,z^{-1}q^n}{q}{\infty}}
.
\end{align*}
We then set $z=1$ and simplify to get the series
\begin{align*}
	PP_1 (q)
	&=
	\sum_{n=1}^\infty pp_1(n)q^n
	=
	\sum_{n=1}^\infty \frac{q^n}
	{\aqprod{q^n}{q}{\infty}\aqprod{q^n}{q}{n+1}}
	,\\			
	PP_2 (q)
	&=
	\sum_{n=1}^\infty pp_2(n)q^n
	=
	\sum_{n=1}^\infty \frac{q^{2n}}
	{\aqprod{q^n}{q}{\infty}\aqprod{q^n}{q}{n+1}}
	,\\
	PP_3 (q)
	&=
	\sum_{n=1}^\infty pp_3(n)q^n
	=
	\sum_{n=1}^\infty \frac{q^{n^2+n}}
	{\aqprod{q^n}{q}{\infty}\aqprod{q^n}{q}{n+1}}
	,\\
	PP_4 (q)
	&=
	\sum_{n=1}^\infty pp_4(n)q^n
	=
	\sum_{n=1}^\infty \frac{q^{n^2}}
	{\aqprod{q^n}{q}{\infty}\aqprod{q^n}{q}{n+1}}
.
\end{align*}
We now interpret the $pp_i(n)$ in terms of the smallest parts of certain 
partition pairs and as partition pairs.
For a partition $\pi$ we let $\ell(\pi)$ denote the largest part, $s(\pi)$ the 
smallest part, and $|\pi|$ the sum of the parts. 
We say a pair of partitions $(\pi_1,\pi_2)$ is a partition pair of $n$ if
$|\pi_1|+|\pi_2|=n$.

We note that
\begin{align*}
	\frac{q^n}{(1-q^n)^2}
	&= 
	\sum_{m=1}^\infty m q^{nm}
.
\end{align*}
Thus $\frac{q^n}{(1-q^n)^2\aqprod{q^{n+1}}{q}{\infty}}$ is the generating 
function for the number of occurrences of the smallest part
in partitions with smallest part $n$. 

We see $pp_1(n)$ is the number of partition pairs $(\pi_1,\pi_2)$ of $n$, 
counted by the number of times $s(\pi_1)$ occurs,
where either $\pi_2$ is empty or $ s(\pi_1)< s(\pi_2)$ and 
$\ell(\pi_2)\le 2s(\pi_1)$. 
Alternatively, we can interpret $pp_1(n)$ as the number of
partition pairs $(\pi_1, \pi_2)$ of $n$, with $\pi_2$ allowed to be empty but
if it is not empty then $s(\pi_1)\le s(\pi_2)$ and 
$\ell(\pi_2)\le 2s(\pi_1)$. 

Similarly we see $pp_2(n)$ is the number of partition pairs $(\pi_1,\pi_2)$ of 
$n$ where the smallest part of $\pi_1$ occurs at 
least twice, counted by the number of times $s(\pi_1)$ occurs past the first,
where either $\pi_2$ is empty or $ s(\pi_1)< s(\pi_2)$ and $\ell(\pi_2)\le 2s(\pi_1)$. 
Alternatively, we can interpret $pp_2(n)$ as the number of
partition pairs $(\pi_1, \pi_2)$ where the smallest part of $\pi_1$ occurs at 
least twice, with $\pi_2$ allowed to be empty but if it is not empty  
then $s(\pi_1)\le s(\pi_2)$ and $\ell(\pi_2)\le 2s(\pi_1)$

We see $pp_3(n)$ is the number of partition pairs $(\pi_1,\pi_2)$ of 
$n$ where the smallest part of $\pi_1$ occurs
more than enough times to form a square in the Ferrers diagram,
counted by the number of times $s(\pi_1)$ occurs
past the first $s(\pi_1)$ times,
where either $\pi_2$ is empty or $ s(\pi_1)< s(\pi_2)$ and $\ell(\pi_2)\le 2s(\pi_1)$. 
Alternatively, $pp_3(n)$ is the number of
partition pairs $(\pi_1, \pi_2)$ where the smallest part of $\pi_1$ occurs 
more than enough times to form a square in the Ferrers diagram,
with $\pi_2$ allowed to be empty but if it is not empty  
then $s(\pi_1)\le s(\pi_2)$ and $\ell(\pi_2)\le 2s(\pi_1)$

We see $pp_4(n)$ is the number of partition pairs $(\pi_1,\pi_2)$ of 
$n$ where the smallest part of $\pi_1$ occurs
enough times to at least form a square in the Ferrers diagram, 
counted by the number of times $s(\pi_1)$ occurs
past the first $s(\pi_1)-1$ times,
where either $\pi_2$ is empty or $ s(\pi_1)< s(\pi_2)$ and $\ell(\pi_2)\le 2s(\pi_1)$. 
Alternatively, $pp_4(n)$ is the number of
partition pairs $(\pi_1, \pi_2)$ where the smallest part of $\pi_1$ occurs 
enough times to at least form a square in the Ferrers diagram, 
with $\pi_2$ allowed to be empty but if it is not empty  
then $s(\pi_1)\le s(\pi_2)$ and $\ell(\pi_2)\le 2s(\pi_1)$

There are of course other ways to interpret the series $PP_i(q)$. 
It is the partition pair interpretations that allow us to easily define cranks in the same
fashion of the $\overline{\mbox{crank}}$ defined in Section 3 of 
\cite{GarvanJennings}.

We will prove the following congruences.
\begin{theorem}
For $n\ge 0$,
\begin{align*}
	pp_1(3n) &\equiv 0 \pmod{3}
	,\\
	pp_2(3n+1) &\equiv 0 \pmod{3}
	,\\
	pp_2(5n+1) &\equiv 0 \pmod{5}
	,\\
	pp_3(5n+4) &\equiv 0 \pmod{5}
	,\\
	pp_3(7n+1) &\equiv 0 \pmod{7}
	,\\
	pp_4(5n+4) &\equiv 0 \pmod{5}
.
\end{align*}
\end{theorem}
We use $PP_i(z,q)$ to prove these congruences as follows. We write
\begin{align*}
	PP_i(z,q) &= \sum_{n=1}^\infty\sum_{m=-\infty}^\infty M_i(m,n)z^mq^n
,
\end{align*}
and define for any positive integer $t$
\begin{align*}
	M_i(k,t,n) = \sum_{m\equiv k\pmod{t}} M_i(m,n)
.
\end{align*}
Thus for any $t$ we have
\begin{align*}
	pp_i(n) &= \sum_{m=-\infty}^\infty M_i(m,n) = \sum_{k=0}^{t-1} M_i(k,t,n)
\end{align*}
and for $\zeta$ a $t^{th}$ root of unity we have
\begin{align*}
	PP_i(\zeta,q) &= \sum_{n=1}^\infty q^n \sum_{k=0}^{t-1}\zeta^kM_i(k,t,n).
\end{align*}
If $\ell$ is prime and $\zeta_\ell$ is a primitive $\ell^{th}$ root of unity, then
the minimal polynomial for $\zeta_\ell$ is $\sum_{k=0}^{\ell-1}x^k$. If the 
coefficient of $q^N$ in $PP_i(\zeta_\ell,q)$ is zero we would then have
$M_i(1,\ell,N)=M_i(2,\ell,N)=\dots=M_i(\ell-1,\ell,N)$. That is, if
the coefficient of $q^N$ in $PP_i(\zeta_\ell,q)$ is zero then
$pp_i(N) = \ell\cdot M_i(1,\ell,N)$ and $pp_i(N)$ is clearly divisible by
$\ell$.

To prove the congruences of Theorem 1 we then need to prove that the coefficients
of $q^{3n}$, $q^{3n+1}$, $q^{5n+1}$, $q^{5n+4}$, $q^{7n+1}$, and $q^{5n+4}$
are zero in $PP_1(\zeta_3,q)$, $PP_2(\zeta_3,q)$, $PP_2(\zeta_5,q)$, 
$PP_3(\zeta_5,q)$, $PP_3(\zeta_7,q)$, and $PP_4(\zeta_5,q)$ respectively. That
these coefficients are zero is not immediately obvious. 

We are actually proving something stronger than just the 
congruences because we are saying how to split up the numbers $pp_i(n)$. While
we can immediately read off a combinatorial interpretation of each $M_i(m,n)$ in
terms of partition triples, this is not particularly satisfying. We conclude 
this paper by defining a crank on each type of partition pair that agrees with
$M_i(m,n)$. Thus each congruence has a combinatorial explanation in terms of
the partition pairs counted by $pp_i$.

We begin by applying
Bailey's Lemma to each $PP_i(z,q)$.
We note in general that if $\alpha$ and $\beta$ form a Bailey pair with 
$\alpha_0 = \beta_0 = 1$ then
\begin{align*}
	&\frac{\aqprod{q}{q}{\infty}}{\aqprod{z,z^{-1}}{q}{\infty}}
	\sum_{n=0}^\infty \aqprod{z,z^{-1}}{q}{n} q^n \beta_n
	-\frac{\aqprod{q}{q}{\infty}}{\aqprod{z,z^{-1}}{q}{\infty}}
	\\
	&=
	\frac{1}{(1-z)(1-z^{-1})\aqprod{q}{q}{\infty}}
	\Parans{1 + \sum_{n=1}^\infty \frac{(1-z)(1-z^{-1})q^n\alpha_n}
		{(1-zq^n)(1-z^{-1}q^n)}
	}
	-
	\frac{\aqprod{q}{q}{\infty}}{\aqprod{z,z^{-1}}{q}{\infty}}
.
\end{align*}
We note
\begin{align*}
	\frac{1}{(1-zq^{-3n-1})(1-z^{-1}q^{-3n-1})}
	&=
	\frac{q^{6n+2}}{(1-zq^{3n+1})(1-z^{-1}q^{3n+1})}
	,\\
	\frac{1}{(1-zq^{-3n})(1-z^{-1}q^{-3n})}
	&=
	\frac{q^{6n}}{(1-zq^{3n})(1-z^{-1}q^{3n})}
.
\end{align*}

Thus
\begin{align*}
	&\sum_{n=1}\frac{(1-z)(1-z^{-1})q^n\alpha_n^1}{(1-zq^n)(1-z^{-1}q^n)}
	\\
	&=
	\sum_{n=1}\frac{(1-z)(1-z^{-1})q^{6n^2+2n}(1+q^{2n})}
		{(1-zq^{3n})(1-z^{-1}q^{3n})}
	-\sum_{n=1}\frac{(1-z)(1-z^{-1})q^{6n^2-2n}}
		{(1-zq^{3n-1})(1-z^{-1}q^{3n-1})}
	-\sum_{n=0}\frac{(1-z)(1-z^{-1})q^{6n^2+8n+2}}
		{(1-zq^{3n+1})(1-z^{-1}q^{3n+1})}
	\\
	&=	
	\sum_{n\not=0}\frac{(1-z)(1-z^{-1})q^{6n^2+2n}}
		{(1-zq^{3n})(1-z^{-1}q^{3n})}
	-\sum_{n=-\infty}^\infty \frac{(1-z)(1-z^{-1})q^{6n^2+8n+2}}
		{(1-zq^{3n+1})(1-z^{-1}q^{3n+1})}
,
\end{align*}
so that
\begin{align*}
	PP_1(z,q)
	&=
	\frac{1}{\aqprod{q}{q}{\infty}}
	\Parans{
		\sum_{n=-\infty}^\infty\frac{q^{6n^2+2n}}
			{(1-zq^{3n})(1-z^{-1}q^{3n})}
		-\sum_{n=-\infty}^\infty \frac{q^{6n^2+8n+2}}
			{(1-zq^{3n+1})(1-z^{-1}q^{3n+1})}
	}
	-\frac{\aqprod{q}{q}{\infty}}{\aqprod{z,z^{-1}}{q}{\infty}}
.
\end{align*}
In the same fashion, we find that
\begin{align*}
	PP_2(z,q)
	&=
	\frac{1}{\aqprod{q}{q}{\infty}}
	\Parans{
		\sum_{n=-\infty}^\infty\frac{q^{6n^2+n}}
			{(1-zq^{3n})(1-z^{-1}q^{3n})}
		-\sum_{n=-\infty}^\infty \frac{q^{6n^2+5n+1}}
			{(1-zq^{3n+1})(1-z^{-1}q^{3n+1})}
	}
	-\frac{\aqprod{q}{q}{\infty}}{\aqprod{z,z^{-1}}{q}{\infty}}
	,\\
	PP_3(z,q)
	&=
	\frac{1}{\aqprod{q}{q}{\infty}}
	\Parans{
		\sum_{n=-\infty}^\infty\frac{q^{3n^2+2n}}
			{(1-zq^{3n})(1-z^{-1}q^{3n})}
		-\sum_{n=-\infty}^\infty \frac{q^{3n^2+4n+1}}
			{(1-zq^{3n+1})(1-z^{-1}q^{3n+1})}
	}
	-\frac{\aqprod{q}{q}{\infty}}{\aqprod{z,z^{-1}}{q}{\infty}}
	,\\
	PP_4(z,q)
	&=
	\frac{1}{\aqprod{q}{q}{\infty}}
	\Parans{
		\sum_{n=-\infty}^\infty\frac{q^{3n^2+n}}
			{(1-zq^{3n})(1-z^{-1}q^{3n})}
		-\sum_{n=-\infty}^\infty \frac{q^{3n^2+7n+2}}
			{(1-zq^{3n+1})(1-z^{-1}q^{3n+1})}
	}
	-\frac{\aqprod{q}{q}{\infty}}{\aqprod{z,z^{-1}}{q}{\infty}}
.
\end{align*}

We determine dissection formulas for the series without
$\frac{\aqprod{q}{q}{\infty}}{\aqprod{z,z^{-1}}{q}{\infty}}$
at roots of unity. Then we use known formulas for 
$\frac{\aqprod{q}{q}{\infty}}{\aqprod{zq,z^{-1}q}{q}{\infty}}$ at roots of unity
to give formulas for the $PP_i(z,q)$. Although it would be possible to use
the methods of this paper to determine all of the dissections of 
$PP_1(z,q)$, $PP_2(z,q)$, $PP_3(z,q)$, and $PP_4(z,q)$ at
$z=\zeta_3$, $\zeta_5$, and $\zeta_7$, we only do so for the cases
that lead to a congruence.

We define
\begin{align*}
	S(z,w,q) &= \sum_{n=-\infty}^\infty \frac{q^{2n(n+1)}w^n}{1-zq^n}
	,
	&S^*(w,q) = \sum_{n\not=0} \frac{q^{2n(n+1)}w^n}{1-q^n}
.
\end{align*}
Although they do not appear in the statement of our theorems, we will
also need the following series,
\begin{align*}
	T(z,w,q) &= \sum_{n=-\infty}^\infty \frac{q^{n(n+1)}w^n}{1-zq^n}
	, 
	&T^*(w,q) &= \sum_{n\not=0} \frac{q^{n(n+1)}w^n}{1-q^n}
.
\end{align*} 
For integers $a$, $b$, $c$ we use the abuse of notation
$\SSeries{a}{b}{c} = \SSeries{q^a}{q^b}{q^c}$, 
$S^*(b,c) = S^*(q^b,q^c)$,
$T(a,b,c) = T(q^a,q^b,q^c)$, and $T^*(b,c) = T^*(q^b,q^c)$.

\begin{theorem}\label{TheoremPP13}
\begin{align*}
	&\frac{1}{\aqprod{q}{q}{\infty}}
	\Parans{
		\sum_{n=-\infty}^\infty\frac{q^{6n^2+2n}}
			{(1-\zeta_3q^{3n})(1-\zeta_3^{-1}q^{3n})}
		-\sum_{n=-\infty}^\infty \frac{q^{6n^2+8n+2}}
			{(1-\zeta_3q^{3n+1})(1-\zeta_3^{-1}q^{3n+1})}
	}
	\\
	&=
	A_{130}(q^3) + qA_{131}(q^3) + q^2A_{132}(q^3)
,
\end{align*}
where 
\begin{align*}
	A_{130}(q) &=
		\frac{1}{3}\frac{\aqprod{q^{9}}{q^{9}}{\infty}^2\jacprod{q^{4}}{q^{9}}^2}
			{\aqprod{q}{q}{\infty}}
		+
		\frac{2}{3}q\frac{\aqprod{q^{9}}{q^{9}}{\infty}^2\jacprod{q,q^2}{q^{9}}}
			{\aqprod{q}{q}{\infty}}
	,\\
	A_{131}(q)
	&=
		\frac{q}{\aqprod{q^{9}}{q^{9}}{\infty}\jacprod{q^{4}}{q^{9}}}
		\Parans{S(-3, -28, 9) - q^{28}S(11, 28, 9) }	
		+\frac{1}{3}\frac{\aqprod{q^{9}}{q^{9}}{\infty}^2\jacprod{q^2,q^{4}}{q^{9}}}
			{\aqprod{q}{q}{\infty}}
		\\&\quad
		+\frac{1}{3} q\frac{\aqprod{q^{9}}{q^{9}}{\infty}^2\jacprod{q}{q^{9}}^2}
			{\aqprod{q}{q}{\infty}}
	,\\
	A_{132}(q)
	&=
		-\frac{q^2}{\aqprod{q^{9}}{q^{9}}{\infty}\jacprod{q}{q^{9}}}
		\Parans{S(3, -2, 9) - q^{2}S(4, 2, 9) }
		+
		\frac{1}{3}\frac{\aqprod{q^{9}}{q^{9}}{\infty}^2\jacprod{q,q^4}{q^{9}}}
			{\aqprod{q}{q}{\infty}}
		\\&\quad
		-\frac{2}{3}\frac{\aqprod{q^{9}}{q^{9}}{\infty}^2\jacprod{q^2}{q^{9}}^2}
			{\aqprod{q}{q}{\infty}}
.
\end{align*}
\end{theorem}

\begin{theorem}\label{TheoremPP23}
\begin{align*}
	&
	\frac{1}{\aqprod{q}{q}{\infty}}
	\Parans{
		\sum_{n=-\infty}^\infty\frac{q^{6n^2+n}}
			{(1-\zeta_3q^{3n})(1-\zeta_3^{-1}q^{3n})}
		-\sum_{n=-\infty}^\infty \frac{q^{6n^2+5n+1}}
			{(1-\zeta_3q^{3n+1})(1-\zeta_3^{-1}q^{3n+1})}
	}
	\\
	&=
	A_{230}(q^3) + qA_{231}(q^3) + q^2A_{232}(q^3)
,
\end{align*}
where 
\begin{align*}
	A_{230}(q) 
	&= 
		\frac{q^{-1}}{\aqprod{q^{9}}{q^{9}}{\infty}\jacprod{q^{2}}{q^{9}}}
			\left(\SSeries{-3}{-22}{9}-q^{22}\SSeries{8}{22}{9}\right)
		-\frac{2}{3}\frac{\aqprod{q^{9}}{q^{9}}{\infty}^2\jacprod{q^{4}}{q^{9}}^2}
			{\aqprod{q}{q}{\infty}} 
		\\&\quad
		-\frac{1}{3}q\frac{\aqprod{q^{9}}{q^{9}}{\infty}^2\jacprod{q,q^2}{q^{9}}}
			{\aqprod{q}{q}{\infty}} 	
	,\\
	A_{231}(q) 
	&= 
		-\frac{2}{3}\frac{\aqprod{q^{9}}{q^{9}}{\infty}^2\jacprod{q^2,q^{4}}{q^{9}}}
			{\aqprod{q}{q}{\infty}} 
		+\frac{1}{3}q\frac{\aqprod{q^{9}}{q^{9}}{\infty}^2\jacprod{q}{q^{9}}^2}
			{\aqprod{q}{q}{\infty}}
	,\\
	A_{232}(q)
	&=
		\frac{q^{2}}{\aqprod{q^{9}}{q^{9}}{\infty}\jacprod{q}{q^{9}}}
			\left(\SSeries{3}{-2}{9}-q^{2}\SSeries{4}{2}{9}\right)
		+\frac{1}{3}\frac{\aqprod{q^{9}}{q^{9}}{\infty}^2\jacprod{q,q^{4}}{q^{9}}}
			{\aqprod{q}{q}{\infty}} 
		\\&\quad
		+\frac{1}{3}\frac{\aqprod{q^{9}}{q^{9}}{\infty}^2\jacprod{q^2}{q^{9}}^2}
			{\aqprod{q}{q}{\infty}}
.
\end{align*}
\end{theorem}

\begin{theorem}\label{TheoremPP25}
\begin{align*}
	&
	\frac{1}{\aqprod{q}{q}{\infty}}
	\Parans{
		\sum_{n=-\infty}^\infty\frac{q^{6n^2+n}}
			{(1-\zeta_5q^{3n})(1-\zeta_5^{-1}q^{3n})}
		-\sum_{n=-\infty}^\infty \frac{q^{6n^2+5n+1}}
			{(1-\zeta_5zq^{3n+1})(1-\zeta_5^{-1}q^{3n+1})}
	}
	\\	
	&=
	A_{250}(q^5) + qA_{251}(q^5) + q^2A_{252}(q^5) + q^3A_{253}(q^5)+ q^4A_{254}(q^5)
,
\end{align*}
where 
\begin{align*}
	A_{250}(q) 
	&= 
		-\frac{q^{-3}}{\aqprod{q^{15}}{q^{15}}{\infty}\jacprod{q}{q^{15}}}
		\Parans{S(-3,-28,15) - q^{28}S(11,28,15)}
		\\&\quad
		+(\zeta_5+\zeta_5^4)\frac{q^{2}}{\aqprod{q^{15}}{q^{15}}{\infty}\jacprod{q^7}{q^{15}}}
		\Parans{S(3,-14,15) - q^{14}S(10,14,15)}
		\\&\quad
		+
		\frac{3+\zeta_5+\zeta_5^4}{5}\frac{\aqprod{q^{5}}{q^{5}}{\infty}\jacprod{q^2}{q^5}}
			{\jacprod{q}{q^5}^2}
		-
		\frac{\aqprod{q^{15}}{q^{15}}{\infty}\jacprod{q^2}{q^5}}
			{\jacprod{q^2,q^3}{q^{15}}\jacprod{q}{q^5}}
		-
		(\zeta_5+\zeta_5^4)q^2\frac{\aqprod{q^{15}}{q^{15}}{\infty}\jacprod{q}{q^{15}}}
			{\jacprod{q^3,q^4,q^5}{q^{15}}}
	,\\
	A_{251}(q) 
	&=
		\frac{-2+\zeta_5+\zeta_5^4}{5}\frac{\aqprod{q^5}{q^5}{\infty}}{\jacprod{q}{q^{5}}}
	,\\
	A_{252}(q)
	&=
		\frac{1-3(\zeta_5+\zeta_5^4)}{5}\frac{\aqprod{q^5}{q^5}{\infty}}{\jacprod{q^2}{q^{5}}}
	,\\
	A_{253}(q)
	&=
		(\zeta_5+\zeta_5^4)\frac{q^{-2}}{\aqprod{q^{15}}{q^{15}}{\infty}\jacprod{q^2}{q^{15}}}
		\Parans{\SSeries{-3}{-34}{15} - q^{34}\SSeries{14}{34}{15}}
		\\&\quad
		+
		\frac{-1+3(\zeta_5+\zeta_5^4)}{5}\frac{\aqprod{q^5}{q^5}{\infty}\jacprod{q}{q^5}}
			{\jacprod{q^2}{q^{5}}^2}
		-(\zeta_5+\zeta_5^4)q^{-1}\frac{\aqprod{q^{15}}{q^{15}}{\infty}\jacprod{q^4,q^6}{q^{15}}}
			{\jacprod{q,q^3,q^3,q^7}{q^{15}}}
	,\\
	A_{254}(q)
	&=
		\frac{q^{2}}{\aqprod{q^{15}}{q^{15}}{\infty}\jacprod{q^4}{q^{15}}}
		\Parans{\SSeries{3}{-8}{15} - q^{8}\SSeries{7}{8}{15}}
		+
		\frac{\aqprod{q^{15}}{q^{15}}{\infty}}{\jacprod{q^3,q^4}{q^{15}}}
.
\end{align*}
\end{theorem}

\begin{theorem}\label{TheoremPP35}
\begin{align*}
	&\frac{1}{\aqprod{q}{q}{\infty}}
	\Parans{
		\sum_{n=-\infty}^\infty\frac{q^{3n^2+2n}}
			{(1-\zeta_5q^{3n})(1-\zeta_5^{-1}q^{3n})}
		-\sum_{n=-\infty}^\infty \frac{q^{3n^2+4n+1}}
			{(1-\zeta_5q^{3n+1})(1-\zeta_5^{-1}q^{3n+1})}
	}
	\\
	&= \aqprod{q^{25}}{q^{25}}{\infty}
	\left(
		A_{350}(q^5)+qA_{351}(q^5)+q^2A_{352}(q^5)+q^3A_{353}(q^5)
	\right)
,
\end{align*}
where
\begin{align*}
	A_{350}(q)
	&=
		\frac{3+\zeta_5+\zeta_5^4}{5}\frac{\jacprod{q^2}{q^5}}{\jacprod{q}{q^5}^2} 
    	-q\frac{\jacprod{q^2}{q^{15}}}{\jacprod{q,q^4,q^5,q^6}{q^{15}}}
	,\\
	A_{351}(q) 
	&= 
		\frac{-2+\zeta_5+\zeta_5^4}{5}\frac{1}{\jacprod{q}{q^5}} 
    	-(\zeta_5+\zeta_5^4)q\frac{1}{\jacprod{q^3,q^5,q^7}{q^{15}}}
	,\\
	A_{352}(q) 
	&= 
		\frac{1+2(\zeta_5+\zeta_5^4)}{5}\frac{1}{\jacprod{q^2}{q^5}} 
    	- (\zeta_5+\zeta_5^4)\frac{1}{\jacprod{q,q^5,q^6}{q^{15}}}
	,\\
	A_{353}(q) 
	&=
		\frac{-1+3(\zeta_5+\zeta_5^4)}{5}\frac{\jacprod{q}{q^5}}{\jacprod{q^2}{q^5}^2} 
    	+(-1+\zeta_5+\zeta_5^4)q\frac{\jacprod{q}{q^{15}}}{\jacprod{q^2,q^3,q^5,q^7}{q^{15}}}
.
\end{align*}
\end{theorem}

\begin{theorem}\label{TheoremPP45}
\begin{align*}
	&\frac{1}{\aqprod{q}{q}{\infty}}
	\Parans{
		\sum_{n=-\infty}^\infty\frac{q^{3n^2+n}}
			{(1-\zeta_5q^{3n})(1-\zeta_5^{-1}q^{3n})}
		-\sum_{n=-\infty}^\infty \frac{q^{3n^2+7n+2}}
			{(1-\zeta_5q^{3n+1})(1-\zeta_5^{-1}q^{3n+1})}
	}
	\\
	&= \aqprod{q^{25}}{q^{25}}{\infty}
	\left(
		A_{450}(q^5)+qA_{451}(q^5)+q^2A_{452}(q^5)+q^3A_{453}(q^5)
	\right)
,
\end{align*}
where
\begin{align*}
	A_{450}(q)
	&=
		\frac{3+\zeta_5+\zeta_5^4}{5}\frac{\jacprod{q^2}{q^5}}{\jacprod{q}{q^5}^2} 
    	-q(1+\zeta_5+\zeta_5^4)\frac{\jacprod{q^2}{q^{15}}}{\jacprod{q,q^4,q^5,q^6}{q^{15}}}
	,\\
	A_{451}(q) 
	&= 
		\frac{-2+\zeta_5+\zeta_5^4}{5}\frac{1}{\jacprod{q}{q^5}} 
    	+\frac{1}{\jacprod{q^2,q^3,q^5}{q^{15}}}
	,\\
	A_{452}(q) 
	&= 
		\frac{1+2(\zeta_5+\zeta_5^4)}{5}\frac{1}{\jacprod{q^2}{q^5}} 
    	- \frac{1}{\jacprod{q,q^5,q^6}{q^{15}}}
	,\\
	A_{453}(q) 
	&=
	\frac{-1-2(\zeta_5+\zeta_5^4)}{5}\frac{\jacprod{q}{q^5}}{\jacprod{q^2}{q^5}^2} 
    - (\zeta+\zeta^4)q\frac{\jacprod{q}{q^{15}}}{\jacprod{q^2,q^3,q^5,q^7}{q^{15}}}
.
\end{align*}
\end{theorem}

\begin{theorem}\label{TheoremPP37}
\begin{align*}
	&
	\frac{1}{\aqprod{q}{q}{\infty}}
	\Parans{
		\sum_{n=-\infty}^\infty\frac{q^{3n^2+2n}}
			{(1-\zeta_7q^{3n})(1-\zeta_7^{-1}q^{3n})}
		-\sum_{n=-\infty}^\infty \frac{q^{3n^2+4n+1}}
			{(1-\zeta_7q^{3n+1})(1-\zeta_7^{-1}q^{3n+1})}
	}
	\\
	&= \aqprod{q^{49}}{q^{49}}{\infty}
	\left(
		A_{370}(q^7)+qA_{371}(q^7)+q^2A_{372}(q^7)+q^3A_{373}(q^7)+q^4A_{374}(q^7)
		+q^5A_{375}(q^7)+q^6A_{376}(q^6)
	\right)
,
\end{align*}
where
\begin{align*}
	A_{370}(q)
	&=
		\frac{6+3(\zeta_7+\zeta_7^6)+(\zeta_7^2+\zeta_7^5)}{7}		
		\frac{\jacprod{q^3}{q^7}}{\jacprod{q,q^2}{q^7}}
		-q^2(1+\zeta+\zeta^6)\frac{1}{\jacprod{q^6,q^7,q^8}{q^{21}}}	
	,\\
	A_{371}(q)
	&=
		\frac{-1+3(\zeta_7+\zeta_7^6)+(\zeta_7^2+\zeta_7^5)}{7}		
		\frac{1}{\jacprod{q}{q^7}}
	,\\
	A_{372}(q)
	&=
		\frac{5-(\zeta_7+\zeta_7^6)+2(\zeta_7^2+\zeta_7^5)}{7}		
		\frac{\jacprod{q^2}{q^7}}{\jacprod{q,q^3}{q^7}}
		-q(1+\zeta^2+\zeta^5)\frac{1}{\jacprod{q^3,q^7,q^{10}}{q^{21}}}	
	,\\
	A_{373}(q)
	&=
		\frac{-3+2(\zeta_7+\zeta_7^6)-4(\zeta_7^2+\zeta_7^5)}{7}		
		\frac{1}{\jacprod{q^2}{q^7}}
	,\\
	A_{374}(q)
	&=
		\frac{2+(\zeta_7+\zeta_7^6)+5(\zeta_7^2+\zeta_7^5)}{7}		
		\frac{1}{\jacprod{q^3}{q^7}}
	,\\
	A_{375}(q)
	&=
		(\zeta+\zeta^6)\frac{\jacprod{q^5}{q^{21}}}{\jacprod{q^3,q^4,q^7,q^7}{q^{21}}}
		+q^2(\zeta+\zeta^6)\frac{\jacprod{q^2}{q^{21}}}{\jacprod{q^3,q^7,q^7,q^{10}}{q^{21}}}		
	,\\
	A_{376}(q)
	&=
		\frac{-4-2(\zeta_7+\zeta_7^6)-3(\zeta_7^2+\zeta_7^5)}{7}		
		\frac{\jacprod{q}{q^7}}{\jacprod{q^2,q^3}{q^7}}
		+\frac{1+(1+\zeta_7^2+\zeta_7^5)}{\jacprod{q^2,q^7,q^9}{q^{21}}}	
.
\end{align*}
\end{theorem}

We recognize
$\frac{\aqprod{q}{q}{\infty}}{\aqprod{zq,q/z}{q}{\infty}}$ as the generating function for the crank
of partitions. The dissections for the crank at roots of
unity are well known by the work of Garvan \cite{Garvan1}. 
We start with $\zeta_3$. Since 
$\frac{1}{\aqprod{\zeta_3q,\zeta_3^{-1}q}{q}{\infty}}
=\frac{\aqprod{q}{q}{\infty}}{\aqprod{q^3}{q^3}{\infty}}$, 
and by Euler's Pentagonal Numbers Theorem along with the Jacobi Triple
Product Identity
\begin{align*}
	\aqprod{q}{q}{\infty}
	&=
	\aqprod{q^{27}}{q^{27}}{\infty}
	\Parans{
		\jacprod{q^{12}}{q^{27}}
		-q\jacprod{q^{6}}{q^{27}}
		-q^2\jacprod{q^{3}}{q^{27}}
	}
,
\end{align*}
we have
\begin{align}
	\label{EquationCrank3Dissection}
	\frac{\aqprod{q}{q}{\infty}}{\aqprod{\zeta_3q,\zeta_3^{-1}q}{q}{\infty}}
	&=
	\frac{\aqprod{q^{27}}{q^{27}}{\infty}^2}{\aqprod{q^3}{q^3}{\infty}}
	\left(
		\jacprod{q^{12}}{q^{27}}^2 + 2q^3\jacprod{q^3,q^6}{q^{27}} 	
		-2q\jacprod{q^6,q^{12}}{q^{27}}+q^4\jacprod{q^3}{q^{27}}^2
		\right.\nonumber\\&\quad\left.
		-2q^2\jacprod{q^3,q^{12}}{q^{27}}+q^2\jacprod{q^6}{q^{27}}^2
		\right)
.
\end{align}
As in (3.8) and Theorem 5.1 of \cite{Garvan1} we have 
\begin{align}
	\label{EquationCrank5Dissection}
	\frac{\aqprod{q}{q}{\infty}}
		{\aqprod{\zeta_5q,\zeta_5^{-1}q}{q}{\infty}}
	&=
	\aqprod{q^{25}}{q^{25}}{\infty}
	\left(
		\frac{\jacprod{q^{10}}{q^{25}}}{\jacprod{q^{5}}{q^{25}}^2}
		+(\zeta_5+\zeta_5^4-1)q\frac{1}{\jacprod{q^{5}}{q^{25}}}
		-(\zeta_5+\zeta_5^4+1)q^2\frac{1}{\jacprod{q^{10}}{q^{25}}}
		\right.\nonumber\\&\left.\quad
		-(\zeta_5+\zeta_5^4)q^3\frac{\jacprod{q^{5}}{q^{25}}}{\jacprod{q^{10}}{q^{25}}^2}
	\right)
	,\\
	\label{EquationCrank7Dissection}
	\frac{\aqprod{q}{q}{\infty}}
		{\aqprod{\zeta_7q,\zeta_7^{-1}q}{q}{\infty}}
	&=
	\aqprod{q^{49}}{q^{49}}{\infty}
	\left(
		\frac{\jacprod{q^{21}}{q^{49}}}{\jacprod{q^{7},q^{14}}{q^{49}}}
		+(\zeta_7+\zeta_7^6-1)q\frac{1}{\jacprod{q^{7}}{q^{49}}}
		+(\zeta_7^2+\zeta_7^5)q^2\frac{\jacprod{q^{14}}{q^{49}}}{\jacprod{q^7,q^{21}}{q^{49}}}
		\right.\nonumber\\&\left.\quad
		+(\zeta_7^3+\zeta_7^4+1)q^3\frac{1}{\jacprod{q^{14}}{q^{49}}}
		-(\zeta_7+\zeta_7^6)q^4\frac{1}{\jacprod{q^{21}}{q^{49}}}
		-(\zeta_7^2+\zeta_7^5+1)q^6\frac{\jacprod{q^{7}}{q^{49}}}{\jacprod{q^{14},q^{21}}{q^{49}}}
	\right)
.
\end{align}

Dividing (\ref{EquationCrank3Dissection}) by $(1-\zeta_3)(1-\zeta_3^{-1})$,
(\ref{EquationCrank5Dissection}) by $(1-\zeta_5)(1-\zeta_5^{-1})$, and
(\ref{EquationCrank7Dissection}) by $(1-\zeta_7)(1-\zeta_7^{-1})$
along with Theorems \ref{TheoremPP13} through \ref{TheoremPP37} gives the
following formulas for $PP_1(z,q)$, $PP_2(z,q)$, $PP_3(z,q)$, and 
$PP_4(z,q)$.

\begin{theorem}
\begin{align*}
	PP_1(\zeta_3,q)
	&=
		\frac{q^4}{\aqprod{q^{27}}{q^{27}}{\infty}\jacprod{q^{12}}{q^{27}}}
		\Parans{S(-9, -84, 27) - q^{84}S(33, 84, 27) }	
		+
		q\frac{\aqprod{q^{27}}{q^{27}}{\infty}^2\jacprod{q^6,q^{12}}{q^{27}}}
			{\aqprod{q^3}{q^3}{\infty}}
	,\\&\quad
		-\frac{q^8}{\aqprod{q^{27}}{q^{27}}{\infty}\jacprod{q^3}{q^{27}}}
		\Parans{S(9, -6, 27) - q^{6}S(12, 6, 27) }
		+
		q^2\frac{\aqprod{q^{27}}{q^{27}}{\infty}^2\jacprod{q^3,q^{12}}{q^{27}}}
			{\aqprod{q^3}{q^3}{\infty}}
		\\&\quad
		-q^2\frac{\aqprod{q^{27}}{q^{27}}{\infty}^2\jacprod{q^6}{q^{27}}^2}
			{\aqprod{q^3}{q^3}{\infty}}
.
\end{align*}
\end{theorem}
\begin{remark}
The congruence $pp_1(3n)\equiv 0\pmod{3}$ then follows as 
$PP_1(\zeta_3,q)$ has no $q^{3n}$ terms.
\end{remark}

\begin{theorem}
\begin{align*}
	PP_2(\zeta_3,q)	
	&=
		\frac{q^{-3}}{\aqprod{q^{27}}{q^{27}}{\infty}\jacprod{q^{6}}{q^{27}}}
			\left(\SSeries{-9}{-66}{27}-q^{66}\SSeries{24}{66}{27}\right)
		-\frac{\aqprod{q^{27}}{q^{27}}{\infty}^2\jacprod{q^{12}}{q^{27}}^2}
			{\aqprod{q^3}{q^3}{\infty}} 
		\\&\quad
		-q^3\frac{\aqprod{q^{27}}{q^{27}}{\infty}^2\jacprod{q^3,q^6}{q^{27}}}
			{\aqprod{q^3}{q^3}{\infty}} 	
		+\frac{q^{8}}{\aqprod{q^{27}}{q^{27}}{\infty}\jacprod{q^3}{q^{27}}}
			\left(\SSeries{9}{-6}{27}-q^{6}\SSeries{12}{6}{27}\right)
		\\&\quad
		+q^2\frac{\aqprod{q^{27}}{q^{27}}{\infty}^2\jacprod{q^3,q^{12}}{q^{27}}}
			{\aqprod{q^3}{q^3}{\infty}} 
.
\end{align*}
\end{theorem}
\begin{remark}
The congruence $pp_2(3n+1)\equiv 0\pmod{3}$ then follows as 
$PP_2(\zeta_3,q)$ has no $q^{3n+1}$ terms.
\end{remark}

\begin{theorem}
\begin{align*}
	PP_2(\zeta_5,q)
	&=
		-\frac{q^{-15}}{\aqprod{q^{75}}{q^{75}}{\infty}\jacprod{q^5}{q^{75}}}
			\Parans{S(-15,-140,75) - q^{140}S(55,140,75)}
	\\&\quad
		+(\zeta_5+\zeta_5^4)\frac{q^{10}}{\aqprod{q^{75}}{q^{75}}{\infty}\jacprod{q^{35}}{q^{75}}}
			\Parans{S(15,-70,75) - q^{70}S(50,70,75)}
	\\&\quad
		-
		\frac{\aqprod{q^{75}}{q^{75}}{\infty}\jacprod{q^{10}}{q^{25}}}
			{\jacprod{q^{10},q^{15}}{q^{75}}\jacprod{q^5}{q^{25}}}
		-
		(\zeta_5+\zeta_5^4)q^{10}\frac{\aqprod{q^{75}}{q^{75}}{\infty}\jacprod{q^5}{q^{75}}}
			{\jacprod{q^{15},q^{20},q^{25}}{q^{75}}}	
		+
		q^2\frac{\aqprod{q^{25}}{q^{25}}{\infty}}{\jacprod{q^{10}}{q^{25}}}
	\\&\quad
		+(\zeta_5+\zeta_5^4)\frac{q^{-7}}{\aqprod{q^{75}}{q^{75}}{\infty}\jacprod{q^{10}}{q^{75}}}
			\Parans{\SSeries{-15}{-170}{75} - q^{170}\SSeries{70}{170}{75}}
	\\&\quad
		+(\zeta_5+\zeta_5^4)q^3\frac{\aqprod{q^{25}}{q^{25}}{\infty}\jacprod{q^5}{q^{25}}}
			{\jacprod{q^{10}}{q^{25}}^2}
		-(\zeta_5+\zeta_5^4)q^{-2}\frac{\aqprod{q^{75}}{q^{75}}{\infty}\jacprod{q^{20},q^{30}}{q^{75}}}
			{\jacprod{q^5,q^{15},q^{15},q^{35}}{q^{75}}}
	\\&\quad
		+\frac{q^{14}}{\aqprod{q^{75}}{q^{75}}{\infty}\jacprod{q^{20}}{q^{75}}}
			\Parans{\SSeries{15}{-40}{75} - q^{40}\SSeries{35}{40}{75}}
		+
		q^4\frac{\aqprod{q^{75}}{q^{75}}{\infty}}{\jacprod{q^{15},q^{20}}{q^{75}}}
.
\end{align*}
\end{theorem}
\begin{remark}
The congruence $pp_2(5n+1)\equiv 0\pmod{5}$ then follows as
$PP_2(\zeta_5,q)$ has no $q^{5n+1}$ terms.
\end{remark}

\begin{theorem}
\begin{align*}
	PP_3(\zeta_5,q)
	&= \aqprod{q^{25}}{q^{25}}{\infty}
	\left(
		-q^5\frac{\jacprod{q^{10}}{q^{75}}}{\jacprod{q^5,q^{20},q^{25},q^{30}}{q^{75}}}
		-(\zeta_5+\zeta_5^4)q^6\frac{1}{\jacprod{q^{15},q^{25},q^{35}}{q^{75}}}
		\right.\\&\quad
		+(1+\zeta_5+\zeta_5^4)q^2\frac{1}{\jacprod{q^{10}}{q^{25}}} 
		-(\zeta_5+\zeta_5^4)q^2\frac{1}{\jacprod{q^5,q^{25},q^{30}}{q^{75}}}
		-q^8\frac{\jacprod{q^5}{q^{75}}}{\jacprod{q^{10},q^{15},q^{25},q^{35}}{q^{75}}}
		\\&\quad\left.
		+ (\zeta+\zeta^4)q^3\frac{\jacprod{q^{20}}{q^{75}}}{\jacprod{q^{10},q^{15},q^{25},q^{35}}{q^{75}}}
	\right)
.
\end{align*}
\end{theorem}
\begin{remark}
The congruence $pp_3(5n+4)\equiv 0\pmod{5}$ then follows as
$PP_3(\zeta_5,q)$ has no $q^{5n+4}$ terms.
\end{remark}

\begin{theorem}
\begin{align*}
	PP_4(\zeta_5,q)
	&=
	\aqprod{q^{25}}{q^{25}}{\infty}
	\left(
    	-(1+\zeta_5+\zeta_5^4)q^5\frac{\jacprod{q^{10}}{q^{75}}}
		{\jacprod{q^5,q^{20},q^{25},q^{30}}{q^{75}}}
		+
		\frac{q}{\jacprod{q^{10},q^{15},q^{25}}{q^{75}}}
		\right.\\&\left.\quad
		+
		(1+\zeta_5+\zeta_5^4)\frac{q^2}{\jacprod{q^{10}}{q^{25}}} 
    	- 
		\frac{q^2}{\jacprod{q^5,q^{25},q^{30}}{q^{75}}}
    - (\zeta+\zeta^4)q^8\frac{\jacprod{q^5}{q^{75}}}
			{\jacprod{q^{10},q^{15},q^{25},q^{35}}{q^{75}}}
	\right)
.
\end{align*}
\end{theorem}
\begin{remark}
The congruence $pp_4(5n+4)\equiv 0\pmod{5}$ then follows
as $PP_4(\zeta_5,q)$ has no $q^{5n+4}$ terms.
\end{remark}

\begin{theorem}
\begin{align*}
	PP_3(\zeta_7,q)
	&= \aqprod{q^{49}}{q^{49}}{\infty}
	\left(
		-(1+\zeta+\zeta^6)\frac{q^{14}}{\jacprod{q^{42},q^{49},q^{56}}{q^{147}}}	
		+q^2\frac{\jacprod{q^{14}}{q^{49}}}{\jacprod{q^7,q^{21}}{q^{49}}}
		\right.\\&\left.\quad
		-(1+\zeta^2+\zeta^5)\frac{q^9}{\jacprod{q^{21},q^{49},q^{70}}{q^{147}}}	
		+(\zeta_7+\zeta_7^6)\frac{q^3}{\jacprod{q^{14}}{q^{49}}}
		+(1+\zeta_7+\zeta_7^2+\zeta_7^5+\zeta_7^6)\frac{q^4}{\jacprod{q^{21}}{q^{49}}}
		\right.\\&\left.\quad
		+(\zeta+\zeta^6)q^5\frac{\jacprod{q^{35}}{q^{147}}}{\jacprod{q^{21},q^{28},q^{49},q^{49}}{q^{147}}}
		+(\zeta+\zeta^6)q^{19}\frac{\jacprod{q^{14}}{q^{147}}}{\jacprod{q^{21},q^{49},q^{49},q^{70}}{q^{147}}}			
		\right.\\&\left.\quad
		+(2+\zeta_7^2+\zeta_7^5)\frac{q^6}{\jacprod{q^{14},q^{49},q^{63}}{q^{147}}}	
	\right)
.
\end{align*}
\end{theorem}
\begin{remark}
The congruence $pp_3(7n+1)\equiv 0\pmod{7}$ then follows as 
$PP_3(\zeta_7,q)$ has no $q^{7n+1}$ terms.
\end{remark}

The rest of the paper is organized as follows. In section 2 we develop the
necessary identities to eventually express the series in Theorems 
\ref{TheoremPP13} through \ref{TheoremPP37}
in terms of products. In the sections
thereafter we use these formulas to reduce the proofs to verifying an identity
between products, which follows by checking that the equality holds in the first 
few terms of the $q$-expansions. In Section 9 we interpret the coefficients
$M_i(m,n)$ in terms of cranks defined on the partition pairs counted by $pp_i$.
We then end with a few remarks and questions.

\section{Preliminary Identities}

Although we are only concerned with the cases $\ell=3,5$ and $7$, all of the
formulas we state and prove in this section are valid for all odd $\ell>1$.
To begin we define
\begin{align*}
	U^a_\ell(b) 
	&= 
	\sum_{n=-\infty}^\infty \frac{q^{an^2+bn}}{1-q^{\ell(3n+1)}}
	,\\
	V^a_\ell(b) 
	&= 
	\sum_{n\not=0} \frac{q^{an^2+bn}}{1-q^{\ell(3n)}}
	.
\end{align*}

We express our series in terms of $U^a_\ell(b)$ and $V^a_\ell(b)$
by using the fact that
\begin{align*}
	1-x^\ell &= \prod_{j=0}^{\ell-1} (1-\zeta_\ell^jx)
.
\end{align*}

With $\zeta_3$ a primitive third root of unity we have
\begin{align}
\label{PP1Zeta3}
	&\sum_{n=-\infty}^\infty\frac{q^{6n^2+2n}}
		{(1-\zeta_3q^{3n})(1-\zeta_3^{-1}q^{3n})}
	-\sum_{n=-\infty}^\infty \frac{q^{6n^2+8n+2}}
		{(1-\zeta_3q^{3n+1})(1-\zeta_3^{-1}q^{3n+1})}
	\nonumber\\
	&=
		\frac{1}{3}+
		V^6_3(2) - V^6_3(5) - q^2U^6_3(8) + q^3U^6_3(11)
,
\end{align}
the $1/3$ is from the $n=0$ term.
We also have
\begin{align}
\label{PP2Zeta3}
	&\sum_{n=-\infty}^\infty\frac{q^{6n^2+n}}
		{(1-\zeta_3q^{3n})(1-\zeta_3^{-1}q^{3n})}
	-\sum_{n=-\infty}^\infty \frac{q^{6n^2+5n+1}}
		{(1-\zeta_3q^{3n+1})(1-\zeta_3^{-1}q^{3n+1})}
	\nonumber\\
	&=
		\frac{1}{3}+
		V^6_3(1) - V^6_3(4) - qU^6_3(5) + q^2U^6_3(8)
.
\end{align}
With $\zeta_5$ a primitive fifth root of unity we find
\begin{align}
\label{PP2Zeta5}
	&\sum_{n=-\infty}^\infty\frac{q^{6n^2+n}}
		{(1-\zeta_5q^{3n})(1-\zeta_5^{-1}q^{3n})}
	-\sum_{n=-\infty}^\infty \frac{q^{6n^2+5n+1}}
		{(1-\zeta_5q^{3n+1})(1-\zeta_5^{-1}q^{3n+1})}
	\nonumber\\
	&=
		\frac{1}{(1-\zeta_5)(1-\zeta_5^4)}
		+V^6_5(1) + (\zeta_5+\zeta_5^4)V^6_5(4) - (\zeta_5+\zeta_5^4)V^6_5(7) - V^6_5(10) 
		-qU^6_5(5) - (\zeta_5+\zeta_5^4)q^2U^6_5(8) 
		\nonumber\\&\quad
		+ (\zeta_5+\zeta_5^4)q^3U^6_5(11) + q^4U^6_5(14) 	
	,\\
	\label{PP3Zeta5}
	&	
		\sum_{n=-\infty}^\infty\frac{q^{3n^2+2n}}
			{(1-\zeta_5q^{3n})(1-\zeta_5^{-1}q^{3n})}
		-\sum_{n=-\infty}^\infty \frac{q^{3n^2+4n+1}}
			{(1-\zeta_5q^{3n+1})(1-\zeta_5^{-1}q^{3n+1})}
	\nonumber\\
	&=
		\frac{1}{(1-\zeta_5)(1-\zeta_5^4)}+
		V^3_5(2) + (\zeta_5+\zeta_5^4)V^3_5(5) - (\zeta_5+\zeta_5^4)V^3_5(8) - V^3_5(11) 
		-qU^3_5(4) - (\zeta_5+\zeta_5^4)q^2U^3_5(7) 
		\nonumber\\&\quad
		+ (\zeta_5+\zeta_5^4)q^3U^3_5(10) + q^4U^3_5(13) 
	,\\
	\label{PP4Zeta5}
	&
		\sum_{n=-\infty}^\infty\frac{q^{3n^2+n}}
			{(1-\zeta_5q^{3n})(1-\zeta_5^{-1}q^{3n})}
		-\sum_{n=-\infty}^\infty \frac{q^{3n^2+7n+2}}
			{(1-\zeta_5q^{3n+1})(1-\zeta_5^{-1}q^{3n+1})}
	\nonumber\\
	&=
		\frac{1}{(1-\zeta_5)(1-\zeta_5^4)}+
		V^3_5(1) + (\zeta_5+\zeta_5^4)V^3_5(4) - (\zeta_5+\zeta_5^4)V^3_5(7) - V^3_5(10) 
		-q^2U^3_5(7) - (\zeta_5+\zeta_5^4)q^3U^3_5(10) 
		\nonumber\\&\quad
		+ (\zeta_5+\zeta_5^4)q^4U^3_5(13) + q^5U^3_5(16) 
.
\end{align}
Lastly with $\zeta_7$ a primitive seventh root of unity, we find
\begin{align}
	\label{PP3Zeta7}
	&\sum_{n=-\infty}^\infty\frac{q^{3n^2+2n}}
		{(1-\zeta_7q^{3n})(1-\zeta_7^{-1}q^{3n})}
	-\sum_{n=-\infty}^\infty \frac{q^{3n^2+4n+1}}
		{(1-\zeta_7q^{3n+1})(1-\zeta_7^{-1}q^{3n+1})}
	\nonumber\\		
	&=
		\frac{1}{(1-\zeta_7)(1-\zeta_7^6)}
		+V^3_7(2) + (\zeta_7+\zeta_7^6)V^3_7(5) + (1+\zeta_7^2+\zeta_7^5)V^3_7(8)
		- (1+\zeta_7^2+\zeta_7^5)V^3_7(11) 	
		\nonumber\\&\quad
		- (\zeta_7+\zeta_7^6)V^3_7(14) 
		- V^3_7(17)
		-qU^3_7(4) - (\zeta_7+\zeta_7^6)q^2U^3_7(7) 
		- (1+\zeta_7^2+\zeta_7^5)q^3U^3_7(10)
		\nonumber\\&\quad
		+ (1+\zeta_7^2+\zeta_7^5)q^4U^3_7(13) 
	+ (\zeta_7+\zeta_7^6)q^5U^3_7(16) + q^6U^3_7(19)
.	
\end{align}

Rearranging the $S$ and $T$ series with $n\mapsto -n$ and $n\mapsto n+1$ we have
\begin{align}
	\label{SigmaProperty1}
	\SSeries{z}{w}{q} &= -z^{-1}\SSeries{z^{-1}}{w^{-1}q^{-3}}{q}
	,\\
	\label{SigmaProperty2}
	\SSeries{z}{w}{q} &= wq^4\SSeries{zq}{wq^4}{q}
	,\\
	\label{SigmaProperty3}
	\SSeries{z}{w}{q} &= -z^{-1}w^{-1}q\SSeries{z^{-1}q}{w^{-1}q}{q}
	,\\
	\label{SigmaStarProperty}
	S^*(w,q) &= -S^*(w^{-1}q^{-3},q)
	,\\
	\label{TProperty1}
	T(z,w,q) &= -z^{-1}T(z^{-1},w^{-1}q^{-1},q)
	,\\
	\label{TProperty2}
	T(z,w,q) &= wq^2T(zq,wq^2,q)
	,\\
	\label{TProperty3}
	T(z,w,q) &= -z^{-1}w^{-1}qT(z^{-1}q,w^{-1}q,q)
	,\\
	\label{TStarProperty}
	T^*(w,q) &= -T^*(w^{-1}q^{-1},q)	
.
\end{align}

We will also often rearrange infinite products without mention by
\begin{align*}
	\jacprod{z}{q}=-z\jacprod{zq}{q}=-z\jacprod{z^{-1}}{q}
.
\end{align*}

To prove the dissection formulas, we first find general formulas that express 
certain differences of $V^a_\ell(b)$ and $U^a_\ell(b)$ primarily in terms of products. 
For this we express $V^6_\ell(b)$ and $U^6_\ell(b)$ in terms of $S(z,w,q)$
and $V^3_\ell(b)$ and $U^3_\ell(b)$ in terms of $T(z,w,q)$. We then express
certain combinations of $S(z,w,q)$ in terms of products and similarly do so
for certain combinatations of $T(z,w,q)$.
The main tools for this
are two specializations of Theorem 2.1 of \cite{Chan}.

For $PP_1$ and $PP_2$ we use $r=0$ and $s=4$ to get
the following identity.
\begin{lemma}\label{ChanLemma1}
\begin{align*}
	&\frac{\aqprod{q}{q}{\infty}^2}{\jacprod{b_1,b_2,b_3,b_4}{q}}
	\\
	&=
		\frac{1}{\jacprod{b_2/b_1,b_3/b_1,b_4/b_1}{q}}
			\SSeries{b_1}{\frac{b_1^3}{b_2b_3b_4}}{q}
		+
		\frac{1}{\jacprod{b_1/b_2,b_3/b_2,b_4/b_2}{q}}
			\SSeries{b_2}{\frac{b_2^3}{b_1b_3b_4}}{q}
		\\&
		+
		\frac{1}{\jacprod{b_1/b_3,b_2/b_3,b_4/b_3}{q}}
			\SSeries{b_3}{\frac{b_3^3}{b_1b_2b_4}}{q}
		+
		\frac{1}{\jacprod{b_1/b_4,b_2/b_4,b_3/b_4}{q}}
			\SSeries{b_4}{\frac{b_4^3}{b_1b_2b_3}}{q}
.
\end{align*}
\end{lemma}

Next We define
\begin{align*}
	g(z,q) &=
	-z^2\SSeries{z}{z^2}{q}
	+\frac{\jacprod{z}{q}}{\jacprod{z^3}{q}}\SSeries{z}{\frac{z^6}{q^3}}{q}
	+z^5\frac{\jacprod{z}{q}}{\jacprod{z^3}{q}}\SSeries{z^2}{z^6}{q}
	+\sum_{n\not=0} \frac{q^{2n(n+1)}z^{-2n}}{1-q^n}
.
\end{align*}
The definition of $g(z,q)$ is motivated as follows. We would like to set one of
the $b_i$ in Lemma \ref{ChanLemma1} equal to $1$, as each $V_\ell^6(b)$ will
contribute a $S^*(w,q)$. For this we let 
$b_1=w$, $b_2=zw$, $b_3=z/w$, and $b_4=1/w$, multiply both sides by the product
$\jacprod{w^2z^{-1},w^2,z^{-1}}{q}$, subtract the $n=0$ term from 
$S(zw^{-1},z^2w^4,q)$, and let $w\rightarrow z$.
In particular, we also have
\begin{align*}
	g(z,q) &=
	\lim_{w\rightarrow z}
	\Parans{
		\frac{\jacprod{w^2z^{-1},z^{-1},w^2}{q}\aqprod{q}{q}{\infty}^2}
		{\jacprod{w,w^{-1},zw,zw^{-1}}{q}}
		-\frac{1}{1-z/w}
	}
.
\end{align*}

For $PP_3$ and $PP_4$ we use $r=0$ and $s=2$ and simplify to get
the following identity.
\begin{lemma}\label{ChanLemma2}
\begin{align*}
	\frac{\aqprod{q}{q}{\infty}^2\jacprod{b_2b_1^{-1}}{q}}
		{\jacprod{b_1,b_2}{q}}
	&=
	T(b_1,b_1b_2^{-1},q) - b_2b_1^{-1}	T(b_2,b_2b_1^{-1},q)
.
\end{align*}
\end{lemma}

We define
\begin{align*}
	h(z,q) 
	&=
	T^*(z^{-1},q) - zT(z,z,q)
.
\end{align*}
We note that this function arises from the the right hand side of 
the equation in Lemma \ref{ChanLemma2} by
removing the $n=0$ term of $T(b_1,b_1b_2^{-1},q)$, setting $b_2=z$, and letting 
$b_1\rightarrow 1$.

\begin{lemma}\label{LemmaFor3b2}
For any integer $b$,
\begin{align*}
	&V^3_\ell(3b+2)-q^{b+1}U^3_\ell(3b+4)
	\\
	&=
	h(q^{3\ell^2-3b\ell-2\ell}, q^{3\ell^2})
	+\sum_{k=1}^{\ell-1}q^{3k^2+2k+3bk}
		\frac{\aqprod{q^{3\ell^2}}{q^{3\ell^2}}{\infty}^2 
			\jacprod{q^{3\ell^2-3b\ell-6k\ell-2\ell}}{q^{3\ell^2}} }
		{\jacprod{q^{3\ell k},q^{3\ell^2-3b\ell-3k\ell-2\ell}}{q^{3\ell^2}} }
.
\end{align*}
\end{lemma}
\begin{proof}
We have
\begin{align*}
	&V^3_\ell(3b+2)-q^{b+1}U^3_\ell(3b+4)
	\\
	&=
	\sum_{n\not=0} \frac{q^{3n^2+(3b+2)n}}{1-q^{3\ell n}}
	-q^{b+1}\sum_{n=-\infty}^\infty \frac{q^{3n^2+(3b+4)n}}{1-q^{3\ell n+\ell}}
	\\
	&=
	\sum_{k=0}^{\ell-1}q^{3k^2+2k+3bk}\sum_{\substack{n=-\infty\\(n,k)\not=(0,0)}}^\infty 
		\frac{q^{3\ell^2n(n+1) - 3\ell^2 n + 3b\ell n + 6k\ell n + 2\ell n}}
		{1-q^{3\ell^2n + 3\ell k}}
		\\&\quad
		-\sum_{k=0}^{\ell-1}
			q^{3k^2 + 2k + 3bk - 6k\ell  - 3b\ell  - 2\ell    + 3\ell^2 }
			\sum_{n=-\infty}^\infty\frac{ q^{3\ell^2n^2 + 6\ell^2n - 3b\ell n - 6k\ell n - 2\ell n} }
				{1-q^{3\ell^2n -3b\ell  + 3\ell^2 - 2\ell  - 3k\ell }}			
	\\
	&=
	T^*(3b\ell+2\ell-3\ell^2,3\ell^2)
	-
	q^{- 3\ell b - 2\ell   + 3\ell^2}
		T(3\ell^2-3b\ell-2\ell, 3\ell^2 - 3b\ell - 2\ell, 3\ell^2)
		\\&\quad
		+
		\sum_{k=1}^{\ell-1}q^{3k^2+2k+3bk}
		T(3\ell k, 3b\ell + 6k\ell + 2\ell - 3\ell^2, 3\ell^2)
		\\&\quad
		-
		\sum_{k=1}^{\ell-1}
			q^{3k^2 + 2k + 3bk - 6k\ell  - 3b\ell  - 2\ell   + 3\ell^2 }
			T(3\ell^2-3b\ell-3k\ell-2\ell, 3\ell^2-3b\ell-6k\ell-2\ell, 3\ell^2)
	\\
	&=
		h(q^{3\ell^2-3b\ell-2\ell}, q^{3\ell^2})
		+\sum_{k=1}^{\ell-1}q^{3k^2+2k+3bk}
			\frac{\aqprod{q^{3\ell^2}}{q^{3\ell^2}}{\infty}^2 
				\jacprod{q^{3\ell^2 - 3b\ell - 6k\ell - 2\ell}}{q^{3\ell^2}} }
			{\jacprod{q^{3\ell k},q^{3\ell^2 - 3b\ell - 3k\ell -2\ell}}{q^{3\ell^2}} }
.	
\end{align*}
In $V^3_\ell(3b+2)$ we have replaced $n$ by $\ell n + k$ and in
$U^3_\ell(3b+4)$ we have replaced $n$ by $\ell n - b - k + \ell - 1 $.
The last equality follows by applying Lemma \ref{ChanLemma2} to the 
$k^{th}$ terms of each sum.
\end{proof}

\begin{lemma}\label{LemmaFor6b1}
For any integer $b$,
\begin{align*}
	& V_\ell^6(6b+1)-q^{2b+1}U_\ell^6(6b+5)
	\\
	&=
	-g(q^{\frac{3\ell^2+\ell}{2}+3b\ell},q^{3\ell^2})
	-\sum_{k=2}^{\ell-1} q^{6bk+6k^2+k-3k\ell}
		\frac{\aqprod{q^{3\ell^2}}{q^{3\ell^2}}{\infty}^2
			\jacprod{q^{3k\ell-3\ell}, q^{\frac{3\ell^2+7\ell}{2}+3b\ell+3k\ell},
			q^{\frac{3\ell^2+\ell}{2}+6k\ell+3b\ell}}{q^{3\ell^2}}}	
		{\jacprod{q^{-3\ell},q^{\frac{3\ell^2+7\ell}{2}+3b\ell},q^{-3k\ell},
			q^{\frac{3\ell^2+\ell}{2} + 3k\ell + 3b\ell}}{q^{3\ell^2}}}		
	\\&\quad
	+q^{\frac{3\ell^2+\ell}{2}+3b\ell }
	\frac{\aqprod{q^{3\ell^2}}{q^{3\ell^2}}{\infty}^2
		\jacprod{q^{-6b\ell-4\ell}, q^{\frac{-3\ell^2+5\ell}{2}-3b\ell}}{q^{3\ell^2}}}
	{\jacprod{q^{3\ell}, q^{\frac{3\ell^2-7\ell}{2}-3b\ell}, q^{-6b\ell-\ell}}{q^{3\ell^2}}}
	\\&\quad
	+
		(-1)^{b+1}q^{-\frac{3b^2+b}{2}-3\ell}\frac{\aqprod{q}{q}{\infty}}
		{\aqprod{q^{3\ell^2}}{q^{3\ell^2}}{\infty}
		\jacprod{q^{\frac{3\ell^2+13\ell}{2}+3b\ell }}{q^{3\ell^2}}}
		\\&\qquad\qquad
		\times\left(
		\SSeries{-3\ell}{-3\ell^2 - 6b\ell - 13\ell}{3\ell^2}
		-q^{3\ell^2 + 6b\ell + 13\ell}\SSeries{ \frac{3\ell^2+7\ell}{2} + 3b\ell}{3\ell^2 + 6b\ell + 13\ell}{3\ell^2}
	\right)
.
\end{align*}
\end{lemma}
\begin{proof}
The proof is similar to Lemma \ref{LemmaFor3b2}, but a little more involved.
In $V_\ell^6(6b+1)$ we use $n\mapsto \ell n+k$ and rearrange the terms
by (\ref{SigmaProperty1}) and (\ref{SigmaStarProperty});
in $U_\ell^6(6b+5)$ we use $n\mapsto\ell  n-k -b+\frac{\ell-1}{2}$
and rearrange the terms by (\ref{SigmaProperty3}).
We isolate the $k=0$ and $k=1$ summands and apply Lemma \ref{ChanLemma1}
so that we only have products and a term involving the $k=1$ summand.
\begin{align*}
	&V_\ell^6(6b+1) -q^{2b+1}U_\ell^6(6b+5)
	\\
	&=
	\sum_{n\not=0} \frac{q^{6n^2+(6b+1)n}}{1-q^{3\ell n}}
	-q^{2b+1}\sum_{n=-\infty}^{\infty} \frac{q^{6n^2+(6b+5)n}}{1-q^{3\ell n+\ell}}
	\\
	&=
		\sum_{k=0}^{\ell-1}q^{6k^2+6bk+k}
		\sum_{\substack{n=-\infty\\(n,k)\not=(0,0)}}^\infty
		\frac{q^{6\ell^2n^2 + 6b\ell n + 12k\ell n +\ell n}}
			{1-q^{3\ell^2n + 3k\ell }}
		\\&\quad
		-
		\sum_{k=0}^{\ell-1}q^{\frac{3\ell^2-\ell}{2} + k + 6bk +6k^2 -6k\ell -3b\ell}
		\sum_{n=-\infty}^\infty
		\frac{q^{6\ell^2n^2 + 6\ell^2n -6b\ell n - 12k\ell n - \ell n }}
			{1-q^{3\ell^2n - 3k\ell  - 3b\ell + \frac{3\ell^2-\ell}{2}}}
	\\
	&=
		S^*(-6\ell^2+6b\ell  + \ell, 3\ell^2)
		+
		\sum_{k=1}^{\ell-1} q^{6bk+6k^2+k}\SSeries{3k\ell}{-6\ell^2 + 6b\ell  +12k\ell  + \ell}{3\ell^2}
		\\&\quad
		-\sum_{k=0}^{\ell-1} q^{\frac{3\ell^2-\ell}{2} + k + 6bk +6k^2 -6k\ell -3b\ell}
		\SSeries{-3k\ell-3b\ell+ \frac{3\ell^2-\ell}{2} }{-6b\ell-12k\ell-\ell}{3\ell^2}
	\\
	&=
		-S^*(-3\ell^2 - 6\ell b - \ell,3\ell^2)
		-\sum_{k=1}^{\ell-1} q^{6bk+6k^2+k-3k\ell}\SSeries{-3k\ell}{-3\ell^2 - 6b\ell  -12k\ell  - \ell}{3\ell^2}
		\\&\quad
		+\sum_{k=0}^{\ell-1}  q^{3\ell^2 + 6b\ell + 9k\ell + 6bk + 6k^2 + \ell +k}
		\SSeries{\frac{3\ell^2+\ell}{2}+3k\ell+3b\ell}{3\ell^2+6b\ell+12k\ell+\ell}{3\ell^2}
	\\
	&=
		-S^*(-3\ell^2-6b\ell-\ell,3\ell^2)	
		+q^{3\ell^2+6b\ell+\ell}\SSeries{ \frac{3\ell^2+\ell}{2}+3b\ell }{3\ell^2+6b\ell+\ell}{3\ell^2}	
		\\&\quad
		-q^{6b+7-3\ell}
		\left(
			\SSeries{-3\ell}{-3\ell^2-6b\ell-13\ell}{3\ell^2}
			-q^{3\ell^2+6b\ell+13\ell}\SSeries{ \frac{3\ell^2+7\ell}{2}+3b\ell}{3\ell^2+6b\ell+13\ell}{3\ell^2}
		\right)
		\\&\quad
		-\sum_{k=2}^{\ell-1} q^{6bk+6k^2+k-3k\ell}\SSeries{-3k\ell}{-3\ell^2 - 6b\ell  -12k\ell  - \ell}{3\ell^2}
		\\&\quad
		+\sum_{k=2}^{\ell-1} q^{3\ell^2 + 6b\ell + 9k\ell + 6bk + 6k^2 + \ell +k}
		\SSeries{\frac{3\ell^2+\ell}{2}+3k\ell+3b\ell}{3\ell^2+6b\ell+12k\ell+\ell}{3\ell^2}
	\\
	&=
		-S^*(-3\ell^2-6b\ell-\ell,3\ell^2)	
		+q^{3\ell^2+6b\ell+\ell}\SSeries{ \frac{3\ell^2+\ell}{2}+3b\ell }{3\ell^2+6b\ell+\ell}{3\ell^2}	
		\\&\quad
		-q^{6b+7-3\ell}
		\left(
			\SSeries{-3\ell}{-3\ell^2-6b\ell-13\ell}{3\ell^2}
			-q^{3\ell^2+6b\ell+13\ell}\SSeries{ \frac{3\ell^2+7\ell}{2}+3b\ell}{3\ell^2+6b\ell+13\ell}{3\ell^2}
		\right)
		\\&\quad
		+\sum_{k=2}^{\ell-1} q^{6bk+6k^2+k-3k\ell}
		\bigg(
			-\SSeries{-3k\ell}{-3\ell^2 - 6b\ell  -12k\ell  - \ell}{3\ell^2}
			\\&\left.\qquad\qquad
			+
			q^{3\ell^2 + 6b\ell + 12k\ell + \ell }
			\SSeries{\frac{3\ell^2+\ell}{2}+3k\ell+3b\ell}{3\ell^2+6b\ell+12k\ell+\ell}{3\ell^2}
		\right)
.
\end{align*}

Applying Lemma \ref{ChanLemma1} with
$q\mapsto q^{3\ell^2}$,
$b_1 = q^{-3\ell}$,
$b_2 = q^{\frac{3\ell^2+7\ell}{2}+3b\ell}$,
$b_3 = q^{-3k\ell}$,
$b_4 = q^{\frac{3\ell^2+\ell}{2} + 3k\ell + 3b\ell}$
and simplifying gives
\begin{align*}
	&
		-\SSeries{-3k\ell}{-3\ell^2 - 6\ell b -12\ell k - \ell}{3\ell^2}
		+q^{3\ell^2+6b\ell+12k\ell+\ell}
		\SSeries{\frac{3\ell^2+\ell}{2}+3k\ell+3b\ell}{3\ell^2+6b\ell+12k\ell+\ell}{3\ell^2}
	\\
	&=
		-\frac{\aqprod{q^{3\ell^2}}{q^{3\ell^2}}{\infty}^2
			\jacprod{
			q^{3k\ell-3\ell},
			q^{\frac{3\ell^2+7\ell}{2}+3b\ell+3k\ell},
			q^{\frac{3\ell^2+\ell}{2}+6k\ell+3b\ell}
			}{q^{3\ell^2}}}	
		{\jacprod{
			q^{-3\ell},
			q^{\frac{3\ell^2+7\ell}{2}+3b\ell},
			q^{-3k\ell},
			q^{\frac{3\ell^2+\ell}{2} + 3k\ell + 3b\ell}
			}{q^{3\ell^2}}}
		\\&\quad		
		-q^{3k\ell-3\ell}\frac{\jacprod{
			q^{\frac{3\ell^2+\ell}{2}+6k\ell+3b\ell}
			}{q^{3\ell^2}}}	
		{\jacprod{
			q^{\frac{3\ell^2+13\ell}{2}+3b\ell}
		}{q^{3\ell^2}}}		
		\SSeries{-3\ell}{-3\ell^2-6b\ell-13\ell}{3\ell^2}
		\\&\quad
		+q^{3k\ell+3\ell^2+6b\ell+10\ell}
		\frac{\jacprod{
			q^{\frac{3\ell^2+\ell}{2}+6k\ell+3b\ell}
			}{q^{3\ell^2}}}	
		{\jacprod{
			q^{\frac{3\ell^2+13\ell}{2}+3b\ell}
		}{q^{3\ell^2}}}		
		\SSeries{ \frac{3\ell^2+7\ell}{2}+3b\ell}{3\ell^2+6b\ell+13\ell}{3\ell^2}
	\\
	&=
		-\frac{\aqprod{q^{3\ell^2}}{q^{3\ell^2}}{\infty}^2
			\jacprod{q^{3k\ell-3\ell}, q^{\frac{3\ell^2+7\ell}{2}+3b\ell+3k\ell},
			q^{\frac{3\ell^2+\ell}{2}+6k\ell+3b\ell}
			}{q^{3\ell^2}}}	
		{\jacprod{q^{-3\ell}, q^{\frac{3\ell^2+7\ell}{2}+3b\ell}, q^{-3k\ell},
			q^{\frac{3\ell^2+\ell}{2} + 3k\ell + 3b\ell}}{q^{3\ell^2}}}
		\\&\quad		
		-q^{3k\ell-3\ell}\frac{\jacprod{q^{\frac{3\ell^2+\ell}{2}+6k\ell+3b\ell}}{q^{3\ell^2}}}	
		{\jacprod{q^{\frac{3\ell^2+13\ell}{2}+3b\ell}}{q^{3\ell^2}}}
		\\&\qquad\times
		\left(
			\SSeries{-3\ell}{-3\ell^2-6b\ell-13\ell}{3\ell^2}
			-
			q^{3\ell^2+6b\ell+13\ell}
			\SSeries{ \frac{3\ell^2+7\ell}{2}+3b\ell}{3\ell^2+6b\ell+13\ell}{3\ell^2}
		\right)
.
\end{align*}

We would like to add in the terms to see 
$g(q^{\frac{3\ell^2+\ell}{2}+3b\ell},q^{3\ell^2})$
and apply Lemma \ref{ChanLemma1} again, however these terms are not well 
defined when $6b\equiv -1 \pmod{\ell}$. Instead we first apply Lemma 
\ref{ChanLemma1} with
$q\mapsto q^{3\ell^2}$,
$b_1 = q^{3\ell}$,
$b_2 = z^{-1}q^{3\ell^2-3\ell}$,
$b_3 = zq^{3\ell^2}$,
$b_4 = z^{-2}q^{3\ell^2}$,
alter the series with (\ref{SigmaProperty1}), (\ref{SigmaProperty2}),  and
(\ref{SigmaProperty3}), 
and simplify to obtain
\begin{align*}
	&
	-\frac{\jacprod{z}{q^{3\ell^2}}}{\jacprod{z^3}{q^{3\ell^2}}}
		\SSeries{z}{z^6q^{-9\ell^2}}{q^{3\ell^2}}
	-z^5\frac{\jacprod{z}{q^{3\ell^2}}}{\jacprod{z^3}{q^{3\ell^2}}}
		\SSeries{z^2}{z^6}{q^{3\ell^2}}
	\\
	&=	
	-z\frac{\aqprod{q^{3\ell^2}}{q^{3\ell^2}}{\infty}^2
			\jacprod{z^{-1}q^{3\ell}, z^{-2}q^{3\ell^2-3\ell}}{q^{3\ell^2}}}
		{\jacprod{q^{3\ell}, z^{-1}q^{3\ell^2-3\ell},z^{-2}q^{3\ell^2}}{q^{3\ell^2}}}
		\\&\quad
	+q^{-3\ell}\frac{\jacprod{z}{q^{3\ell^2}}}{\jacprod{z^{-1}q^{3\ell^2-6\ell}}{q^{3\ell^2}}}
		\SSeries{q^{-3\ell}}{z^{-2}q^{-12\ell}}{q^{3\ell^2}}
	-z^2q^{9\ell}\frac{\jacprod{z}{q^{3\ell^2}}}{\jacprod{z^{-1}q^{3\ell^2-6\ell}}{q^{3\ell^2}}}
		\SSeries{zq^{3\ell}}{z^2q^{12\ell}}{q^{3\ell^2}}
.
\end{align*}
Thus by the definition of $g(z,q^{3\ell^2})$ we have
\begin{align*}
	&	
	-S^*(z^{-2},q^{3\ell^2})+z^2\SSeries{z}{z^2}{q^{3\ell^2}}
	\\
	&= -g(z,q^{3\ell^2})
	+z\frac{\aqprod{q^{3\ell^2}}{q^{3\ell^2}}{\infty}^2
		\jacprod{z^{-1}q^{3\ell},z^{-2}q^{3\ell^2-3\ell}}{q^{3\ell^2}}
		}{\jacprod{q^{3\ell},z^{-1}q^{3\ell^2-3\ell},z^{-2}q^{3\ell^2}}{q^{3\ell^2}}}	
		\\&\quad
		-q^{-3\ell}\frac{\jacprod{z}{q^{3\ell^2}}}
		{\jacprod{z^{-1}q^{3\ell^2-6\ell}}{q^{3\ell^2}}}
		\Parans{\SSeries{q^{-3\ell}}{z^{-2}q^{-12\ell}}{q^{3\ell^2}}
			-z^2q^{12\ell}\SSeries{zq^{3\ell}}{z^2q^{12\ell}}{q^{3\ell^2}}	
		}
\end{align*}
and here it is now valid to set $z = q^{\frac{3\ell^2+\ell}{2}+3b\ell}$.

Thus
\begin{align*}
	&V_\ell^6(6b+1)-q^{2b+1}U_\ell^6(6b+5)
	\\
	&=
	-g(q^{\frac{3\ell^2+\ell}{2}+3b\ell},q^{3\ell^2})
	-\sum_{k=2}^{\ell-1} q^{6bk+6k^2+k-3k\ell}
		\frac{\aqprod{q^{3\ell^2}}{q^{3\ell^2}}{\infty}^2
			\jacprod{
			q^{3k\ell-3\ell},
			q^{\frac{3\ell^2+7\ell}{2}+3b\ell+3k\ell},
			q^{\frac{3\ell^2+\ell}{2}+6k\ell+3b\ell}
			}{q^{3\ell^2}}}	
		{\jacprod{
			q^{-3\ell},
			q^{\frac{3\ell^2+7\ell}{2}+3b\ell},
			q^{-3k\ell},
			q^{\frac{3\ell^2+\ell}{2} + 3k\ell + 3b\ell}
			}{q^{3\ell^2}}}		
	\\&\quad
	+q^{\frac{3\ell^2+\ell}{2}+3b\ell }
	\frac{\aqprod{q^{3\ell^2}}{q^{3\ell^2}}{\infty}^2
		\jacprod{
			q^{-6b\ell-4\ell}, 
			q^{\frac{-3\ell^2+5\ell}{2}-3b\ell}
		}{q^{3\ell^2}}}
	{\jacprod{
		q^{3\ell}, 
		q^{\frac{3\ell^2-7\ell}{2}-3b\ell}, 
		q^{-6b\ell-\ell}
		}{q^{3\ell^2}}
	}
	\\&\quad
	+X\cdot\left(
		\SSeries{-3\ell}{-3\ell^2-6b\ell-13\ell}{3\ell^2}
		-q^{3\ell^2+6b\ell+13\ell}\SSeries{ \frac{3\ell^2+7\ell}{2}+3b\ell}{3\ell^2+6b\ell+13\ell}{3\ell^2}
	\right)
\end{align*}
where
\begin{align*}
	X &=
			-q^{6b+7-3\ell}
			-q^{-3\ell}
			\frac{\jacprod{
				q^{\frac{3\ell^2+\ell}{2}+3b\ell}
			}{q^{3\ell^2}}}
			{\jacprod{
				q^{\frac{3\ell^2-13\ell}{2}-3b\ell} 
			}{q^{3\ell^2}}}
		-\sum_{k=2}^{\ell-1} q^{6bk+6k^2+k-3\ell}
		\frac{\jacprod{
			q^{\frac{3\ell^2+\ell}{2}+6k\ell+3b\ell}
			}{q^{3\ell^2}}}	
		{\jacprod{
			q^{\frac{3\ell^2+13\ell}{2}+3b\ell}
		}{q^{3\ell^2}}}		
.
\end{align*}

We note in fact
\begin{align*}
	X
	&=
	-\sum_{k=0}^{\ell-1} q^{6bk+6k^2+k-3\ell}
	\frac{\jacprod{q^{\frac{3\ell^2+\ell}{2}+6k\ell+3b\ell}}{q^{3\ell^2}}}	
		{\jacprod{q^{\frac{3\ell^2+13\ell}{2}+3b\ell}}{q^{3\ell^2}}}		
.
\end{align*}
But by Euler's Pentagonal Numbers Theorem and the Jacobi Triple Product
identity we have
\begin{align*}
	\aqprod{q}{q}{\infty}
	&=
	\sum_{n=-\infty}^\infty	(-1)^n q^{\frac{n(3n+1)}{2}}
	\\
	&=
	\sum_{k=0}^{\ell-1}	
	\sum_{n=-\infty}^\infty	(-1)^{\ell n+2k+b} q^{\frac{ (\ell n+2k+b)(3(\ell n+2k+b)+1)}{2}}
	\\
	&=
	(-1)^b q^{\frac{3b^2+b}{2}}
	\sum_{k=0}^{\ell-1}q^{6bk+6k^2+k}
	\sum_{n=-\infty}^\infty (-1)^nq^{3b\ell n+6k\ell n + \frac{\ell n}{2}}q^{\frac{3\ell^2n^2}{2}}
	\\
	&=
	(-1)^b q^{\frac{3b^2+b}{2}}
	\aqprod{q^{3\ell^2}}{q^{3\ell^2}}{\infty}
	\sum_{k=0}^{\ell-1}q^{6bk+6k^2+k}\jacprod{q^{\frac{3\ell^2+\ell}{2}+3b\ell+6k\ell} } {q^{3\ell^2}}
.
\end{align*}
Thus
\begin{align*}
	X &= (-1)^{b+1}q^{-\frac{3b^2+b}{2}-3\ell}\frac{\aqprod{q}{q}{\infty}}
		{\aqprod{q^{3\ell^2}}{q^{3\ell^2}}{\infty}
		\jacprod{q^{\frac{3\ell^2+13\ell}{2}+3b\ell }}{q^{3\ell^2}}}
.
\end{align*}
\end{proof}

\begin{lemma}\label{LemmaFor6b4}
For any integer $b$,
\begin{align*}
	&
	V_\ell^6(6b+4)-q^{2b+2}U_\ell^6(6b+8)
	\\
	&=	
		g(3\ell^2-3b\ell-2\ell,q^{3\ell^2})
		+
		\sum_{k=2}^{\ell-1} q^{6k^2+4k+6bk}
			\frac{\aqprod{q^{3\ell^2}}{q^{3\ell^2}}{\infty}^2
			\jacprod{q^{-3k\ell+3\ell}, q^{3\ell^2-3b\ell-3k\ell-5\ell}, q^{3\ell^2-3b\ell-6k\ell-2\ell}
				}{q^{3\ell^2}}
			}{\jacprod{q^{3\ell}, q^{3\ell^2-3b\ell-5\ell}, q^{3k\ell}, q^{3\ell^2-3k\ell-3b\ell-2\ell}
			}{q^{3\ell^2}}}
		\\&\quad
		-
		\frac{\aqprod{q^{3\ell^2}}{q^{3\ell^2}}{\infty}^2
			\jacprod{q^{3\ell^2-3b\ell+\ell}, q^{6\ell^2-6b\ell-7\ell}}{q^{3\ell^2}}
		}{\jacprod{q^{3\ell},q^{3\ell^2 -3b\ell-5\ell},q^{6\ell^2 -6b\ell-4\ell}}{q^{3\ell^2}}}
		\\&\quad
		+(-1)^{\frac{\ell-1}{2}+b}
		q^{4\ell+\frac{3b\ell-3b^2}{2}-2b-\frac{3\ell^2+5}{8}}
		\frac{\aqprod{q}{q}{\infty}}
		{\aqprod{q^{3\ell^2}}{q^{3\ell^2}}{\infty}\jacprod{q^{3b\ell+8\ell}}{q^{3\ell^2}} }
		\\&\qquad\qquad
		\times\left(
			\SSeries{3\ell}{6b\ell+16\ell-6\ell^2}{3\ell^2}
			-q^{6\ell^2-6b\ell-16\ell}\SSeries{3\ell^2-3b\ell-5\ell}{6\ell^2-6b\ell-16\ell}{3\ell^2}	
		\right)
.
\end{align*}
\end{lemma}
\begin{proof}
For $V_\ell^6(6b+4)$ we use $n\mapsto \ell n+k$ and in
$U_\ell^6(6b+8)$ we use $n\mapsto \ell n-k+\ell-b-1$.
\begin{align*}
	&V_\ell^6(6b+4)-q^{2b+2}U_\ell^6(6b+8)
	\\
	&=
	\sum_{n\not=0} \frac{q^{6n^2+(6b+4)n}}{1-q^{3n\ell}}
	-q^{2b+2}
	\sum_{n=-\infty}^\infty \frac{q^{6n^2+(6b+8)n}}{1-q^{3n\ell+\ell}}
	\\
	&=
		\sum_{k=0}^{\ell-1} q^{6k^2+6bk+4k}
		\sum_{\substack{n=-\infty\\(n,k)\not=(0,0)}}^\infty
		\frac{q^{6\ell^2n^2 + 6b\ell n + 12k\ell n + 4\ell n}}
			{1-q^{3\ell^2n + 3k\ell }}
		\\&\quad
		-
		\sum_{k=0}^{\ell-1} q^{6\ell^2-6b\ell-12k\ell+6bk+6k^2-4\ell+4k}
		\sum_{n=-\infty}^\infty
		\frac{q^{6\ell^2n^2  + 12\ell^2n - 6b\ell n - 12k\ell n - 4\ell n}}
			{1-q^{3\ell^2n + 3\ell^2 - 3b\ell  -  3k\ell  - 2\ell}}
	\\
	&=
	S^*(6b\ell+4\ell-6\ell^2,3\ell^2)
	+
	\sum_{k=1}^{\ell-1} q^{6k^2+4k+6bk}
		\SSeries{3k\ell}{6b\ell+12k\ell+4\ell-6\ell^2}{3\ell^2}
		\\&\quad
		-
		\sum_{k=0}^{\ell-1} q^{6\ell^2-6b\ell-12k\ell+6bk+6k^2-4\ell+4k}
			\SSeries{-3k\ell+3\ell^2-3b\ell-2\ell}{6\ell^2-6b\ell-12k\ell-4\ell}{3\ell^2}
	\\
	&=
		S^*(6b\ell+4\ell-6\ell^2,  3\ell^2)
		-q^{6\ell^2-6b\ell-4\ell}\SSeries{3\ell^2-3b\ell-2\ell}{6\ell^2-6b\ell-4\ell}{3\ell^2}
		\\&\quad
		+q^{10+6b}\SSeries{3\ell}{6b\ell+16\ell-6\ell^2}{3\ell^2}
		-q^{10+6b+6\ell^2-6b\ell-16\ell}\SSeries{3\ell^2-3b\ell-5\ell}{6\ell^2-6b\ell-16\ell}{3\ell^2}
		\\&\quad
		+\sum_{k=2}^{\ell-1} q^{6k^2+4k+6bk}
		\SSeries{3k\ell}{6b\ell+12k\ell+4\ell-6\ell^2}{3\ell^2}
		\\&\quad
		-
		\sum_{k=2}^{\ell-1} q^{6\ell^2-6b\ell-12k\ell+6bk+6k^2-4\ell+4k}
		\SSeries{-3k\ell+3\ell^2-3b\ell-2\ell}{6\ell^2-6b\ell-12k\ell-4\ell}{3\ell^2}
	\\
	&=
		S^*(6b\ell+4\ell-6\ell^2,  3\ell^2)
		-q^{6\ell^2-6b\ell-4\ell}\SSeries{3\ell^2-3b\ell-2\ell}{6\ell^2-6b\ell-4\ell}{3\ell^2}
		\\&\quad
		+q^{10+6b}\SSeries{3\ell}{6b\ell+16\ell-6\ell^2}{3\ell^2}
		-q^{10+6b+6\ell^2-6b\ell-16\ell}\SSeries{3\ell^2-3b\ell-5\ell}{6\ell^2-6b\ell-16\ell}{3\ell^2}
		\\&\quad
		+\sum_{k=2}^{\ell-1} q^{6k^2+4k+6bk}
		\left(
			\SSeries{3k\ell}{6b\ell+12k\ell+4\ell-6\ell^2}{3\ell^2}
			\right.\\&\left.\qquad\qquad
			-
			q^{6\ell^2-6b\ell-12k\ell-4\ell}
			\SSeries{-3k\ell+3\ell^2-3b\ell-2\ell}{6\ell^2-6b\ell-12k\ell-4\ell}{3\ell^2}
		\right)
.
\end{align*}

We apply Lemma \ref{ChanLemma2} with
$q\mapsto q^{3\ell^2}$,
$b_1 = q^{3\ell}$,
$b_2 = q^{3\ell^2-3b\ell-5\ell}$,
$b_3 = q^{3k\ell}$,
$b_4 = q^{3\ell^2-3k\ell-3b\ell-2\ell}$
and simplify to find that
\begin{align*}
	&
		\SSeries{3k\ell}{6b\ell+12k\ell+4\ell-6\ell^2}{3\ell^2}
		-q^{6\ell^2-6b\ell-12k\ell-4\ell}
		\SSeries{3\ell^2-3k\ell-3b\ell-2\ell}{6\ell^2-6b\ell-12k\ell-4\ell}{3\ell^2}
	\\
	&=
		\frac{\aqprod{q^{3\ell^2}}{q^{3\ell^2}}{\infty}^2
			\jacprod{
				q^{-3k\ell+3\ell}, 
				q^{3\ell^2-3b\ell-3k\ell-5\ell}, 
				q^{3\ell^2-3b\ell-6k\ell-2\ell}
			}{q^{3\ell^2}}
		}{\jacprod{
			q^{3\ell},
			q^{3\ell^2-3b\ell-5\ell},
			q^{3k\ell},
			q^{3\ell^2-3k\ell-3b\ell-2\ell}
		}{q^{3\ell^2}}}
		\\&\quad		
		+q^{3\ell-3k\ell}\frac{\jacprod{
			q^{3\ell^2-3b\ell-6k\ell-2\ell}
		}{q^{3\ell^2}}
		}{\jacprod{
			q^{3\ell^2-3b\ell-8\ell}
		}{q^{3\ell^2}}}
		\SSeries{3\ell}{6b\ell+16\ell-6\ell^2}{3\ell^2}
		\\&\quad
		-q^{6\ell^2-6b\ell-3k\ell-13\ell}
		\frac{\jacprod{
			q^{3\ell^2-3b\ell-6k\ell-2\ell}
		}{q^{3\ell^2}}
		}{\jacprod{
			q^{3\ell^2-3b\ell-8\ell}
		}{q^{3\ell^2}}}
		\SSeries{3\ell^2-3b\ell-5\ell}{6\ell^2-6b\ell-16\ell}{3\ell^2}
	\\
	&=
		\frac{\aqprod{q^{3\ell^2}}{q^{3\ell^2}}{\infty}^2
			\jacprod{q^{-3k\ell+3\ell}, q^{3\ell^2-3b\ell-3k\ell-5\ell},q^{3\ell^2-3b\ell-6k\ell-2\ell}
			}{q^{3\ell^2}}
		}{\jacprod{q^{3\ell},q^{3\ell^2-3b\ell-5\ell},q^{3k\ell},
			q^{3\ell^2-3k\ell-3b\ell-2\ell}
		}{q^{3\ell^2}}}
		\\&\quad		
		+q^{3\ell-3k\ell}
		\frac{\jacprod{q^{3\ell^2-3b\ell-6k\ell-2\ell}}{q^{3\ell^2}}}
			{\jacprod{q^{3\ell^2-3b\ell-8\ell}}{q^{3\ell^2}}}
		\\&\qquad\times
		\left(
			\SSeries{3\ell}{6b\ell+16\ell-6\ell^2}{3\ell^2}
			-
			q^{6\ell^2-6b\ell-16\ell}
			\SSeries{3\ell^2-3b\ell-5\ell}{6\ell^2-6b\ell-16\ell}{3\ell^2}
		\right)
.
\end{align*}

This time there is an issue with $g(q^{3\ell^2-3b\ell-2\ell},q^{3\ell^2})$ when 
when $3b\equiv-2\pmod{\ell}$. To avoid this we first apply
Lemma \ref{ChanLemma1} with 
$q\mapsto q^{3\ell^2}$,
$b_1 = q^{3\ell}$,
$b_2 = zq^{-3\ell}$,
$b_3 = z^{-1}$,
$b_4 = z^2$	
and find that
\begin{align*}
	&
	\frac{\jacprod{z}{q^{3\ell^2}}}{\jacprod{z^3}{q^{3\ell^2}}}
		\SSeries{z}{z^6q^{-9\ell^2}}{q^{3\ell^2}}
	+z^5\frac{\jacprod{z}{q^{3\ell^2}}}{\jacprod{z^3}{q^{3\ell^2}}}
		\SSeries{z^2}{z^6}{q^{3\ell^2}}
	\\
	&=	
	\frac{\aqprod{q^{3\ell^2}}{q^{3\ell^2}}{\infty}^2
			\jacprod{z^2q^{-3\ell}, zq^{3\ell}}{q^{3\ell^2}}}
		{\jacprod{q^{3\ell}, zq^{-3\ell},z^2}{q^{3\ell^2}}}
		\\&\quad
	-q^{3\ell}\frac{\jacprod{z}{q^{3\ell^2}}}{\jacprod{zq^{-6\ell}}{q^{3\ell^2}}}
		\SSeries{q^{3\ell}}{z^{-2}q^{12\ell}}{q^{3\ell^2}}
	+z^2q^{-9\ell}\frac{\jacprod{z}{q^{3\ell^2}}}{\jacprod{zq^{-6\ell}}{q^{3\ell^2}}}
		\SSeries{zq^{-3\ell}}{z^2q^{-12\ell}}{q^{3\ell^2}}
.
\end{align*}

We then have
\begin{align*}
	&S^*(z^{-2},q^{3\ell^2})-z^2\SSeries{z}{z^2}{q^{3\ell^2}}
	\\
	&=
		g(z,q^{3\ell^2})
		-\frac{\aqprod{q^{3\ell^2}}{q^{3\ell^2}}{\infty}^2  
			\jacprod{z^2q^{-3\ell}, zq^{3\ell}}{q^{3\ell^2}} 
		}{\jacprod{q^{3\ell}, zq^{-3\ell}, z^2}{q^{3\ell^2}}}
		\\&\quad
		+q^{3\ell}\frac{\jacprod{z}{q^{3\ell^2}}}{\jacprod{zq^{-6\ell}}{q^{3\ell^2}}}
		\Parans{  
		\SSeries{q^{3\ell}}{z^{-2}q^{12\ell}}{q^{3\ell^2}}
		-z^{2}q^{-12\ell}\SSeries{zq^{-3\ell}}{z^{2}q^{-12\ell}}{q^{3\ell^2}}
		}
\end{align*}
and here we can set $z=q^{3\ell^2-3b\ell-2\ell}$.

Thus
\begin{align*}
	&V(6b+4)-q^{2b+2}U(6b+8)
	\\
	&=	
		g(q^{3\ell^2-3b\ell-2\ell},q^{3\ell^2})
		+
		\sum_{k=2}^{\ell-1} q^{6k^2+4k+6bk}
			\frac{\aqprod{q^{3\ell^2}}{q^{3\ell^2}}{\infty}^2
			\jacprod{
				q^{-3k\ell+3\ell}, 
				q^{3\ell^2-3b\ell-3k\ell-5\ell}, 
				q^{3\ell^2-3b\ell-6k\ell-2\ell}
			}{q^{3\ell^2}}
			}{\jacprod{
				q^{3\ell},
				q^{3\ell^2-3b\ell-5\ell},
				q^{3k\ell},
				q^{3\ell^2-3k\ell-3b\ell-2\ell}
			}{q^{3\ell^2}}}
		\\&\quad
		-
		\frac{\aqprod{q^{3\ell^2}}{q^{3\ell^2}}{\infty}^2
			\jacprod{
				q^{3\ell^2-3b\ell+\ell}, 
				q^{6\ell^2-6b\ell-7\ell} 
			}{q^{3\ell^2}}
		}{\jacprod{
			q^{3\ell},
			q^{3\ell^2 -3b\ell-5\ell},
			q^{6\ell^2 -6b\ell-4\ell}
		}{q^{3\ell^2}}}
		\\&\quad
		+Y\cdot\left(
			\SSeries{3\ell}{6b\ell+16\ell-6\ell^2}{3\ell^2}
			-q^{6\ell^2-6b\ell-16\ell}\SSeries{3\ell^2-3b\ell-5\ell}{6\ell^2-6b\ell-16\ell}{3\ell^2}	
		\right)
\end{align*}
where
\begin{align*}
	Y
	&=
		q^{10+6b}
		+
		q^{3\ell}\frac{\jacprod{
				q^{3\ell^2-3b\ell-2\ell}
			}{q^{3\ell^2}}
		}{\jacprod{
			q^{3\ell^2-3b\ell-8\ell} 
		}{q^{3\ell^2}}}
		+\sum_{k=2}^{\ell-1} q^{6k^2+4k+6bk-3k\ell+3\ell}
		\frac{\jacprod{
			q^{3\ell^2-3b\ell-6k\ell-2\ell}
		}{q^{3\ell^2}}
		}{\jacprod{
			q^{3\ell^2-3b\ell-8\ell} 
		}{q^{3\ell^2}}}
.
\end{align*}

We note
\begin{align*}
	Y
	&=
	\sum_{k=0}^{\ell-1}
	q^{6k^2+4k+6bk-3k\ell+3\ell}
	\frac{ \jacprod{q^{3b\ell+6k\ell+2\ell}}{q^{3\ell^2}} }
	{ \jacprod{q^{3b\ell+8\ell}}{q^{3\ell^2}} }
\end{align*}

In Euler's Pentagonal Numbers Theorem
we replace $n$ by $\ell n -\frac{\ell-1}{2}+2k+b$ to obtain
\begin{align*}
	\aqprod{q}{q}{\infty}
	&=
	\sum_{n=-\infty}^\infty (-1)^n q^{\frac{n(3n+1)}{2}}
	\\
	&=
	(-1)^{\frac{\ell-1}{2}+b}
	q^{-\ell+\frac{3b^2-3b\ell}{2}+2b+\frac{3\ell^2+5}{8}}
	\aqprod{q^{3\ell^2}}{q^{3\ell^2}}{\infty}
	\sum_{k=0}^{\ell-1}
	q^{6k^2+4k+6bk-3k\ell}
	\jacprod{q^{2\ell + 6k\ell + 3b\ell}}{q^{3\ell^2}}
\end{align*}
so that
\begin{align*}
	Y
	&=
	(-1)^{\frac{\ell-1}{2}+b}
	q^{4\ell+\frac{3b\ell-3b^2}{2}-2b-\frac{3\ell^2+5}{8}}
	\frac{\aqprod{q}{q}{\infty}}
	{\aqprod{q^{3\ell^2}}{q^{3\ell^2}}{\infty}
		\jacprod{q^{3b\ell+8\ell}}{q^{3\ell^2}} 
	}
.
\end{align*}
\end{proof}

Next we need identities for $g(z,q)$ and $h(z,q)$.
From the limit definitions, we find that $g(z,q)$ and $h(z,q)$ are basically logarithmic
derivatives of theta functions and are surprisingly related.
\begin{lemma}\label{LemmaLogThetaForGAndG*}
\begin{align*}
	g(z,q)
	&=
	1 - \sum_{n=0}^\infty \frac{z^2q^n}{1-z^2q^n} 
	+ \sum_{n=1}^\infty \frac{z^{-2}q^n}{1-z^{-2}q^n} 
	,\\
	h(z,q)
	&=
	- \sum_{n=0}^\infty \frac{zq^n}{1-zq^n} 
	+ \sum_{n=1}^\infty \frac{z^{-1}q^n}{1-z^{-1}q^n} 
.
\end{align*}
\end{lemma}
\begin{proof}
We have
\begin{align*}
	g(z,q) 
	&=
		\lim_{w\rightarrow z}
		\Parans{
			\frac{\jacprod{w^2z^{-1},z^{-1},w^2}{q}\aqprod{q}{q}{\infty}^2}
			{\jacprod{w,w^{-1},zw,zw^{-1}}{q}}
			-\frac{1}{1-z/w}
		}
	\\
	&=
		\lim_{w\rightarrow z}
		\frac{
			\frac{\jacprod{w^2z^{-1},z^{-1},w^2}{q}\aqprod{q}{q}{\infty}^2}
			{\jacprod{w,w^{-1},zw}{q}\aqprod{zw^{-1}q,z^{-1}wq}{q}{\infty} }
			-
			1
		}{1-z/w}
	\\		
	&=
		\lim_{w\rightarrow z}
		z\cdot
		\frac{\partial}{\partial w}
			\frac{\jacprod{w^2z^{-1},z^{-1},w^2}{q}\aqprod{q}{q}{\infty}^2}
			{\jacprod{w,w^{-1},zw}{q}\aqprod{zw^{-1}q,z^{-1}wq}{q}{\infty}}
.
\end{align*}
We handle the partial derivative with a logarithmic derivative. We have
\begin{align*}
	&
	\frac{\jacprod{w,w^{-1},zw}{q}\aqprod{zw^{-1}q,z^{-1}wq}{q}{\infty}}
	{\jacprod{w^2z^{-1},z^{-1},w^2}{q}\aqprod{q}{q}{\infty}^2}
	\cdot		
	\frac{\partial}{\partial w}
	\frac{\jacprod{w^2z^{-1},z^{-1},w^2}{q}\aqprod{q}{q}{\infty}^2}
	{\jacprod{w,w^{-1},zw}{q}\aqprod{zw^{-1}q,z^{-1}wq}{q}{\infty}}
	\\&=
		-2\sum_{n=0}^\infty \frac{wz^{-1}q^n}{1-w^{2}z^{-1}q^n}
		+2\sum_{n=1}^\infty \frac{w^{-3}zq^n}{1-w^{-2}zq^n}
		-2\sum_{n=0}^\infty \frac{wq^n}{1-w^{2}zq^n}
		+2\sum_{n=1}^\infty \frac{w^{-3}q^n}{1-w^{-2}q^n}
		+\sum_{n=0}^\infty \frac{q^n}{1-wq^n}
		\\&\quad
		-\sum_{n=1}^\infty \frac{w^{-2}q^n}{1-w^{-1}q^n}
		-\sum_{n=0}^\infty \frac{w^{-2}q^n}{1-w^{-1}q^n}
		+\sum_{n=1}^\infty \frac{q^n}{1-wq^n}
		+\sum_{n=0}^\infty \frac{zq^n}{1-wzq^n}
		-\sum_{n=1}^\infty \frac{w^{-2}z^{-1}q^n}{1-w^{-1}z^{-1}q^n}
		\\&\quad
		-\sum_{n=1}^\infty \frac{w^{-2}zq^n}{1-w^{-1}zq^n}
		+\sum_{n=1}^\infty \frac{z^{-1}q^n}{1-wz^{-1}q^n}
.
\end{align*}
Multiplying by $z$, letting $w\rightarrow z$, and simplifying then gives
the result. The proof for $h(z,q)$ is similar.
\end{proof}

Here we see taking $b_4=1/b_1$ gives nice cancellations.
Since $h(z,q) = g(z^{1/2},q)-1$, 
we see we need only prove identities for one of the functions.
Using Lemma \ref{LemmaLogThetaForGAndG*} we can 
quickly deduce the following identities.
\begin{lemma}\label{LemmaGShifts}
\begin{align*}
	g(z,q) - g(zq,q) &= 2
	,\\
	g(z,q) + g(q/z,q) &= 1
	,\\
	h(z,q) -  h(zq,q) &= 1
	,\\
	h(z,q) + h(z/q,q) &= 0 
.
\end{align*}
\end{lemma}

We will also need formulas that turn $g(z,q)$ and $h(z,q)$ into products. 
For the $3$-dissections, we can actually just use a product identity for a 
similar function from \cite{AS}. We will use this formula only for $g$ and not
$h$.
\begin{lemma}\label{CorollaryGProducts}
\begin{align*}
	3 - 2g(z,q) - g(z^2,q) + g(z^4,q)
	&=
	\frac{\jacprod{z^6}{q}^2\aqprod{q}{q}{\infty}^2}
	{\jacprod{z^2}{q}^2\jacprod{z^8}{q}}
.
\end{align*}
\end{lemma}
\begin{proof}
We let
\begin{align*}
	k(z,q)
	&=
	z\frac{\jacprod{z^2}{q}}{\jacprod{z}{q}}
	\sum_{n=-\infty}^\infty \frac{(-1)^nq^{3n(n+1)/2}}{1-zq^n}
	-z^3\sum_{n=-\infty}^\infty \frac{(-1)^nq^{3n(n+1)/2}z^{3n}}{1-z^2q^n}
	-\sum_{n\not=0} \frac{(-1)^nq^{3n(n+1)/2}z^{-3n}}{1-q^n}
.
\end{align*}
With Lemma 7 of \cite{AS} (which is also follows by a specialization of 
Theorem 2.1 \cite{Chan}) ,we find that
\begin{align*}
	k(z,q) &= 
	\lim_{w\rightarrow z}
	\frac{1}{1-z/w}
	-
	\frac{\jacprod{w,w^2}{q}\aqprod{q}{q}{\infty}}{\jacprod{z,zw,zw^{-1}}{q}}
	\\
	&=
	\sum_{n=0}^\infty \frac{zq^n}{1-zq^n}
	-\sum_{n=1}^\infty \frac{z^{-1}q^n}{1-z^{-1}q^n}
	+\sum_{n=0}^\infty \frac{z^2 q^n}{1-z^2 q^n}
	-\sum_{n=1}^\infty \frac{z^{-2}q^n}{1-z^{-2}q^n}
,	
\end{align*}
so that
\begin{align*}
	k(z,q) &= 2 - g(z^{1/2},q)  - g(z,q).
\end{align*}

Replacing $z$ by $z^2$ in equation (5.6) of \cite{AS} we have
\begin{align*}
	2k(z^2,q) - k(z^4,q) + 1
	&=
	\frac{\jacprod{z^6}{q}^2\aqprod{q}{q}{\infty}^2}
	{\jacprod{z^2}{q}^2\jacprod{z^8}{q}}
.
\end{align*}
But also $2k(z^2,q) - k(z^4,q) + 1 = 	3 - 2g(z,q) - g(z^2,q) + g(z^4,q)$
and so we are done.
\end{proof}

For the $5$ and $7$-dissections, we need the following product formulas.
\begin{lemma}\label{LemmaGToProducts}
\begin{align*}
	4g(z,q) - 2g(z^2,q)
	&=
	3 
	- \frac{z^2\aqprod{q}{q}{\infty}^2\jacprod{z^2,z^8}{q}}
		{\jacprod{z^4}{q}^2\jacprod{z^6}{q}}
	- \frac{\aqprod{q}{q}{\infty}^2\jacprod{z^4}{q}^3}
		{\jacprod{z^2}{q}^3\jacprod{z^6}{q}}
	,\\
	4h(z,q) - 2h(z^2,q)
	&=
		1 
		- \frac{z\aqprod{q}{q}{\infty}^2\jacprod{z,z^4}{q}}
			{\jacprod{z^2}{q}^2\jacprod{z^3}{q}}
		- \frac{\aqprod{q}{q}{\infty}^2\jacprod{z^2}{q}^3}
			{\jacprod{z}{q}^3\jacprod{z^3}{q}}
.
\end{align*}
\end{lemma}
\begin{proof}
We need only prove the formula for $h(z,q)$, the formula for $g(z,q)$ then
follows.

By Lemma \ref{LemmaGShifts} we find that $4h(z,q) - 2h(z^2,q)$ is
invariant under $z\mapsto zq$, as are the products
$\frac{z\aqprod{q}{q}{\infty}^2\jacprod{z,z^4}{q}}
{\jacprod{z^2}{q}^2\jacprod{z^3}{q}}$
and
$\frac{\aqprod{q}{q}{\infty}^2\jacprod{z^2}{q}^3}
{\jacprod{z}{q}^3\jacprod{z^3}{q}}$.
We define
\begin{align*}
	F(z) := F(z,q)
	&=
	4h(z,q) - 2h(z^2,q)
	+\frac{z\aqprod{q}{q}{\infty}^2\jacprod{z,z^4}{q}}
		{\jacprod{z^2}{q}^2\jacprod{z^3}{q}}
	+\frac{\aqprod{q}{q}{\infty}^2\jacprod{z^2}{q}^3}
		{\jacprod{z}{q}^3\jacprod{z^3}{q}}
	-1
.
\end{align*}
By Lemma $2$ of \cite{AS}, if $F(z)$ is not identically zero, then $F$ has
exactly as many poles and zeros in the region $|q|<|z|\le 1$. Our proof is then
to show all the poles cancel out and find a single zero of $F(z)$.

First we show $q^{1/3}$ is a zero of $F$, to make the calculations cleaner
we use $q\mapsto q^3$ and $z=q$. 
By Lemma \ref{LemmaLogThetaForGAndG*} we have
\begin{align*}
	4h(q,q^3) - 2h(q^2,q^3)
	&=
	-6\sum_{n=0}^\infty \frac{q^{3n+1}}{1-q^{3n+1}}
	+6\sum_{n=0}^\infty \frac{q^{3n+2}}{1-q^{3n+2}}
.
\end{align*}

For the products we have
\begin{align*}
	&\lim_{z\rightarrow q}
	\left.
	\Parans{
		\frac{z\aqprod{q^3}{q^3}{\infty}^2 \jacprod{z,z^4}{q^3}}
			{\jacprod{z^2}{q^3}^2\aqprod{z^3,q^{6}/z^3}{q^3}{\infty}}
		+\frac{\aqprod{q^3}{q^3}{\infty}^2 \jacprod{z^2}{q^3}^3}
			{\jacprod{z}{q^3}^3\aqprod{z^3,q^{6}/z^3}{q^3}{\infty}}
	}\right/ \Parans{1-q^3/z^3}
	\\
	&=
	\frac{q}{3}
	\lim_{z\rightarrow q}
	\frac{\partial}{\partial z}\Parans{
		\frac{z\aqprod{q^3}{q^3}{\infty}^2 \jacprod{z,z^4}{q^3}}
			{\jacprod{z^2}{q^3}^2\aqprod{z^3,q^{6}/z^3}{q^3}{\infty}}
		+\frac{\aqprod{q^3}{q^3}{\infty}^2 \jacprod{z^2}{q^3}^3}
			{\jacprod{z}{q^3}^3\aqprod{z^3,q^{6}/z^3}{q^3}{\infty}}
	}	
.
\end{align*}
We note that
\begin{align*}
	\lim_{z\rightarrow q}
	\frac{z\aqprod{q^3}{q^3}{\infty}^2 \jacprod{z,z^4}{q^3}}
		{\jacprod{z^2}{q^3}^2\aqprod{z^3,q^{6}/z^3}{q^3}{\infty}}
	&= -1
	,
	&\lim_{z\rightarrow q}
	\frac{\aqprod{q^3}{q^3}{\infty}^2 \jacprod{z^2}{q^3}^3}
		{\jacprod{z}{q^3}^3\aqprod{z^3,q^{6}/z^3}{q^3}{\infty}}
	&= 1
,
\end{align*}
so logarithmic differentiation yields
\begin{align*}
	&q\lim_{z\rightarrow q}
	\frac{\partial}{\partial z}
		\frac{z\aqprod{q^3}{q^3}{\infty}^2 \jacprod{z,z^4}{q^3}}
			{\jacprod{z^2}{q^3}^2\aqprod{z^3,q^{6}/z^3}{q^3}{\infty}}
	\\
	&= 
	-1
	+\sum_{n=0}^\infty \frac{q^{3n+1}}{1-q^{3n+1}}
	-\sum_{n=1}^\infty \frac{q^{3n-1}}{1-q^{3n-1}}
	+4\sum_{n=0}^\infty \frac{q^{3n+4}}{1-q^{3n+4}}
	-4\sum_{n=1}^\infty \frac{q^{3n-4}}{1-q^{3n-4}}
		\\&\quad
		-4\sum_{n=0}^\infty \frac{q^{3n+2}}{1-q^{3n+2}}
		+4\sum_{n=1}^\infty \frac{q^{3n-2}}{1-q^{3n-2}}
		-3\sum_{n=0}^\infty \frac{q^{3n+3}}{1-q^{3n+3}}
		+3\sum_{n=2}^\infty \frac{q^{3n-3}}{1-q^{3n-3}}
	\\
	&=
	3 + 6\sum_{n=0}^\infty \frac{q^{3n+1}}{1-q^{3n+1}}
	-6\sum_{n=0}^\infty \frac{q^{3n+2}}{1-q^{3n+2}}
\end{align*}
and similarly
\begin{align*}
	q\lim_{z\rightarrow q}
	\frac{\partial}{\partial z}\Parans{
		\frac{\aqprod{q^3}{q^3}{\infty}^2 \jacprod{z^2}{q^3}^3}
			{\jacprod{z}{q^3}^3\aqprod{z^3,q^{6}/z^3}{q^3}{\infty}}
	}	
	&=
	9\sum_{n=0}^\infty \frac{q^{3n+1}}{1-q^{3n+1}}
	-9\sum_{n=0}^\infty \frac{q^{3n+2}}{1-q^{3n+2}}
	.
\end{align*}
We then have
\begin{align*}
	\lim_{z\rightarrow q}
		\frac{z\aqprod{q^3}{q^3}{\infty}^2 \jacprod{z,z^4}{q^3}}
			{\jacprod{z^2}{q^3}^2\jacprod{z^3}{q^3}}
		+\frac{\aqprod{q^3}{q^3}{\infty}^2 \jacprod{z^2}{q^3}^3}
			{\jacprod{z}{q^3}^3\jacprod{z^3}{q^3}}
	&=
	1 + 6\sum_{n=0}^\infty \frac{q^{3n+1}}{1-q^{3n+1}} 
	-6\sum_{n=0}^\infty \frac{q^{3n+2}}{1-q^{3n+2}}
,
\end{align*}
and so $q^{1/3}$ is a zero of $F(z)$.

We see $F(z)$ has at worst simple poles when $z$, $z^2$, or $z^3$ is an 
integral power of $q$.
Elementary manipulations show that the poles between the two products cancel 
out for $z^3 = q$ and $q^2$ and for $z$ a primitive third root of unity. Thus we need only compute residues for
$z=1,-1,q^{1/2},$ and $-q^{1/2}$. For $h(z,q)$ these residues are
$1,0,0,$ and $0$; for $h(z^2,q)$ these residues are
$\frac{1}{2},-\frac{1}{2},\frac{q^{1/2}}{2},$ and $-\frac{q^{1/2}}{2}$; for 
$\frac{z\aqprod{q}{q}{\infty}^2\jacprod{z,z^4}{q}}
{\jacprod{z^2}{q}^2\jacprod{z^3}{q}}$
these residues are $-\frac{1}{3}, -1, q^{1/2}$ and $-q^{1/2}$;
and for 
$\frac{\aqprod{q}{q}{\infty}^2\jacprod{z^2}{q}^3}
{\jacprod{z}{q}^3\jacprod{z^3}{q}}$
these residues are $-\frac{8}{3}, 0, 0$ and $0$.
Thus at all four points the residues cancel out in $F(z)$ so that
$F(z)$ has no poles in $|q|<|z|\le 1$. Thus $F$ is identically zero and
the theorem holds.
\end{proof}

We can use Lemma \ref{LemmaGToProducts} to expand $g(q^{\ell a},q^{3\ell^2})$ 
into products as follows. Suppose $n$ is a positive integer with 
$2^n\equiv 1\pmod{3\ell}$ and
$2^n-1=b3\ell$. Then by applying Lemma \ref{LemmaGShifts} $ba$ times,with
$q\mapsto q^{3\ell^2}$ and $z = q^{\ell a}$, we have
\begin{align*}
	g(q^{\ell a};q^{3\ell^2})
	&=
	g(q^{\ell a+3\ell^2ba};q^{3\ell^2}) + 2ba
,
\end{align*}
and so
\begin{align}\label{EqGToProducts}
	\nonumber
	&(2^{n+1}-2)g(q^{\ell a};q^{3\ell^2})
	=
	2^{n+1}g(q^{\ell a};q^{3\ell^2})
	-2g(q^{\ell a(1+3\ell b)};q^{3\ell^2}) - 4ba
	\\\nonumber
	&=
	2^{n+1}g(q^{\ell a};q^{3\ell^2})
	-2g(q^{\ell a2^n};q^{3\ell^2}) - 4ba	
	\\\nonumber
	&=
	-4ba + \sum_{k=0}^{n-1}2^{n-k-1}\left(
		4g(q^{\ell a2^k};q^{3\ell^2})-2g(q^{\ell a2^{k+1}};q^{3\ell^2})
		\right)
	\\\nonumber
	&=
	-4ba + \sum_{k=0}^{n-1}2^{n-k-1}\left(
		3
		-\frac{ q^{\ell a2^{k+1}} \aqprod{q^{3\ell^2}}{q^{3\ell^2}}{\infty}^2
				\jacprod{ q^{\ell a2^{k+1}}, q^{\ell a2^{k+3}}}{q^{3\ell^2}}  }   
		{ \jacprod{q^{\ell a2^{k+2}}}{q^{3\ell^2}}^2  \jacprod{q^{3\ell a2^{k+1}}}{q^{3\ell^2}}}
		\right.\\&\quad\left.
		-\frac{\aqprod{q^{3\ell^2}}{q^{3\ell^2}}{\infty}^2 \jacprod{q^{\ell a2^{k+2}}}{q^{3\ell^2}}^3}
		{\jacprod{q^{\ell a2^{k+1}}}{q^{3\ell^2}}^3\jacprod{q^{3\ell a2^{k+1}}}{q^{3\ell^2}}}
	\right)
.
\end{align}
Similarly we have
\begin{align}
	\label{EqHToProducts}
	\nonumber
	&(2^{n+1}-2)h(q^{\ell a};q^{3\ell^2})
	\\
	&=
	-2ba + \sum_{k=0}^{n-1}2^{n-k-1}\left(
		1
		-\frac{ q^{\ell a2^{k}} \aqprod{q^{3\ell^2}}{q^{3\ell^2}}{\infty}^2
				\jacprod{ q^{\ell a2^{k}}, q^{\ell a2^{k+2}}}{q^{3\ell^2}}  }   
		{ \jacprod{q^{\ell a2^{k+1}}}{q^{3\ell^2}}^2  \jacprod{q^{3\ell a2^{k}}}{q^{3\ell^2}}}
		-\frac{\aqprod{q^{3\ell^2}}{q^{3\ell^2}}{\infty}^2 \jacprod{q^{\ell a2^{k+1}}}{q^{3\ell^2}}^3}
		{\jacprod{q^{\ell a2^{k}}}{q^{3\ell^2}}^3\jacprod{q^{3\ell a2^{k}}}{q^{3\ell^2}}}
	\right)
.
\end{align}

As an example, these identities give
\begin{align*}
	h(q^5,q^{75})
	=
	g(q^{40},q^{75})
	&=	
	\frac{13}{30} 
	-\frac{4}{15}\frac{q^{5}\aqprod{q^{75}}{q^{75}}{\infty}^2\jacprod{q^{5},q^{20}}{q^{75}}}
		{\jacprod{q^{10}}{q^{75}}^2\jacprod{q^{15}}{q^{75}}}
	-\frac{4}{15}\frac{\aqprod{q^{75}}{q^{75}}{\infty}^2\jacprod{q^{10}}{q^{75}}^3}
		{\jacprod{q^{5}}{q^{75}}^3\jacprod{q^{15}}{q^{75}}} 
	\\&\quad
	-\frac{2}{15}\frac{q^{10}\aqprod{q^{75}}{q^{75}}{\infty}^2\jacprod{q^{10},q^{35}}{q^{75}}}
		{\jacprod{q^{20}}{q^{75}}^2\jacprod{q^{30}}{q^{75}}}
	-\frac{2}{15}\frac{\aqprod{q^{75}}{q^{75}}{\infty}^2\jacprod{q^{20}}{q^{75}}^3}
		{\jacprod{q^{10}}{q^{75}}^3\jacprod{q^{30}}{q^{75}}}
	\\&\quad
	+\frac{1}{15}\frac{q^{15}\aqprod{q^{75}}{q^{75}}{\infty}^2\jacprod{q^{5},q^{20}}{q^{75}}}
		{\jacprod{q^{15}}{q^{75}}\jacprod{q^{35}}{q^{75}}^2}
	-\frac{1}{15}\frac{\aqprod{q^{75}}{q^{75}}{\infty}^2\jacprod{q^{35}}{q^{75}}^3}
		{\jacprod{q^{15}}{q^{75}}\jacprod{q^{20}}{q^{75}}^3} 
	\\&\quad
	+\frac{1}{30}\frac{\aqprod{q^{75}}{q^{75}}{\infty}^2\jacprod{q^{10},q^{35}}{q^{75}}}
		{\jacprod{q^{5}}{q^{75}}^2\jacprod{q^{30}}{q^{75}}}
	-\frac{1}{30}\frac{q^{30}\aqprod{q^{75}}{q^{75}}{\infty}^2\jacprod{q^{5}}{q^{75}}^3}
		{\jacprod{q^{30}}{q^{75}}\jacprod{q^{35}}{q^{75}}^3}
.
\end{align*}

We can now describe the method used to prove our Theorems. For 
$PP_1$ and $PP_2$ we use Lemmas \ref{LemmaFor6b1} and \ref{LemmaFor6b4}
to write the series in terms of $g(z,q)$, two $\SSeries{z}{w}{q}$ series,
and a sum of products. The two $\SSeries{z}{w}{q}$ series have a factor
of $\aqprod{q}{q}{\infty}$ out front and so they are immediately dealt with.
We can expand $g(z,q)$ into a sum of products, so we only have an identity
between products to verify. For $PP_3$ and $PP_4$ we instead use Lemma 
\ref{LemmaFor3b2} and so there are only products.
While
these identities can be split into several smaller identities, some of which
follow from simple rearrangements, there is always at least one identity that
is not easily proved by hand. For this reason we just prove the single larger
equality between infinite products by recognizing them as modular functions. 
In particular, the eta function is defined by
\begin{align*}
	\eta(\tau) &= q^{1/24}\aqprod{q}{q}{\infty},
\end{align*}
and the generalized eta function is defined by
\begin{align*}
	\GEta{\delta}{g}{\tau} 
	&= 
	q^{P(g/\delta)\delta/2 }
	\prod_{\substack{n>0\\n\equiv g \pmod{d}}} (1-q^n)
	\prod_{\substack{n>0\\n\equiv -g \pmod{d}}} (1-q^n)
,
\end{align*}
where $q = e^{2\pi i \tau}$ and
$P(t) = \CBrackets{t}^2-\CBrackets{t}+\frac{1}{6}$. 
So $\GEta{\delta}{0}{\tau}=\eta(\delta\tau)^2$ and 
$\GEta{\delta}{g}{\tau} = q^{P(g/\delta)\delta/2 }\jacprod{q^g}{q^{\delta}}$
for $0<g<\delta$.
We use Theorem 3 of
\cite{Robins} to determine when a quotient of $\GEta{\delta}{g}{\tau}$
is a modular function with respect to a congruence subgroup 
$\Gamma_1(N)$ and use Theorem 4 of \cite{Robins} to determine the order
at the cusps.

We recall some facts about modular functions as in \cite{Rankin} and use 
the notation in \cite{Berndt}. 
Suppose $f$ is a modular function with respect to the congruence subgroup $\Gamma$ 
of $\Gamma_0(1)$. For $A\in\Gamma_0(1)$ we have a cusp given 
by $\zeta=A^{-1}\infty$. The width of the cusp $N:=N(\Gamma,\zeta)$ is
given by
\begin{align*}
	N(\Gamma,\zeta) &= \min\{ k>0:\pm A^{-1}T^kA\in\Gamma \},
\end{align*}
where $T$ is the translation matrix 
\begin{align*}
	T &= {\Parans{\begin{array}{cc}
				1&1\\
				0&1
			\end{array}}}
.
\end{align*}

If
\begin{align*}
	f(A^{-1}\tau) &= \sum_{m=m_0}^\infty b_mq^{m/N}
\end{align*}
and $b_{m_0}\not=0$, then we say $m_0$ is the order of $f$ at
$\zeta$ with respect to $\Gamma$ and we denote this value 
by $Ord_\Gamma(f;\zeta)$. By
$ord(f;\zeta)$ we mean the invariant order of $f$ at $\zeta$
given by
\begin{align*}
	ord(f;\zeta) = \frac{Ord_\Gamma(f;\zeta)}{N}.
\end{align*}

For $z$ in the upper half plane $\mathcal{H}$, we write 
$ord(f;z)$ for the order of $f$ at $z$ as an analytic function
in $z$. We define the order of $f$ at $z$ with respect to
$\Gamma$ by
\begin{align*}
	Ord_\Gamma(f;z) = \frac{ord(f;z)}{m},
\end{align*}
where $m$ is the order of $z$ as a fixed point of $\Gamma$.

The valence formula for modular functions is as follows.
Suppose a subset $\mathcal{F}$ of 
$\mathcal{H}\cup\{\infty\}\cup\mathbb{Q}$ is a fundamental
region for the action of $\Gamma$ along with a complete set of
inequivalent cusps, if $f$ is not the zero 
function then
\begin{align}
	\sum_{z\in\mathcal{F}}Ord_\Gamma(f;z) &=0.
\end{align}

We can verify an identity between sums of generalized eta quotients as follows.
Suppose we are to show
\begin{align*}
	a_1f_1 + a_2f_2 + \dots + a_kf_k &=	a_{k+1}f_{k+1} + a_{k+2}f_{k+2} + \dots + a_{k+m}f_{k+m}
,\end{align*}
where each $a_i\in\mathbb{C}$ and each $f_i$ is of the form
\begin{align*}
	f_i &= \prod_{j=1}^{m}  \GEta{\delta_j}{g_j}{\tau}^{r_j}   
.
\end{align*}
We verify each $f_i$ is a modular function with respect to a common
$\Gamma_1(N)$, so that $f = a_1f_1+\dots+a_kf_k - a_{k+1}f_{k+1}-\dots-a_{k+m}f_{k+m}$ 
is a modular function
with respect to $\Gamma_1(N)$. Although $f$ may have zeros at points other than
the cusps, the poles must occur only at the cusps.
At each cusp $\zeta$, not equivalent to $\infty$, we compute a lower bound
for $Ord_\Gamma(f;\zeta)$ by taking the minimum of the $Ord_\Gamma(f_i,\zeta)$
, we call this lower bound $B_\zeta$. We then use the $q$-expansion
of $f$ to find $Ord_\Gamma(f;\infty)$ is larger than 
$-\sum_{\zeta\in\mathcal{C}'} B_\zeta$, where $\mathcal{C}'$ is a set 
of cusps with a representative of each cusp not equivalent to $\infty$. By the
valence formula we have $f\equiv 0$ since 
$\sum_{z\in\mathcal{F}}Ord_\Gamma(f;z) >0$.

We cannot expand $g(z,q)$ into products when $z^3$ is a power of
$q$, however these terms can still be viewed as modular forms. 
We let $\chi(n) = \Jac{-3}{n}$ and set
\begin{align*}
	V_{\chi,1}(\tau) 
	&= 
	\frac{1}{6} + \sum_{m=1}^\infty\sum_{n=1}^\infty\chi(n)q^{mn}
	\\
	&= 
	\frac{1}{6} 
	+ \sum_{n=1}^\infty \frac{\chi(n)q^{n}}{1-q^n}
	\\
	&= 
	\frac{1}{6} 
	+ \sum_{n=0}^\infty \frac{q^{3n+1}}{1-q^{3n+1}}
	- \sum_{n=0}^\infty \frac{q^{3n+2}}{1-q^{3n+2}}
	\\
	&=
 	g(q,q^3)-\frac{5}{6}
.
\end{align*}
By \cite{Kolberg} this is a weight $1$ modular form with respect to $\Gamma_0(3)$
and character $\chi$. We note we also have
$V_{\chi,1}(\tau) = -g(q^2,q^3)+\frac{1}{6} = h(q^2,q^3)+\frac{1}{6}$.


\section{Proof of Theorem \ref{TheoremPP13} }
We note $V_3^6(2)=-V_3^6(7)$ and $V_3^6(5)=-V_3^6(4)$ and so we set
\begin{align*}
	A &= \frac{1}{3} + V_3^6(4) - q^2U_3^6(8) - V_3^6(7) + q^3U_3^6(11)
.
\end{align*}
By (\ref{PP1Zeta3}) the left hand side of Theorem \ref{TheoremPP13}
is $\frac{A}{\aqprod{q}{q}{\infty}}$,
so we must show
\begin{align}\label{EqThingToProve}
	\nonumber
	\frac{A}{\aqprod{q}{q}{\infty}}
	&=
	\frac{\aqprod{q^{27}}{q^{27}}{\infty}^2}{\aqprod{q^3}{q^3}{\infty}}
	\left(
		\frac{1}{3}\jacprod{q^{12}}{q^{27}}^2 
		+\frac{2}{3}q^3\jacprod{q^3,q^6}{q^{27}} 	
		+\frac{1}{3}q\jacprod{q^6,q^{12}}{q^{27}} 
		+\frac{1}{3} q^4\jacprod{q^3}{q^{27}}^2
		\right.\\\nonumber&\left.\quad
		+\frac{1}{3}q^2\jacprod{q^3,q^{12}}{q^{27}} 
		-\frac{2}{3}q^2\jacprod{q^6}{q^{27}}^2
	\right)
	\\\nonumber&\quad
	-\frac{q^{8}}{\aqprod{q^{27}}{q^{27}}{\infty}\jacprod{q^{3}}{q^{27}}}
		\left(\SSeries{9}{-6}{27}-q^{6}\SSeries{12}{6}{27}\right)
	\\&\quad
	+\frac{q^{4}}{\aqprod{q^{27}}{q^{27}}{\infty}\jacprod{q^{12}}{q^{27}}}
		\left(\SSeries{-9}{-84}{27}-q^{84}\SSeries{33}{84}{27}\right)
.
\end{align}

But by Lemmas \ref{LemmaFor6b1} and \ref{LemmaFor6b4}
\begin{align*}
	A
	&=
	\frac{1}{3} + g(q^{21},q^{27})  +g(q^{24},q^{27})
	-q^{2}\frac{\aqprod{q^{27}}{q^{27}}{\infty}^2\jacprod{q^{6}}{q^{27}}}
		{\jacprod{q^{3},q^{9}}{q^{27}}}
	+q^{6}\frac{\aqprod{q^{27}}{q^{27}}{\infty}^2\jacprod{q^{3},q^{6}}{q^{27}}}
		{\jacprod{q^{9},q^{12},q^{12}}{q^{27}}}
	\\&\quad
	-q^{5}\frac{\aqprod{q^{27}}{q^{27}}{\infty}^2\jacprod{q^{3}}{q^{27}}}
		{\jacprod{q^{9},q^{12}}{q^{27}}}
	+q^{3}\frac{\aqprod{q^{27}}{q^{27}}{\infty}^2\jacprod{q^{3},q^{12}}{q^{27}}}
		{\jacprod{q^{6},q^{6},q^{9}}{q^{27}}}
	\\&\quad
	-q^{8}\frac{\aqprod{q}{q}{\infty}}{\aqprod{q^{27}}{q^{27}}{\infty}\jacprod{q^{3}}{q^{27}}}
		\left(\SSeries{9}{-6}{27}-q^{6}\SSeries{12}{6}{27}\right)
	\\&\quad
	+q^{4}\frac{\aqprod{q}{q}{\infty}}{\aqprod{q^{27}}{q^{27}}{\infty}\jacprod{q^{12}}{q^{27}}}
		\left(\SSeries{-9}{-84}{27}-q^{84}\SSeries{33}{84}{27}\right)
.
\end{align*}

By Lemma \ref{CorollaryGProducts} we have
\begin{align*}
	1+3g(q^{24},q^{27})+3g(q^{21},q^{27})
	&=
	6 - 2g(q^3,q^{27}) - 3g(q^{6},q^{27}) + g(q^{24},q^{27})	
	\\
	&=
	3 - 2g(q^3,q^{27}) - g(q^{6},q^{27}) + g(q^{12},q^{27})
		\\&\quad
		+3 - 2g(q^6,q^{27}) - g(q^{12},q^{27}) + g(q^{24},q^{27})	
	\\
	&=
	\frac{\aqprod{q^{27}}{q^{27}}{\infty}^2\jacprod{q^{9}}{q^{27}}^2}
	{\jacprod{q^6}{q^{27}}^2\jacprod{q^{3}}{q^{27}}}
	-
	q^{3}\frac{\aqprod{q^{27}}{q^{27}}{\infty}^2\jacprod{q^9}{q^{27}}^2}
	{\jacprod{q^{12}}{q^{27}}^2\jacprod{q^{6}}{q^{27}}}
.
\end{align*}

Noting the terms involving $\SSeries{9}{-6}{27}-q^{6}\SSeries{12}{6}{27}$
and $\SSeries{-9}{-84}{27}-q^{84}\SSeries{33}{84}{27}$ work out as claimed
in (\ref{EqThingToProve}),
by combining the products we find it only remains to show that
\begin{align}\label{EqPP1Mod3Eq1}
	&\quad
	-q^{2}\frac{\aqprod{q^{27}}{q^{27}}{\infty}^2\jacprod{q^{6}}{q^{27}}}
		{\jacprod{q^{3},q^{9}}{q^{27}}}
	+q^{6}\frac{\aqprod{q^{27}}{q^{27}}{\infty}^2\jacprod{q^{3},q^{6}}{q^{27}}}
		{\jacprod{q^{9},q^{12},q^{12}}{q^{27}}}
	-q^{5}\frac{\aqprod{q^{27}}{q^{27}}{\infty}^2\jacprod{q^{3}}{q^{27}}}
		{\jacprod{q^{9},q^{12}}{q^{27}}}
	\nonumber\\&\quad
	+q^{3}\frac{\aqprod{q^{27}}{q^{27}}{\infty}^2\jacprod{q^{3},q^{12}}{q^{27}}}
		{\jacprod{q^{6},q^{6},q^{9}}{q^{27}}}
	+\frac{1}{3}\frac{\aqprod{q^{27}}{q^{27}}{\infty}^2\jacprod{q^{9}}{q^{27}}^2}
		{\jacprod{q^6}{q^{27}}^2\jacprod{q^{3}}{q^{27}}}
	-
	\frac{1}{3}q^{3}\frac{\aqprod{q^{27}}{q^{27}}{\infty}^2\jacprod{q^9}{q^{27}}^2}
		{\jacprod{q^{12}}{q^{27}}^2\jacprod{q^{6}}{q^{27}}}
	\nonumber\\
	&=
	\frac{\aqprod{q}{q}{\infty}\aqprod{q^{27}}{q^{27}}{\infty}^2}{\aqprod{q^3}{q^3}{\infty}}
	\left(
		\frac{1}{3}\jacprod{q^{12}}{q^{27}}^2 
		+\frac{2}{3}q^3\jacprod{q^3,q^6}{q^{27}} 	
		+\frac{1}{3}q\jacprod{q^6,q^{12}}{q^{27}} 
		+\frac{1}{3} q^4\jacprod{q^3}{q^{27}}^2
		\right.\nonumber\\&\left.\quad
		+\frac{1}{3}q^2\jacprod{q^3,q^{12}}{q^{27}} 
		-\frac{2}{3}q^2\jacprod{q^6}{q^{27}}^2
	\right)
.
\end{align}

If we multiply both sides of (\ref{EqPP1Mod3Eq1}) by 
$q^{-2}\frac{\aqprod{q}{q}{\infty}\jacprod{q^3,q^9}{q^{27}}}{\jacprod{q^6}{q^{27}}}$
we find it is equivalent to
\begin{align*}
	&
	-1
	+\frac{\GEta{27}{3}{\tau}^{2}}
		{\GEta{27}{12}{\tau}^{2}}
	-\frac{\GEta{27}{3}{\tau}^{2}}
		{\GEta{27}{6}{\tau}\GEta{27}{12}{\tau}}
	+\frac{\GEta{27}{3}{\tau}^{2}\GEta{27}{12}{\tau}}
		{\GEta{27}{6}{\tau}^{3}}
	+\frac{1}{3}\frac{\GEta{27}{9}{\tau}^{3}}
		{\GEta{27}{6}{\tau}^{3}}
	-\frac{1}{3}\frac{\GEta{27}{3}{\tau}\GEta{27}{9}{\tau}^{3}}
		{\GEta{27}{6}{\tau}^{2}\GEta{27}{12}{\tau}^{2}}
	\\
	&=
	\frac{1}{3}\frac{\GEta{3}{1}{\tau}\GEta{27}{3}{\tau}\GEta{27}{9}{\tau}\GEta{27}{12}{\tau}^{2}}
		{\GEta{27}{6}{\tau}}
	+\frac{2}{3}\GEta{3}{1}{\tau}\GEta{27}{3}{\tau}^{2}\GEta{27}{9}{\tau}
	+\frac{1}{3}\GEta{3}{1}{\tau}\GEta{27}{3}{\tau}\GEta{27}{9}{\tau}\GEta{27}{12}{\tau}
	\\&\quad
	+\frac{1}{3}\frac{\GEta{3}{1}{\tau}\GEta{27}{3}{\tau}^{3}\GEta{27}{9}{\tau}}
		{\GEta{27}{6}{\tau}}
	+\frac{1}{3}\frac{\GEta{3}{1}{\tau}\GEta{27}{3}{\tau}^{2}\GEta{27}{9}{\tau}\GEta{27}{12}{\tau}}
		{\GEta{27}{6}{\tau}}
	\\&\quad
	-\frac{2}{3}\GEta{3}{1}{\tau}\GEta{27}{3}{\tau}\GEta{27}{6}{\tau}\GEta{27}{9}{\tau}
.
\end{align*}
Each term is a modular function with respect to $\Gamma_1(27)$ and by the 
reasoning explained in Section 2 by valence formula we need only verify the 
identity in the $q$-expansion past $q^{25}$. Thus Theorem \ref{TheoremPP13} holds.


\section{Proof of Theorem \ref{TheoremPP23} }

We set
\begin{align*}
	A &= \frac{1}{3} + V_3^6(1) -qU_3^6(5) -V_3^6(4) +q^2U_2^6(8)
.
\end{align*}
By (\ref{PP2Zeta3}) the left hand side of Theorem \ref{TheoremPP23}
is $\frac{A}{\aqprod{q}{q}{\infty}}$,
so we must show
\begin{align}\label{EqPP2Mod3Eq1}
	\frac{A}{\aqprod{q}{q}{\infty}}
	&=
	\frac{\aqprod{q^{27}}{q^{27}}{\infty}^2}{\aqprod{q^3}{q^3}{\infty}}
	\left(
		-\frac{2}{3}\jacprod{q^{12}}{q^{27}}^2 
		-\frac{1}{3}q^3\jacprod{q^3,q^6}{q^{27}} 	
		-\frac{2}{3}q\jacprod{q^6,q^{12}}{q^{27}} 
		+\frac{1}{3} q^4\jacprod{q^3}{q^{27}}^2
		\right.\nonumber\\&\left.\quad
		+\frac{1}{3}q^2\jacprod{q^3,q^{12}}{q^{27}} 
		+\frac{1}{3}q^2\jacprod{q^6}{q^{27}}^2
	\right)
	\nonumber\\&\quad
	+\frac{q^{-3}}{\aqprod{q^{27}}{q^{27}}{\infty}\jacprod{q^{6}}{q^{27}}}
		\left(\SSeries{-9}{-66}{27}-q^{66}\SSeries{24}{66}{27}\right)
	\nonumber\\&\quad
	+\frac{q^{8}}{\aqprod{q^{27}}{q^{27}}{\infty}\jacprod{q^{3}}{q^{27}}}
		\left(\SSeries{9}{-6}{27}-q^{6}\SSeries{12}{6}{27}\right)
.
\end{align}

By Lemmas \ref{LemmaFor6b1} and \ref{LemmaFor6b4}
\begin{align*}
	A
	&=
	\frac{1}{3} - g(q^{15},q^{27}) - g(q^{21},q^{27})
	+q^{2}\frac{\aqprod{q^{27}}{q^{27}}{\infty}^2\jacprod{q^{12}}{q^{27}}}
		{\jacprod{q^{6},q^{9}}{q^{27}}}
	-\frac{\aqprod{q^{27}}{q^{27}}{\infty}^2\jacprod{q^{6},q^{12}}{q^{27}}}
		{\jacprod{q^{3},q^{3},q^{9}}{q^{27}}}
	\\&\quad
	+q^{2}\frac{\aqprod{q^{27}}{q^{27}}{\infty}^2\jacprod{q^{6}}{q^{27}}}
		{\jacprod{q^{3},q^{9}}{q^{27}}}
	-q^{6}\frac{\aqprod{q^{27}}{q^{27}}{\infty}^2\jacprod{q^{3},q^{6}}{q^{27}}}
		{\jacprod{q^{9},q^{12},q^{12}}{q^{27}}}
	\\&\quad
	+q^{-3}\frac{\aqprod{q}{q}{\infty}}{\aqprod{q^{27}}{q^{27}}{\infty}\jacprod{q^{6}}{q^{27}}}
		\left(\SSeries{-9}{-66}{27}-q^{66}\SSeries{24}{66}{27}\right)
	\\&\quad
	+q^{8}\frac{\aqprod{q}{q}{\infty}}{\aqprod{q^{27}}{q^{27}}{\infty}\jacprod{q^{3}}{q^{27}}}
		\left(\SSeries{9}{-6}{27}-q^{6}\SSeries{12}{6}{27}\right)
.
\end{align*}

By Lemma \ref{CorollaryGProducts} we have
\begin{align*}
	1-3g(q^{24},q^{27})-3g(q^{21},q^{27})
	&=
	-3 + 2g(q^6,q^{27}) + g(q^{12},q^{27}) - g(q^{24},q^{27})
		\\&\quad
		-3 + 2g(q^{12}q^{27}) + g(q^{24},q^{27}) - g(q^{48},q^{27})	
	\\
	&=
	q^{3}\frac{\aqprod{q^{27}}{q^{27}}{\infty}^2\jacprod{q^9}{q^{27}}^2}
	{\jacprod{q^{12}}{q^{27}}^2\jacprod{q^{6}}{q^{27}}}
	+
	\frac{\aqprod{q^{27}}{q^{27}}{\infty}^2\jacprod{q^{9}}{q^{27}}^2}
	{\jacprod{q^3}{q^{27}}^2\jacprod{q^{12}}{q^{27}}}
.
\end{align*}

We now have only an identity to verify between products. Multiplying 
(\ref{EqPP2Mod3Eq1}) by
$q^{-2}\frac{\aqprod{q}{q}{\infty}\jacprod{q^6,q^9}{q^{27}}}
{\aqprod{q^{27}}{q^{27}}{\infty}^2\jacprod{q^{12}}{q^{27}}}
$
we find we must show
\begin{align*}
	&
	1
	-\frac{\GEta{27}{6}{\tau}^{2}}
		{\GEta{27}{3}{\tau}^{2}}
	+\frac{\GEta{27}{6}{\tau}^{2}}
		{\GEta{27}{3}{\tau}\GEta{27}{12}{\tau}}
	-\frac{\GEta{27}{3}{\tau}\GEta{27}{6}{\tau}^{2}}
		{\GEta{27}{12}{\tau}^{3}}
	+\frac{1}{3}\frac{\GEta{27}{9}{\tau}^{3}}
		{\GEta{27}{12}{\tau}^{3}}
	+\frac{1}{3}\frac{\GEta{27}{6}{\tau}\GEta{27}{9}{\tau}^{3}}
		{\GEta{27}{3}{\tau}^{2}\GEta{27}{12}{\tau}^{2}}
	\\
	&=
	-\frac{2}{3}\GEta{3}{1}{\tau}\GEta{27}{6}{\tau}\GEta{27}{9}{\tau}\GEta{27}{12}{\tau}
	-\frac{1}{3}\frac{\GEta{3}{1}{\tau}\GEta{27}{3}{\tau}\GEta{27}{6}{\tau}^{2}\GEta{27}{9}{\tau}}
		{\GEta{27}{12}{\tau}}
	-\frac{2}{3}\GEta{3}{1}{\tau}\GEta{27}{6}{\tau}^{2}\GEta{27}{9}{\tau}
	\\&\quad
	+\frac{1}{3}\frac{\GEta{3}{1}{\tau}\GEta{27}{3}{\tau}^{2}\GEta{27}{6}{\tau}\GEta{27}{9}{\tau}}
		{\GEta{27}{12}{\tau}}
	+\frac{1}{3}\GEta{3}{1}{\tau}\GEta{27}{3}{\tau}\GEta{27}{6}{\tau}\GEta{27}{9}{\tau}
	+\frac{1}{3}\frac{\GEta{3}{1}{\tau}\GEta{27}{6}{\tau}^{3}\GEta{27}{9}{\tau}}
		{\GEta{27}{12}{\tau}}
.
\end{align*}
Again each term is a modular function with respect to $\Gamma_1(27)$ and again
it is sufficient to verify the identity in the $q$-expansion past $q^{25}$.
Thus Theorem \ref{TheoremPP23} holds.


\section{Proof of Theorem \ref{TheoremPP25} }

First we set
\begin{align*}
	A 	&= \frac{3}{5} + V_5^6(1) - qU_5^6(5) - V_5^6(10) - q^4U_5^6(14)
	,\\
	B &= \frac{1}{5} + V_5^6(4) - q^2U_5^6(8) - V_5^6(7) + q^3U_5^7(11)
.
\end{align*}
By (\ref{PP2Zeta5}) the left hand side of Theorem \ref{TheoremPP25}
is $\frac{A+(\zeta_5+\zeta_5^4)B}{\aqprod{q}{q}{\infty}}$,
so we must show
\begin{align}\label{EqPP2Mod5Eq1}
	\frac{A}{\aqprod{q}{q}{\infty}}
	&=
	\aqprod{q^{25}}{q^{25}}{\infty}
	\left(
		\frac{3}{5}\frac{\jacprod{q^{10}}{q^{25}}}{\jacprod{q^5}{q^{25}}^2}
		-\frac{\jacprod{q^{10}}{q^{25}}}{\jacprod{q^{10},q^{15},q^{25}}{q^{75}}\jacprod{q^5}{q^{25}}}
		-\frac{2}{5}q\frac{1}{\jacprod{q^5}{q^{25}}}
		+\frac{1}{5}\frac{q^2}{\jacprod{q^{10}}{q^{25}}}
		\right.\nonumber\\&\quad\left.		
		-\frac{1}{5}q^3\frac{\jacprod{q^{5}}{q^{25}}}{\jacprod{q^{10}}{q^{25}}^2}
		+q^4\frac{1}{\jacprod{q^{15},q^{20},q^{25}}{q^{75}}}
	\right)
	\nonumber\\&\quad
	-\frac{q^{-15}}
		{\aqprod{q^{75}}{q^{75}}{\infty}\jacprod{q^5}{q^{75}}}
		\left(\SSeries{-15}{-140}{75}-q^{140}\SSeries{55}{140}{75}\right)
	\nonumber\\&\quad
	+\frac{q^{14}}
		{\aqprod{q^{75}}{q^{75}}{\infty}\jacprod{q^{20}}{q^{75}}}
	\left(\SSeries{15}{-40}{75}-q^{40}\SSeries{35}{40}{75}\right)
	,\\
	\label{EqPP2Mod5Eq2}
	\frac{B}{\aqprod{q}{q}{\infty}}
	&=
	\aqprod{q^{25}}{q^{25}}{\infty}
	\left(
		\frac{1}{5}\frac{\jacprod{q^{10}}{q^{25}}}{\jacprod{q^5}{q^{25}}^2}
		-q^{10}\frac{\jacprod{q^{5}}{q^{75}}}{\jacprod{q^{15},q^{20},q^{25},q^{25}}{q^{75}}}
		+\frac{1}{5}q\frac{1}{\jacprod{q^5}{q^{25}}}
		-\frac{3}{5}\frac{q^2}{\jacprod{q^{10}}{q^{25}}}
		\right.\nonumber\\&\quad\left.		
		+\frac{3}{5}q^3\frac{\jacprod{q^{5}}{q^{25}}}{\jacprod{q^{10}}{q^{25}}^2}
		-q^{-2}\frac{\jacprod{q^{20},q^{30}}{q^{75}}}{\jacprod{q^5,q^{15},q^{15},q^{25},q^{35}}{q^{75}}}
	\right)
	\nonumber\\&\quad
	+\frac{q^{10}}{\aqprod{q^{75}}{q^{75}}{\infty}\jacprod{q^{35}}{q^{75}}}
		\left(\SSeries{15}{-70}{75}-q^{70}\SSeries{50}{70}{75}\right)
	\nonumber\\&\quad
	+\frac{q^{-7}}{\aqprod{q^{75}}{q^{75}}{\infty}\jacprod{q^{10}}{q^{75}}}
		\left(\SSeries{-15}{-170}{75}-q^{170}\SSeries{70}{170}{75}\right)
.
\end{align}

By Lemmas \ref{LemmaFor6b1} and \ref{LemmaFor6b4}
\begin{align*}
	A 
	&= 
	\frac{3}{5}
	-g(q^{40},q^{75})
	-g(q^{50},q^{75})
	-q^{6}\frac{\aqprod{q^{75}}{q^{75}}{\infty}^2\jacprod{q^{10},q^{25}}{q^{75}}}
		{\jacprod{q^{5},q^{20},q^{30}}{q^{75}}}
	+q^{2}\frac{\aqprod{q^{75}}{q^{75}}{\infty}^2\jacprod{q^{25}}{q^{75}}}
		{\jacprod{q^{10},q^{15}}{q^{75}}}
	\\&\quad
	-q^{5}\frac{\aqprod{q^{75}}{q^{75}}{\infty}^2\jacprod{q^{10},q^{30},q^{35}}{q^{75}}}
		{\jacprod{q^{15},q^{15},q^{20},q^{25}}{q^{75}}}
	-\frac{\aqprod{q^{75}}{q^{75}}{\infty}^2\jacprod{q^{25}}{q^{75}}}
		{\jacprod{q^{5},q^{15}}{q^{75}}}
	-q^{19}\frac{\aqprod{q^{75}}{q^{75}}{\infty}^2\jacprod{q^{5},q^{10}}{q^{75}}}
		{\jacprod{q^{20},q^{30},q^{35}}{q^{75}}}
	\\&\quad
	+q^{4}\frac{\aqprod{q^{75}}{q^{75}}{\infty}^2\jacprod{q^{10}}{q^{75}}}
		{\jacprod{q^{5},q^{15}}{q^{75}}}
	-q^{6}\frac{\aqprod{q^{75}}{q^{75}}{\infty}^2\jacprod{q^{5},q^{25},q^{30}}{q^{75}}}
		{\jacprod{q^{10},q^{15},q^{15},q^{35}}{q^{75}}}
	+q^{15}\frac{\aqprod{q^{75}}{q^{75}}{\infty}^2\jacprod{q^{10},q^{10}}{q^{75}}}
		{\jacprod{q^{15},q^{25},q^{35}}{q^{75}}}
	\nonumber\\&\quad
	-q^{-15}\frac{\aqprod{q}{q}{\infty}}
		{\aqprod{q^{75}}{q^{75}}{\infty}\jacprod{q^5}{q^{75}}}
		\left(\SSeries{-15}{-140}{75}-q^{140}\SSeries{55}{140}{75}\right)
	\nonumber\\&\quad
	+q^{14}\frac{\aqprod{q}{q}{\infty}}
		{\aqprod{q^{75}}{q^{75}}{\infty}\jacprod{q^{20}}{q^{75}}}
	\left(\SSeries{15}{-40}{75}-q^{40}\SSeries{35}{40}{75}\right)
.
\end{align*}

By expanding 
$g(q^{40},q^{75})$ 
into products by (\ref{EqGToProducts}),
replacing $g(q^{50},q^{75})$ with $V_{\chi,1}(25\tau)$,
and then
multiplying by 
{\allowbreak$q^{-6}\frac{\aqprod{q}{q}{\infty}\jacprod{q^5,q^{20},q^{30}}{q^{75}}}
{\aqprod{q^{75}}{q^{75}}{\infty}^2\jacprod{q^{10},q^{15}}{q^{75}}}$},
we find (\ref{EqPP2Mod5Eq1}) to be equivalent to
\begin{align*}
	&\quad
	-1
	+\frac{\GEta{75}{5}{\tau}\GEta{75}{20}{\tau}\GEta{75}{30}{\tau}}
		{\GEta{75}{10}{\tau}^{2}\GEta{75}{15}{\tau}}
	-\frac{\GEta{75}{5}{\tau}\GEta{75}{30}{\tau}^{2}\GEta{75}{35}{\tau}}
		{\GEta{75}{15}{\tau}^{2}\GEta{75}{25}{\tau}^{2}}
	-\frac{\GEta{75}{20}{\tau}\GEta{75}{30}{\tau}}
		{\GEta{75}{10}{\tau}\GEta{75}{15}{\tau}}
	\\&\quad
	-\frac{\GEta{75}{5}{\tau}^{2}}
		{\GEta{75}{25}{\tau}\GEta{75}{35}{\tau}}
	+\frac{\GEta{75}{20}{\tau}\GEta{75}{30}{\tau}}
		{\GEta{75}{15}{\tau}\GEta{75}{25}{\tau}}
	-\frac{\GEta{75}{5}{\tau}^{2}\GEta{75}{20}{\tau}\GEta{75}{30}{\tau}^{2}}
		{\GEta{75}{10}{\tau}^{2}\GEta{75}{15}{\tau}^{2}\GEta{75}{35}{\tau}}
	\\&\quad
	+\frac{\GEta{75}{5}{\tau}\GEta{75}{10}{\tau}\GEta{75}{20}{\tau}\GEta{75}{30}{\tau}}
		{\GEta{75}{15}{\tau}\GEta{75}{25}{\tau}^{2}\GEta{75}{35}{\tau}}
	+\frac{4}{15}\frac{\GEta{75}{5}{\tau}^{2}\GEta{75}{20}{\tau}^{2}\GEta{75}{30}{\tau}}
		{\GEta{75}{10}{\tau}^{3}\GEta{75}{15}{\tau}\GEta{75}{25}{\tau}}
	+\frac{2}{15}\frac{\GEta{75}{5}{\tau}\GEta{75}{35}{\tau}}
		{\GEta{75}{20}{\tau}\GEta{75}{25}{\tau}}
	\\&\quad
	-\frac{1}{15}\frac{\GEta{75}{5}{\tau}^{2}\GEta{75}{20}{\tau}^{2}\GEta{75}{30}{\tau}}
		{\GEta{75}{10}{\tau}\GEta{75}{15}{\tau}\GEta{75}{25}{\tau}\GEta{75}{35}{\tau}^{2}}
	-\frac{1}{30}\frac{\GEta{75}{20}{\tau}\GEta{75}{35}{\tau}}
		{\GEta{75}{5}{\tau}\GEta{75}{25}{\tau}}
	+\frac{4}{15}\frac{\GEta{75}{10}{\tau}^{2}\GEta{75}{20}{\tau}\GEta{75}{30}{\tau}}
		{\GEta{75}{5}{\tau}^{2}\GEta{75}{15}{\tau}\GEta{75}{25}{\tau}}
	\\&\quad
	+\frac{2}{15}\frac{\GEta{75}{5}{\tau}\GEta{75}{20}{\tau}^{4}}
		{\GEta{75}{10}{\tau}^{4}\GEta{75}{25}{\tau}}
	+\frac{1}{15}\frac{\GEta{75}{5}{\tau}\GEta{75}{30}{\tau}\GEta{75}{35}{\tau}^{3}}
		{\GEta{75}{10}{\tau}\GEta{75}{15}{\tau}\GEta{75}{20}{\tau}^{2}\GEta{75}{25}{\tau}}
	+\frac{1}{30}\frac{\GEta{75}{5}{\tau}^{4}\GEta{75}{20}{\tau}}
		{\GEta{75}{10}{\tau}\GEta{75}{25}{\tau}\GEta{75}{35}{\tau}^{3}}
	\\&\quad
	+\frac{\GEta{75}{5}{\tau}\GEta{75}{20}{\tau}\GEta{75}{30}{\tau}}
		{\GEta{75}{10}{\tau}\GEta{75}{25}{\tau}\GEta{75}{0}{\tau}}
		V_{\chi,1}(25\tau)
	\\
	&=
	\frac{3}{5}\frac{\GEta{1}{0}{\tau}^{1/2}\GEta{75}{5}{\tau}\GEta{75}{20}{\tau}\GEta{75}{30}{\tau}\GEta{25}{10}{\tau}}
			{\GEta{75}{0}{\tau}^{1/2}\GEta{75}{10}{\tau}\GEta{25}{5}{\tau}^2}	
	-\frac{\GEta{1}{0}{\tau}^{1/2}\GEta{75}{5}{\tau}\GEta{75}{20}{\tau}\GEta{75}{30}{\tau}\GEta{25}{10}{\tau}}
			{\GEta{75}{0}{\tau}^{1/2}\GEta{75}{10}{\tau}^2\GEta{75}{15}{\tau}\GEta{75}{25}{\tau}\GEta{25}{5}{\tau}}
	\\&\quad
	-\frac{2}{5}\frac{\GEta{1}{0}{\tau}^{1/2}\GEta{75}{5}{\tau}\GEta{75}{20}{\tau}\GEta{75}{30}{\tau}}
			{\GEta{75}{0}{\tau}^{1/2}\GEta{75}{10}{\tau}\GEta{25}{5}{\tau}}
	+\frac{1}{5}\frac{\GEta{1}{0}{\tau}^{1/2}\GEta{75}{5}{\tau}\GEta{75}{20}{\tau}\GEta{75}{30}{\tau}}
			{\GEta{75}{0}{\tau}^{1/2}\GEta{75}{10}{\tau}\GEta{25}{10}{\tau}}
	\\&\quad
	-\frac{1}{5}\frac{\GEta{1}{0}{\tau}^{1/2}\GEta{75}{5}{\tau}\GEta{75}{20}{\tau}\GEta{75}{30}{\tau}\GEta{25}{5}{\tau}}
			{\GEta{75}{0}{\tau}^{1/2}\GEta{75}{10}{\tau}\GEta{25}{10}{\tau}^2}
	+\frac{\GEta{1}{0}{\tau}^{1/2}\GEta{75}{5}{\tau}\GEta{75}{20}{\tau}\GEta{75}{30}{\tau}}
			{\GEta{75}{0}{\tau}^{1/2}\GEta{75}{10}{\tau}\GEta{75}{15}{\tau}\GEta{75}{20}{\tau}\GEta{75}{25}{\tau}}
.
\end{align*}
However each term is a modular function with respect to $\Gamma_1(75)$ and by 
the valence formula it is sufficient to verify the identity in the 
$q$-expansion past $q^{199}$. For the term with $V_{\chi,1}(25\tau)$ we use that
\begin{align*}
	&\frac{\GEta{75}{5}{\tau}\GEta{75}{20}{\tau}\GEta{75}{30}{\tau}}
		{\GEta{75}{10}{\tau}\GEta{75}{25}{\tau}\GEta{75}{0}{\tau}}
	V_{\chi,1}(25\tau)
	\\
	&=
	\frac{\GEta{75}{5}{\tau}\GEta{75}{20}{\tau}\GEta{75}{30}{\tau}\GEta{1}{0}{\tau}^{1/2}\GEta{3}{0}{\tau}^{1/2}}
		{\GEta{75}{10}{\tau}\GEta{75}{25}{\tau}\GEta{5}{0}{\tau}}
	\cdot
	\frac{\eta(5\tau)^2}{\eta(\tau)\eta(3\tau)\eta(75\tau)^2}
	\cdot
	V_{\chi,1}(25\tau)
,
\end{align*}
to see a product of a modular function, a meromorphic modular form of weight 
$-1$, and a holomorphic modular form of weight $1$ all with respect to
$\Gamma_1(75)$. We can compute the order of the eta quotient at the cusps
of $\Gamma_1(75)$ based on the formulas for the order at the cusps of 
$\Gamma_0(75)$ (Theorem 1.65 of \cite{Ono1}) and using
\begin{align*}
	Ord_{\Gamma_1(75)} (f;\zeta)
	&=
	ord(f;\zeta)N(\Gamma_1(75),\zeta) 
	= \frac{ Ord_{\Gamma_0(75)} (f;\zeta ')N(\Gamma_1(75),\zeta)}
		{N(\Gamma_0(75),\zeta ')}
,
\end{align*}
where $\zeta$ is a cusp of $\Gamma_1(75)$ and $\zeta '$ is a cusp of 
$\Gamma_0(75)$ that is $\Gamma_0(75)$-equivalent to $\zeta$. We can ignore the
contribution of $V_{\chi,1}(25\tau)$ since it can only possibly increase the
order at a cusp.

By Lemmas \ref{LemmaFor6b1} and \ref{LemmaFor6b4}
\begin{align*}
	B &=
	\frac{1}{5}
	+ g(q^{65},q^{75})  + g(q^{55},q^{75})
	-q^{17}\frac{\aqprod{q^{75}}{q^{75}}{\infty}^2\jacprod{q^{5},q^{20}}{q^{75}}}
		{\jacprod{q^{25},q^{30},q^{35}}{q^{75}}}
	+q^{11}\frac{\aqprod{q^{75}}{q^{75}}{\infty}^2\jacprod{q^{5}}{q^{75}}}
		{\jacprod{q^{15},q^{20}}{q^{75}}}
	\\&\quad
	-q^{2}\frac{\aqprod{q^{75}}{q^{75}}{\infty}^2\jacprod{q^{10},q^{20},q^{30}}{q^{75}}}
		{\jacprod{q^{5},q^{15},q^{15},q^{25}}{q^{75}}}
	+q^{10}\frac{\aqprod{q^{75}}{q^{75}}{\infty}^2\jacprod{q^{5},q^{35}}{q^{75}}}
		{\jacprod{q^{15},q^{20},q^{25}}{q^{75}}}
	-q^{-2}\frac{\aqprod{q^{75}}{q^{75}}{\infty}^2\jacprod{q^{25},q^{35}}{q^{75}}}
		{\jacprod{q^{5},q^{10},q^{30}}{q^{75}}}
	\\&\quad
	-q^{5}\frac{\aqprod{q^{75}}{q^{75}}{\infty}^2\jacprod{q^{35}}{q^{75}}}
		{\jacprod{q^{15},q^{25}}{q^{75}}}
	+q^{-1}\frac{\aqprod{q^{75}}{q^{75}}{\infty}^2\jacprod{q^{20},q^{25},q^{30}}{q^{75}}}
		{\jacprod{q^{5},q^{15},q^{15},q^{35}}{q^{75}}}
	+\frac{\aqprod{q^{75}}{q^{75}}{\infty}^2\jacprod{q^{25}}{q^{75}}}
		{\jacprod{q^{5},q^{15}}{q^{75}}}
	\\&\quad
	+q^{10}\frac{\aqprod{q}{q}{\infty}}{\aqprod{q^{75}}{q^{75}}{\infty}\jacprod{q^{35}}{q^{75}}}
		\left(\SSeries{15}{-70}{75}-q^{70}\SSeries{50}{70}{75}\right)
	\\&\quad
	+q^{-7}\frac{\aqprod{q}{q}{\infty}}{\aqprod{q^{75}}{q^{75}}{\infty}\jacprod{q^{10}}{q^{75}}}
		\left(\SSeries{-15}{-170}{75}-q^{170}\SSeries{70}{170}{75}\right)
.	
\end{align*}

\sloppy
By expanding 
$g(q^{65},q^{75})$  and $g(q^{55},q^{75})$
into products by (\ref{EqGToProducts}) and then
multiplying by 
$
q^{-17}\frac{\aqprod{q}{q}{\infty}\jacprod{q^{25},q^{30},q^{35}}{q^{75}}}
{\aqprod{q^{75}}{q^{75}}\infty{}^2\jacprod{q^5,q^{20}}{q^{75}}}
$
we find (\ref{EqPP2Mod5Eq2}) to be equivalent to
\begin{align*}
	&\quad
	-1
	+\frac{\GEta{75}{25}{\tau}\GEta{75}{30}{\tau}\GEta{75}{35}{\tau}}
		{\GEta{75}{15}{\tau}\GEta{75}{20}{\tau}^{2}}
	-\frac{\GEta{75}{10}{\tau}\GEta{75}{30}{\tau}^{2}\GEta{75}{35}{\tau}}
		{\GEta{75}{5}{\tau}^{2}\GEta{75}{15}{\tau}^{2}}
	+\frac{\GEta{75}{30}{\tau}\GEta{75}{35}{\tau}^{2}}
		{\GEta{75}{15}{\tau}\GEta{75}{20}{\tau}^{2}}
	\\&\quad
	-\frac{\GEta{75}{25}{\tau}^{2}\GEta{75}{35}{\tau}^{2}}
		{\GEta{75}{5}{\tau}^{2}\GEta{75}{10}{\tau}\GEta{75}{20}{\tau}}
	-\frac{\GEta{75}{30}{\tau}\GEta{75}{35}{\tau}^{2}}
		{\GEta{75}{5}{\tau}\GEta{75}{15}{\tau}\GEta{75}{20}{\tau}}
	+\frac{\GEta{75}{25}{\tau}^{2}\GEta{75}{30}{\tau}^{2}}
		{\GEta{75}{5}{\tau}^{2}\GEta{75}{15}{\tau}^{2}}
	\\&\quad
	+\frac{\GEta{75}{25}{\tau}^{2}\GEta{75}{30}{\tau}\GEta{75}{35}{\tau}}
		{\GEta{75}{5}{\tau}^{2}\GEta{75}{15}{\tau}\GEta{75}{20}{\tau}}
	-\frac{3}{10}\frac{\GEta{75}{25}{\tau}\GEta{75}{30}{\tau}}
		{\GEta{75}{15}{\tau}\GEta{75}{35}{\tau}}
	-\frac{2}{5}\frac{\GEta{75}{10}{\tau}\GEta{75}{25}{\tau}\GEta{75}{35}{\tau}^{2}}
		{\GEta{75}{5}{\tau}^{3}\GEta{75}{20}{\tau}}
	\\&\quad
	+\frac{1}{5}\frac{\GEta{75}{25}{\tau}\GEta{75}{30}{\tau}\GEta{75}{35}{\tau}}
		{\GEta{75}{10}{\tau}^{2}\GEta{75}{15}{\tau}}
	+\frac{1}{10}\frac{\GEta{75}{10}{\tau}\GEta{75}{25}{\tau}\GEta{75}{35}{\tau}^{2}}
		{\GEta{75}{5}{\tau}\GEta{75}{20}{\tau}^{3}}
	+\frac{3}{10}\frac{\GEta{75}{25}{\tau}\GEta{75}{30}{\tau}\GEta{75}{35}{\tau}^{4}}
		{\GEta{75}{5}{\tau}\GEta{75}{15}{\tau}\GEta{75}{20}{\tau}^{4}}
	\\&\quad
	+\frac{2}{5}\frac{\GEta{75}{5}{\tau}^{2}\GEta{75}{25}{\tau}}
		{\GEta{75}{20}{\tau}\GEta{75}{35}{\tau}^{2}}
	+\frac{1}{5}\frac{\GEta{75}{10}{\tau}^{3}\GEta{75}{25}{\tau}\GEta{75}{30}{\tau}\GEta{75}{35}{\tau}}
		{\GEta{75}{5}{\tau}^{4}\GEta{75}{15}{\tau}\GEta{75}{20}{\tau}}
	+\frac{1}{10}\frac{\GEta{75}{20}{\tau}^{2}\GEta{75}{25}{\tau}\GEta{75}{35}{\tau}}
		{\GEta{75}{5}{\tau}\GEta{75}{10}{\tau}^{3}}
	\\
	&=
	\frac{1}{5}\frac{\GEta{1}{0}{\tau}^{1/2}\GEta{75}{25}{\tau}^2\GEta{75}{30}{\tau}\GEta{75}{35}{\tau}\GEta{25}{10}{\tau}}
			{\GEta{75}{0}{\tau}^{1/2}\GEta{75}{5}{\tau}\GEta{75}{20}{\tau}\GEta{25}{5}{\tau}^2}	
	-\frac{\GEta{1}{0}{\tau}^{1/2}\GEta{75}{30}{\tau}\GEta{75}{35}{\tau}}
			{\GEta{75}{0}{\tau}^{1/2}\GEta{75}{15}{\tau}\GEta{75}{20}{\tau}^2}
	\\&\quad
	+\frac{1}{5}\frac{\GEta{1}{0}{\tau}^{1/2}\GEta{75}{25}{\tau}^2\GEta{75}{30}{\tau}\GEta{75}{35}{\tau}}
			{\GEta{75}{0}{\tau}^{1/2}\GEta{75}{5}{\tau}\GEta{75}{20}{\tau}\GEta{25}{5}{\tau}}
	-\frac{3}{5}\frac{\GEta{1}{0}{\tau}^{1/2}\GEta{75}{25}{\tau}^2\GEta{75}{30}{\tau}\GEta{75}{35}{\tau}}
			{\GEta{75}{0}{\tau}^{1/2}\GEta{75}{5}{\tau}\GEta{75}{20}{\tau}\GEta{25}{10}{\tau}}
	\\&\quad
	+\frac{3}{5}\frac{\GEta{1}{0}{\tau}^{1/2}\GEta{75}{25}{\tau}^2\GEta{75}{30}{\tau}\GEta{75}{35}{\tau}\GEta{25}{5}{\tau}}
			{\GEta{75}{0}{\tau}^{1/2}\GEta{75}{5}{\tau}\GEta{75}{20}{\tau}\GEta{25}{10}{\tau}^2}
	-\frac{\GEta{1}{0}{\tau}^{1/2}\GEta{75}{25}{\tau}\GEta{75}{30}{\tau}^2}
			{\GEta{75}{0}{\tau}^{1/2}\GEta{75}{5}{\tau}^2\GEta{75}{15}{\tau}^2}
.
\end{align*}
However each term is a modular function with respect to $\Gamma_1(75)$ and by
the valence formula we find it is sufficient to verify the identity in the
$q$-expansion past $q^{201}$. Thus Theorem \ref{TheoremPP25} holds.
\fussy

\section{Proof of Theorem \ref{TheoremPP35}}

First we set
\begin{align*}
	A &= \frac{3}{5} + V(2) - V(11)-qU(4)+q^4U(13)
	,\\
	B	&=	\frac{1}{5} +V(5) - V(8) - q^2U(7) + q^3U(10)
.
\end{align*}
By (\ref{PP3Zeta5}) the left hand side of Theorem \ref{TheoremPP35}
is $\frac{A+(\zeta_5+\zeta_5^4)B}{\aqprod{q}{q}{\infty}}$,
so we must show
\begin{align}\label{EqPP3Mod5Eq1}
	\frac{A}{\aqprod{q}{q}{\infty}}
	&=
	\aqprod{q^{25}}{q^{25}}{\infty}
	\left(
		\frac{3}{5}\frac{\jacprod{q^{10}}{q^{25}}}{\jacprod{q^5}{q^{25}}^2}
		-q^5\frac{\jacprod{q^{10}}{q^{75}}}{\jacprod{q^5,q^{20},q^{25},q^{30}}{q^{75}}}
		-\frac{2}{5}\frac{q}{\jacprod{q^5}{q^{25}}}
		+\frac{1}{5}\frac{q^2}{\jacprod{q^{10}}{q^{25}}}
		\right.\nonumber\\&\quad\left.		
		-\frac{1}{5}q^3\frac{\jacprod{q^{5}}{q^{25}}}{\jacprod{q^{10}}{q^{25}}^2}
		-q^8\frac{\jacprod{q^{5}}{q^{75}}}{\jacprod{q^{10},q^{15},q^{25},q^{35}}{q^{75}}}
	\right)
	,\\
	\label{EqPP3Mod5Eq2}
	\frac{B}{\aqprod{q}{q}{\infty}}
	&= 
	\aqprod{q^{25}}{q^{25}}{\infty}
	\left(
		\frac{1}{5}\frac{\jacprod{q^{10}}{q^{25}}}{\jacprod{q^5}{q^{25}}^2}
		+\frac{1}{5}q\frac{1}{\jacprod{q^5}{q^{25}}}
		-q^6\frac{1}{\jacprod{q^{15},q^{25},q^{35}}{q^{75}}}
		+\frac{2}{5}\frac{q^2}{\jacprod{q^{10}}{q^{25}}}
		\right.\nonumber\\&\quad\left.		
		-q^2\frac{1}{\jacprod{q^{5},q^{25},q^{30}}{q^{75}}}
		+\frac{3}{5}q^3\frac{\jacprod{q^{5}}{q^{25}}}{\jacprod{q^{10}}{q^{25}}^2}
		+q^8\frac{\jacprod{q^{5}}{q^{75}}}{\jacprod{q^{10},q^{15},q^{25},q^{35}}{q^{75}}}
	\right)
.
\end{align}

By Lemma \ref{LemmaFor3b2}
\begin{align*}
	A
	&=
	\frac{3}{5} + h(q^{65},q^{75}) - h(q^{20},q^{75})
	+q^{5}\frac{\aqprod{q^{75}}{q^{75}}{\infty}^2\jacprod{q^{35}}{q^{75}}}
		{\jacprod{q^{15},q^{25}}{q^{75}}}
	+q^{16}\frac{\aqprod{q^{75}}{q^{75}}{\infty}^2\jacprod{q^{5}}{q^{75}}}
		{\jacprod{q^{30},q^{35}}{q^{75}}}
	\\&\quad
	-q^{8}\frac{\aqprod{q^{75}}{q^{75}}{\infty}^2\jacprod{q^{25}}{q^{75}}}
		{\jacprod{q^{20},q^{30}}{q^{75}}}
	-q\frac{\aqprod{q^{75}}{q^{75}}{\infty}^2\jacprod{q^{20}}{q^{75}}}
		{\jacprod{q^{5},q^{15}}{q^{75}}}
	+q^{4}\frac{\aqprod{q^{75}}{q^{75}}{\infty}^2\jacprod{q^{10}}{q^{75}}}
		{\jacprod{q^{5},q^{15}}{q^{75}}}
	\\&\quad
	-q^{4}\frac{\aqprod{q^{75}}{q^{75}}{\infty}^2\jacprod{q^{35}}{q^{75}}}
		{\jacprod{q^{10},q^{30}}{q^{75}}}
	-q^{15}\frac{\aqprod{q^{75}}{q^{75}}{\infty}^2\jacprod{q^{5}}{q^{75}}}
		{\jacprod{q^{25},q^{30}}{q^{75}}}
	+q^{7}\frac{\aqprod{q^{75}}{q^{75}}{\infty}^2\jacprod{q^{25}}{q^{75}}}
		{\jacprod{q^{15},q^{35}}{q^{75}}}
.
\end{align*}

\sloppy
By expanding 
$h(q^{65},q^{75})$ and $h(q^{20},q^{75})$
into products by (\ref{EqHToProducts}) and then
multiplying (\ref{EqPP3Mod5Eq1}) by  
$q^{-1}\frac{\aqprod{q}{q}{\infty}\jacprod{q^5,q^{15}}{q^{75}}}
{\aqprod{q^{75}}{q^{75}}{\infty}^2\jacprod{q^{20}}{q^{75}}}$,
we find we are to show
\begin{align*} 
	&-1
		+\frac{\GEta{75}{5}{\tau}\GEta{75}{35}{\tau}}
			{\GEta{75}{20}{\tau}\GEta{75}{25}{\tau}}
		-\frac{\GEta{75}{5}{\tau}\GEta{75}{15}{\tau}\GEta{75}{25}{\tau}}
			{\GEta{75}{20}{\tau}^2\GEta{75}{30}{\tau}}
		+\frac{\GEta{75}{5}{\tau}^2\GEta{75}{15}{\tau}}
			{\GEta{75}{20}{\tau}\GEta{75}{30}{\tau}\GEta{75}{35}{\tau}}
		\\&\quad
		+\frac{\GEta{75}{10}{\tau}}
			{\GEta{75}{20}{\tau}}
		-\frac{\GEta{75}{5}{\tau}\GEta{75}{15}{\tau}\GEta{75}{35}{\tau}}
			{\GEta{75}{10}{\tau}\GEta{75}{20}{\tau}\GEta{75}{30}{\tau}}
		+\frac{\GEta{75}{5}{\tau}\GEta{75}{25}{\tau}}
			{\GEta{75}{20}{\tau}\GEta{75}{35}{\tau}}
		-\frac{\GEta{75}{5}{\tau}^2\GEta{75}{15}{\tau}}
			{\GEta{75}{20}{\tau}\GEta{75}{25}{\tau}\GEta{75}{30}{\tau}}
		\\&\quad
		+\frac{3}{10}\frac{\GEta{75}{5}{\tau}\GEta{75}{10}{\tau}\GEta{75}{15}{\tau}\GEta{75}{35}{\tau}}
			{\GEta{75}{20}{\tau}^3\GEta{75}{30}{\tau}}
		+\frac{3}{10}\frac{\GEta{75}{5}{\tau}\GEta{75}{20}{\tau}^2\GEta{75}{15}{\tau}}
			{\GEta{75}{10}{\tau}^3\GEta{75}{30}{\tau}}
		-\frac{2}{5}\frac{\GEta{75}{5}{\tau}^2}
			{\GEta{75}{35}{\tau}^2}
		\\&\quad
		+\frac{2}{5}\frac{\GEta{75}{5}{\tau}\GEta{75}{35}{\tau}^3}
			{\GEta{75}{20}{\tau}^4}
		-\frac{1}{5}\frac{\GEta{75}{10}{\tau}\GEta{75}{15}{\tau}\GEta{75}{35}{\tau}}
			{\GEta{75}{5}{\tau}\GEta{75}{20}{\tau}\GEta{75}{30}{\tau}}
		+\frac{1}{5}\frac{\GEta{75}{5}{\tau}^4\GEta{75}{15}{\tau}}
			{\GEta{75}{20}{\tau}\GEta{75}{35}{\tau}^3\GEta{75}{30}{\tau}}
		\\&\quad
		+\frac{1}{10}\frac{\GEta{75}{5}{\tau}^2}
			{\GEta{75}{10}{\tau}^2}
		+\frac{1}{10}\frac{\GEta{75}{10}{\tau}^3}
			{\GEta{75}{5}{\tau}^2\GEta{75}{20}{\tau}}
	\\
	&=
	\frac{3}{5}\frac{\GEta{1}{0}{\tau}^{1/2}\GEta{75}{5}{\tau}\GEta{75}{15}{\tau}\GEta{75}{25}{\tau}\GEta{25}{10}{\tau}}
			{\GEta{75}{0}{\tau}^{1/2}\GEta{75}{20}{\tau}\GEta{25}{5}{\tau}^2}	
	-\frac{\GEta{1}{0}{\tau}^{1/2}\GEta{75}{10}{\tau}\GEta{75}{15}{\tau}}
			{\GEta{75}{0}{\tau}^{1/2}\GEta{75}{20}{\tau}^2\GEta{75}{30}{\tau}}
	\\&\quad
	-\frac{2}{5}\frac{\GEta{1}{0}{\tau}^{1/2}\GEta{75}{5}{\tau}\GEta{75}{15}{\tau}\GEta{75}{25}{\tau}}
			{\GEta{75}{0}{\tau}^{1/2}\GEta{75}{20}{\tau}\GEta{25}{5}{\tau}}
	+\frac{1}{5}\frac{\GEta{1}{0}{\tau}^{1/2}\GEta{75}{5}{\tau}\GEta{75}{15}{\tau}\GEta{75}{25}{\tau}}
			{\GEta{75}{0}{\tau}^{1/2}\GEta{75}{20}{\tau}\GEta{25}{10}{\tau}}
	\\&\quad
	-\frac{1}{5}\frac{\GEta{1}{0}{\tau}^{1/2}\GEta{75}{5}{\tau}\GEta{75}{15}{\tau}\GEta{75}{25}{\tau}\GEta{25}{5}{\tau}}
			{\GEta{75}{0}{\tau}^{1/2}\GEta{75}{20}{\tau}\GEta{25}{10}{\tau}^2}
	-\frac{\GEta{1}{0}{\tau}^{1/2}\GEta{75}{5}{\tau}^2}
			{\GEta{75}{0}{\tau}^{1/2}\GEta{75}{10}{\tau}\GEta{75}{20}{\tau}\GEta{75}{35}{\tau}}
.
\end{align*}
Each term 
is a modular function with respect to $\Gamma_1(75)$ and by the valence formula
the identity holds if we verify the $q$-expansion past $q^{199}$.
\fussy

By Lemma \ref{LemmaFor3b2}
\begin{align*}
	B
	&=
	\frac{1}{5} + h(q^{50},q^{75}) - h(q^{35},q^{75})
	+q^{8}\frac{\aqprod{q^{75}}{q^{75}}{\infty}^2\jacprod{q^{20}}{q^{75}}}{\jacprod{q^{15},q^{35}}{q^{75}}}
	-q^{12}\frac{\aqprod{q^{75}}{q^{75}}{\infty}^2\jacprod{q^{10}}{q^{75}}}{\jacprod{q^{20},q^{30}}{q^{75}}}
	\\&\quad
	-q^{2}\frac{\aqprod{q^{75}}{q^{75}}{\infty}^2\jacprod{q^{35}}{q^{75}}}{\jacprod{q^{5},q^{30}}{q^{75}}}
	+q^{8}\frac{\aqprod{q^{75}}{q^{75}}{\infty}^2\jacprod{q^{5}}{q^{75}}}{\jacprod{q^{10},q^{15}}{q^{75}}}
	-q^{11}\frac{\aqprod{q^{75}}{q^{75}}{\infty}^2\jacprod{q^{5}}{q^{75}}}{\jacprod{q^{15},q^{20}}{q^{75}}}
	\\&\quad
	+q^{3}\frac{\aqprod{q^{75}}{q^{75}}{\infty}^2\jacprod{q^{25}}{q^{75}}}{\jacprod{q^{5},q^{30}}{q^{75}}}
	-q^{6}\frac{\aqprod{q^{75}}{q^{75}}{\infty}^2\jacprod{q^{20}}{q^{75}}}{\jacprod{q^{10},q^{30}}{q^{75}}}
	+q^{10}\frac{\aqprod{q^{75}}{q^{75}}{\infty}^2\jacprod{q^{10}}{q^{75}}}{\jacprod{q^{15},q^{25}}{q^{75}}}
.
\end{align*}

By expanding 
$h(q^{35},q^{75})$
into products by (\ref{EqHToProducts}),
replacing $h(q^{50},q^{75})$ with $V_{\chi,1}(25\tau)$,
and then
multiplying (\ref{EqPP3Mod5Eq2}) by 
$q^{-2}\frac{\aqprod{q^{75}}{q^{75}}{\infty}^2\jacprod{q^{5},q^{30}}{q^{75}}}
{\aqprod{q}{q}{\infty}\jacprod{q^{35}}{q^{75}}}$
we find we are to show
\begin{align*}
	&-1 
	+\frac{\GEta{75}{25}{\tau}}
		{\GEta{75}{35}{\tau}}
	+\frac{\GEta{75}{5}{\tau}\GEta{75}{20}{\tau}\GEta{75}{30}{\tau}}
		{\GEta{75}{15}{\tau}\GEta{75}{35}{\tau}^2}
	-\frac{\GEta{75}{5}{\tau}^2\GEta{75}{30}{\tau}}
		{\GEta{75}{15}{\tau}\GEta{75}{20}{\tau}\GEta{75}{35}{\tau}}
	\\&
	-\frac{\GEta{75}{5}{\tau}\GEta{75}{10}{\tau}}
		{\GEta{75}{20}{\tau}\GEta{75}{35}{\tau}}
	-\frac{\GEta{75}{5}{\tau}\GEta{75}{20}{\tau}}
		{\GEta{75}{10}{\tau}\GEta{75}{35}{\tau}}
	+\frac{\GEta{75}{5}{\tau}\GEta{75}{10}{\tau}\GEta{75}{30}{\tau}}
		{\GEta{75}{15}{\tau}\GEta{75}{25}{\tau}\GEta{75}{35}{\tau}}
	\\&
	+\frac{\GEta{75}{5}{\tau}^2\GEta{75}{30}{\tau}}
		{\GEta{75}{10}{\tau}\GEta{75}{15}{\tau}\GEta{75}{35}{\tau}}
	+\frac{4}{15}\frac{\GEta{75}{10}{\tau}}
		{\GEta{75}{5}{\tau}}
	-\frac{4}{15}\frac{\GEta{75}{5}{\tau}^4}
		{\GEta{75}{35}{\tau}^4}
	-\frac{2}{15}\frac{\GEta{75}{5}{\tau}^2\GEta{75}{20}{\tau}\GEta{75}{30}{\tau}}
		{\GEta{75}{10}{\tau}^2\GEta{75}{15}{\tau}\GEta{75}{35}{\tau}}
	\\&
	-\frac{2}{15}\frac{\GEta{75}{10}{\tau}^3\GEta{75}{30}{\tau}}
		{\GEta{75}{5}{\tau}^2\GEta{75}{15}{\tau}\GEta{75}{35}{\tau}}
	-\frac{1}{15}\frac{\GEta{75}{5}{\tau}\GEta{75}{10}{\tau}}
		{\GEta{75}{20}{\tau}^2}	
	-\frac{1}{15}\frac{\GEta{75}{5}{\tau}\GEta{75}{20}{\tau}^3}
		{\GEta{75}{10}{\tau}^3\GEta{75}{35}{\tau}}
	\\&
	+\frac{1}{30}\frac{\GEta{75}{5}{\tau}^2\GEta{75}{20}{\tau}\GEta{75}{30}{\tau}}
		{\GEta{75}{35}{\tau}^3\GEta{75}{15}{\tau}}
	-\frac{1}{30}\frac{\GEta{75}{5}{\tau}\GEta{75}{30}{\tau}\GEta{75}{35}{\tau}^2}
		{\GEta{75}{20}{\tau}^3\GEta{75}{15}{\tau}}
	+\frac{\GEta{75}{5}{\tau}\GEta{75}{30}{\tau}}{\GEta{75}{0}{\tau}\GEta{75}{35}{\tau}} V_{\chi,1}(25\tau)
	\\
	=&
	\frac{1}{5}\frac{\GEta{1}{0}{\tau}^{1/2}\GEta{75}{5}{\tau}\GEta{75}{25}{\tau}\GEta{75}{30}{\tau}\GEta{25}{10}{\tau}}
		{\GEta{75}{0}{\tau}^{1/2}\GEta{75}{35}{\tau}\GEta{25}{5}{\tau}^2}
	+\frac{1}{5}\frac{\GEta{1}{0}{\tau}^{1/2}\GEta{75}{5}{\tau}\GEta{75}{25}{\tau}\GEta{75}{30}{\tau}}
		{\GEta{75}{0}{\tau}^{1/2}\GEta{75}{35}{\tau}\GEta{25}{5}{\tau}}
	\\&
	-\frac{\GEta{1}{0}{\tau}^{1/2}\GEta{75}{5}{\tau}\GEta{75}{30}{\tau}}
		{\GEta{75}{0}{\tau}^{1/2}\GEta{75}{15}{\tau}\GEta{75}{35}{\tau}^2}
	+\frac{2}{5}\frac{\GEta{1}{0}{\tau}^{1/2}\GEta{75}{5}{\tau}\GEta{75}{25}{\tau}\GEta{75}{30}{\tau}}
		{\GEta{75}{0}{\tau}^{1/2}\GEta{75}{35}{\tau}\GEta{25}{10}{\tau}}
	-\frac{\GEta{1}{0}{\tau}^{1/2}}
		{\GEta{75}{0}{\tau}^{1/2}\GEta{75}{35}{\tau}}
	\\&
	+\frac{2}{5}\frac{\GEta{1}{0}{\tau}^{1/2}\GEta{75}{5}{\tau}\GEta{75}{25}{\tau}\GEta{75}{30}{\tau}\GEta{25}{5}{\tau}}
		{\GEta{75}{0}{\tau}^{1/2}\GEta{75}{35}{\tau}\GEta{25}{10}{\tau}^2}
	+\frac{\GEta{1}{0}{\tau}^{1/2}\GEta{75}{5}{\tau}^2\GEta{75}{30}{\tau}}
		{\GEta{75}{0}{\tau}^{1/2}\GEta{75}{10}{\tau}\GEta{75}{15}{\tau}\GEta{75}{35}{\tau}^2}
.
\end{align*}
Each term is a modular functions with respect to $\Gamma_1(75)$ and
by the valence formula we need only verify the equality holds in the $q$-expansion
past $q^{198}$. Here we handle the term involving $V_{\chi,1}(25\tau)$
as in the proof of Theorem \ref{TheoremPP25}. Thus Theorem \ref{TheoremPP35} holds.


\section{Proof of Theorem \ref{TheoremPP45}}
We note that $V_3^5(4) = - V_3^5(11)$, $V_3^5(7) = - V_3^5(8)$, and
$V_3^5(10) = - V_3^5(5)$ so 
we set
\begin{align*}
	A &= \frac{3}{5} + V_3^5(1)+V_3^5(5) -q^2U_3^5(7)+q^5U_3^5(16)
	,\\
	B &= \frac{1}{5} - V_3^5(11) + V_3^5(8) - q^3U_3^5(10) + q^4U_3^5(13)
.
\end{align*}
By (\ref{PP4Zeta5}) the left hand side of Theorem \ref{TheoremPP45}
is $\frac{A+(\zeta_5+\zeta_5^4)B}{\aqprod{q}{q}{\infty}}$,
so we must show
\begin{align}\label{EqPP4Mod5Eq1}
	\frac{A}{\aqprod{q}{q}{\infty}}
	&=
	\aqprod{q^{25}}{q^{25}}{\infty}
	\left(
		\frac{3}{5}\frac{\jacprod{q^{10}}{q^{25}}}{\jacprod{q^5}{q^{25}}^2}
		-q^5\frac{\jacprod{q^{10}}{q^{75}}}{\jacprod{q^{5},q^{20},q^{25},q^{30}}{q^{75}}}
		-\frac{2}{5}q\frac{1}{\jacprod{q^5}{q^{25}}}
		+q\frac{1}{\jacprod{q^{10},q^{15},q^{25}}{q^{75}}}
		\right.\nonumber\\&\quad\left.		
		+\frac{1}{5}\frac{q^2}{\jacprod{q^{10}}{q^{25}}}
		-q^2\frac{1}{\jacprod{q^{5},q^{25},q^{30}}{q^{75}}}
		-\frac{1}{5}q^3\frac{\jacprod{q^{5}}{q^{25}}}{\jacprod{q^{10}}{q^{25}}^2}
	\right)
	,\\
	\label{EqPP4Mod5Eq2}
	\frac{B}{\aqprod{q}{q}{\infty}}
	&=
	\aqprod{q^{25}}{q^{25}}{\infty}
	\left(
		\frac{1}{5}\frac{\jacprod{q^{10}}{q^{25}}}{\jacprod{q^5}{q^{25}}^2}
		-q^5\frac{\jacprod{q^{10}}{q^{75}}}{\jacprod{q^{5},q^{20},q^{25},q^{30}}{q^{75}}}
		+\frac{1}{5}q\frac{1}{\jacprod{q^5}{q^{25}}}
		+\frac{2}{5}\frac{q^2}{\jacprod{q^{10}}{q^{25}}}
		\right.\nonumber\\&\quad\left.
		-\frac{2}{5}q^3\frac{\jacprod{q^{5}}{q^{25}}}{\jacprod{q^{10}}{q^{25}}^2}
		-q^8\frac{\jacprod{q^5}{q^{75}}}{\jacprod{q^{10},q^{15},q^{25},q^{35}}{q^{75}}}
	\right)
.
\end{align}

By Lemma \ref{LemmaFor3b2}
\begin{align*}
	A 
	&=
	\frac{3}{5}+h(q^{50},q^{75}) - h(q^5,q^{75})
	+q^{8}\frac{\aqprod{q^{75}}{q^{75}}{\infty}^2\jacprod{q^{20}}{q^{75}}}
		{\jacprod{q^{15},q^{35}}{q^{75}}}
	-q^{12}\frac{\aqprod{q^{75}}{q^{75}}{\infty}^2\jacprod{q^{10}}{q^{75}}}
		{\jacprod{q^{20},q^{30}}{q^{75}}}
	\\&\quad
	-q^{2}\frac{\aqprod{q^{75}}{q^{75}}{\infty}^2\jacprod{q^{35}}{q^{75}}}
		{\jacprod{q^{5},q^{30}}{q^{75}}}
	+q^{8}\frac{\aqprod{q^{75}}{q^{75}}{\infty}^2\jacprod{q^{5}}{q^{75}}}
		{\jacprod{q^{10},q^{15}}{q^{75}}}
	-q^{2}\frac{\aqprod{q^{75}}{q^{75}}{\infty}^2\jacprod{q^{25}}{q^{75}}}
		{\jacprod{q^{10},q^{15}}{q^{75}}}
	\\&\quad
	-q^{10}\frac{\aqprod{q^{75}}{q^{75}}{\infty}^2\jacprod{q^{20}}{q^{75}}}
		{\jacprod{q^{25},q^{30}}{q^{75}}}
	+q^{14}\frac{\aqprod{q^{75}}{q^{75}}{\infty}^2\jacprod{q^{10}}{q^{75}}}
		{\jacprod{q^{30},q^{35}}{q^{75}}}
	+q^{4}\frac{\aqprod{q^{75}}{q^{75}}{\infty}^2\jacprod{q^{35}}{q^{75}}}
		{\jacprod{q^{15},q^{20}}{q^{75}}}
\end{align*}
and with this, after expanding 	
$h(q^5,q^{75})$ into products and replacing 
$h(q^{50},q^{75})$ with $V_{\chi,1}(\tau)$,
we find (\ref{EqPP4Mod5Eq1}) is equivalent to
\begin{align*}
	&\quad
	1
	-\frac{\GEta{75}{10}{\tau}\GEta{75}{15}{\tau}\GEta{75}{35}{\tau}}
		{\GEta{75}{20}{\tau}^{2}\GEta{75}{30}{\tau}}
	-\frac{\GEta{75}{15}{\tau}\GEta{75}{35}{\tau}^{2}}
		{\GEta{75}{5}{\tau}\GEta{75}{20}{\tau}\GEta{75}{30}{\tau}}
	+\frac{\GEta{75}{5}{\tau}\GEta{75}{35}{\tau}}
		{\GEta{75}{10}{\tau}\GEta{75}{20}{\tau}}
	\\&\quad
	-\frac{\GEta{75}{25}{\tau}\GEta{75}{35}{\tau}}
		{\GEta{75}{10}{\tau}\GEta{75}{20}{\tau}}
	-\frac{\GEta{75}{15}{\tau}\GEta{75}{35}{\tau}}
		{\GEta{75}{25}{\tau}\GEta{75}{30}{\tau}}
	+\frac{\GEta{75}{10}{\tau}\GEta{75}{15}{\tau}}
		{\GEta{75}{20}{\tau}\GEta{75}{30}{\tau}}
	\\&\quad
	+\frac{\GEta{75}{35}{\tau}^{2}}
		{\GEta{75}{20}{\tau}^{2}}
	+\frac{4}{15}q^{-3}\frac{\GEta{75}{5}{\tau}\GEta{75}{35}{\tau}}
		{\GEta{75}{10}{\tau}^{2}}
	+\frac{4}{15}\frac{\GEta{75}{10}{\tau}^{3}\GEta{75}{35}{\tau}}
		{\GEta{75}{5}{\tau}^{3}\GEta{75}{20}{\tau}}
	\\&\quad
	+\frac{2}{15}\frac{\GEta{75}{10}{\tau}\GEta{75}{15}{\tau}\GEta{75}{35}{\tau}^{2}}
		{\GEta{75}{20}{\tau}^{3}\GEta{75}{30}{\tau}}
	+\frac{2}{15}\frac{\GEta{75}{15}{\tau}\GEta{75}{20}{\tau}^{2}\GEta{75}{35}{\tau}}
		{\GEta{75}{10}{\tau}^{3}\GEta{75}{30}{\tau}}
	-\frac{1}{15}\frac{\GEta{75}{5}{\tau}}
		{\GEta{75}{35}{\tau}}
	\\&\quad
	+\frac{1}{15}\frac{\GEta{75}{35}{\tau}^{4}}
		{\GEta{75}{20}{\tau}^{4}}
	-\frac{1}{30}\frac{\GEta{75}{10}{\tau}\GEta{75}{15}{\tau}\GEta{75}{35}{\tau}^{2}}
		{\GEta{75}{5}{\tau}^{2}\GEta{75}{20}{\tau}\GEta{75}{30}{\tau}}
	+\frac{1}{30}q^{22}\frac{\GEta{75}{5}{\tau}^{3}\GEta{75}{15}{\tau}}
		{\GEta{75}{20}{\tau}\GEta{75}{30}{\tau}\GEta{75}{35}{\tau}^{2}}
	\\&\quad
	+\frac{\GEta{75}{15}{\tau}\GEta{75}{35}{\tau}\GEta{1}{0}{\tau}^{1/2}\GEta{3}{0}{\tau}^{1/2}}
		{\GEta{75}{30}{\tau}\GEta{5}{0}{\tau}}
		\frac{\eta(5\tau)^2}{\eta(\tau)\eta(3\tau)\eta(75\tau)^2}
		V_{\chi,1}(25\tau)
	\\
	&=
	\frac{3}{5}\frac{\GEta{1}{0}{\tau}^{1/2}\GEta{25}{10}{\tau}\GEta{75}{15}{\tau}\GEta{75}{25}{\tau}\GEta{75}{35}{\tau}}
		{\GEta{75}{0}{\tau}^{1/2}\GEta{25}{5}{\tau}^2\GEta{75}{20}{\tau}}
	-\frac{\GEta{1}{0}{\tau}^{1/2}\GEta{75}{10}{\tau}\GEta{75}{15}{\tau}\GEta{75}{35}{\tau}}
		{\GEta{75}{0}{\tau}^{1/2}\GEta{75}{5}{\tau}\GEta{75}{20}{\tau}^2\GEta{75}{30}{\tau}}
	\\&\quad
	-\frac{2}{5}\frac{\GEta{1}{0}{\tau}^{1/2}\GEta{75}{15}{\tau}\GEta{75}{25}{\tau}\GEta{75}{35}{\tau}}
		{\GEta{75}{0}{\tau}^{1/2}\GEta{25}{5}{\tau}\GEta{75}{20}{\tau}}
	+\frac{\GEta{1}{0}{\tau}^{1/2}\GEta{75}{35}{\tau}}
		{\GEta{75}{0}{\tau}^{1/2}\GEta{75}{10}{\tau}\GEta{75}{20}{\tau}}
	\\&\quad
	+\frac{1}{5}\frac{\GEta{1}{0}{\tau}^{1/2}\GEta{75}{15}{\tau}\GEta{75}{25}{\tau}\GEta{75}{35}{\tau}}
		{\GEta{75}{0}{\tau}^{1/2}\GEta{25}{10}{\tau}\GEta{75}{20}{\tau}}
	-\frac{\GEta{1}{0}{\tau}^{1/2}\GEta{75}{15}{\tau}\GEta{75}{35}{\tau}}
		{\GEta{75}{0}{\tau}^{1/2}\GEta{75}{5}{\tau}\GEta{75}{20}{\tau}\GEta{75}{30}{\tau}}
	\\&\quad
	-\frac{1}{5}\frac{\GEta{1}{0}{\tau}^{1/2}\GEta{25}{5}{\tau}\GEta{75}{15}{\tau}\GEta{75}{25}{\tau}\GEta{75}{35}{\tau}}
		{\GEta{75}{0}{\tau}^{1/2}\GEta{25}{10}{\tau}^2\GEta{75}{20}{\tau}}
.	
\end{align*}
However each term is a modular function with respect to $\Gamma_0(75)$, 
where we have treated the term with $V_{\chi,1}(25\tau)$ as before. This time
by the valence formula it suffices to verify the identity in the 
$q$-expansion past $q^{192}$.

Next we have by Lemma \ref{LemmaFor3b2}
\begin{align*}
	B
	&=
	\frac{1}{5} + h(q^{35},q^{75}) - h(q^{20},q^{75})
	+q^{11}\frac{\aqprod{q^{75}}{q^{75}}{\infty}^2\jacprod{q^{5}}{q^{75}}}
		{\jacprod{q^{15},q^{20}}{q^{75}}}
	-q^{3}\frac{\aqprod{q^{75}}{q^{75}}{\infty}^2\jacprod{q^{25}}{q^{75}}}
		{\jacprod{q^{5},q^{30}}{q^{75}}}
	\\&\quad
	+q^{6}\frac{\aqprod{q^{75}}{q^{75}}{\infty}^2\jacprod{q^{20}}{q^{75}}}
		{\jacprod{q^{10},q^{30}}{q^{75}}}
	-q^{10}\frac{\aqprod{q^{75}}{q^{75}}{\infty}^2\jacprod{q^{10}}{q^{75}}}
		{\jacprod{q^{15},q^{25}}{q^{75}}}
	+q^{4}\frac{\aqprod{q^{75}}{q^{75}}{\infty}^2\jacprod{q^{10}}{q^{75}}}
		{\jacprod{q^{5},q^{15}}{q^{75}}}
	\\&\quad
	-q^{4}\frac{\aqprod{q^{75}}{q^{75}}{\infty}^2\jacprod{q^{35}}{q^{75}}}
		{\jacprod{q^{10},q^{30}}{q^{75}}}
	-q^{15}\frac{\aqprod{q^{75}}{q^{75}}{\infty}^2\jacprod{q^{5}}{q^{75}}}
		{\jacprod{q^{25},q^{30}}{q^{75}}}
	+q^{7}\frac{\aqprod{q^{75}}{q^{75}}{\infty}^2\jacprod{q^{25}}{q^{75}}}
		{\jacprod{q^{15},q^{35}}{q^{75}}}
\end{align*}
and with this,
after expanding $h(q^{35},q^{75})$ and $h(q^{20},q^{75})$
into products,
we find (\ref{EqPP4Mod5Eq2})
to be equivalent to
\begin{align*}
	&\quad
	1
	-\frac{\GEta{75}{15}{\tau}\GEta{75}{20}{\tau}\GEta{75}{25}{\tau}}
		{\GEta{75}{5}{\tau}^{2}\GEta{75}{30}{\tau}}
	+\frac{\GEta{75}{15}{\tau}\GEta{75}{20}{\tau}^{2}}
		{\GEta{75}{5}{\tau}\GEta{75}{10}{\tau}\GEta{75}{30}{\tau}}
	-\frac{\GEta{75}{10}{\tau}\GEta{75}{20}{\tau}}
		{\GEta{75}{5}{\tau}\GEta{75}{25}{\tau}}
	\\&\quad
	+\frac{\GEta{75}{10}{\tau}\GEta{75}{20}{\tau}}
		{\GEta{75}{5}{\tau}^{2}}
	-\frac{\GEta{75}{15}{\tau}\GEta{75}{20}{\tau}\GEta{75}{35}{\tau}}
		{\GEta{75}{5}{\tau}\GEta{75}{10}{\tau}\GEta{75}{30}{\tau}}
	-\frac{\GEta{75}{15}{\tau}\GEta{75}{20}{\tau}}
		{\GEta{75}{25}{\tau}\GEta{75}{30}{\tau}}
	+\frac{\GEta{75}{20}{\tau}\GEta{75}{25}{\tau}}
		{\GEta{75}{5}{\tau}\GEta{75}{35}{\tau}}
	\\&\quad
	-\frac{3}{10}\frac{\GEta{75}{20}{\tau}^{2}}
		{\GEta{75}{35}{\tau}^{2}}
	+\frac{3}{10}\frac{\GEta{75}{35}{\tau}^{3}}
		{\GEta{75}{5}{\tau}\GEta{75}{20}{\tau}^{2}}
	-\frac{2}{5}\frac{\GEta{75}{10}{\tau}\GEta{75}{15}{\tau}\GEta{75}{20}{\tau}\GEta{75}{35}{\tau}}
		{\GEta{75}{5}{\tau}^{3}\GEta{75}{30}{\tau}}
	\\&\quad
	+\frac{2}{5}\frac{\GEta{75}{5}{\tau}^{2}\GEta{75}{15}{\tau}\GEta{75}{20}{\tau}}
		{\GEta{75}{30}{\tau}\GEta{75}{35}{\tau}^{3}}
	+\frac{1}{5}\frac{\GEta{75}{20}{\tau}^{2}}
		{\GEta{75}{10}{\tau}^{2}}
	+\frac{1}{5}\frac{\GEta{75}{10}{\tau}^{3}\GEta{75}{20}{\tau}}
		{\GEta{75}{5}{\tau}^{4}}
	\\&\quad
	+\frac{1}{10}\frac{\GEta{75}{10}{\tau}\GEta{75}{15}{\tau}\GEta{75}{35}{\tau}}
		{\GEta{75}{5}{\tau}\GEta{75}{20}{\tau}\GEta{75}{30}{\tau}}
	+\frac{1}{10}\frac{\GEta{75}{15}{\tau}\GEta{75}{20}{\tau}^{4}}
		{\GEta{75}{5}{\tau}\GEta{75}{10}{\tau}^{3}\GEta{75}{30}{\tau}}
	\\
	&=	
	\frac{1}{5}\frac{\GEta{1}{0}{\tau}^{1/2}\GEta{25}{10}{\tau}\GEta{75}{15}{\tau}\GEta{75}{25}{\tau}\GEta{75}{20}{\tau}}
		{\GEta{75}{0}{\tau}^{1/2}\GEta{25}{5}{\tau}^2\GEta{75}{5}{\tau}}
	-\frac{\GEta{1}{0}{\tau}^{1/2}\GEta{75}{10}{\tau}\GEta{75}{15}{\tau}}
		{\GEta{75}{0}{\tau}^{1/2}\GEta{75}{5}{\tau}^2\GEta{75}{30}{\tau}}
	\\&\quad
	+\frac{1}{5}\frac{\GEta{1}{0}{\tau}^{1/2}\GEta{75}{15}{\tau}\GEta{75}{25}{\tau}\GEta{75}{20}{\tau}}
		{\GEta{75}{0}{\tau}^{1/2}\GEta{25}{5}{\tau}\GEta{75}{5}{\tau}}
	+\frac{2}{5}\frac{\GEta{1}{0}{\tau}^{1/2}\GEta{75}{15}{\tau}\GEta{75}{25}{\tau}\GEta{75}{20}{\tau}}
		{\GEta{75}{0}{\tau}^{1/2}\GEta{25}{10}{\tau}\GEta{75}{5}{\tau}}
	\\&\quad
	-\frac{2}{5}\frac{\GEta{1}{0}{\tau}^{1/2}\GEta{25}{5}{\tau}\GEta{75}{15}{\tau}\GEta{75}{25}{\tau}\GEta{75}{20}{\tau}}
		{\GEta{75}{0}{\tau}^{1/2}\GEta{25}{10}{\tau}^2\GEta{75}{5}{\tau}}
	-\frac{\GEta{1}{0}{\tau}^{1/2}\GEta{75}{20}{\tau}}
		{\GEta{75}{0}{\tau}^{1/2}\GEta{75}{10}{\tau}\GEta{75}{35}{\tau}}
.	
\end{align*}
Again each term is a modular function with respect to $\Gamma_1(75)$
and we need only verify the $q$-expansion past $q^{189}$.
Thus Theorem \ref{TheoremPP45} holds.

\section{Proof of Theorem \ref{TheoremPP37} }

First we set
\begin{align*}
	A &=\frac{5}{7} + V_7^3(2)- V_7^3(17) -qU_7^3(4)+q^6U_7^3(19)
	,\\
	B &=\frac{3}{7} + V_7^3(5)- V_7^3(14) -q^2U_7^3(7)+q^5U_7^3(16)
	,\\
	C &=\frac{1}{7} + V_7^3(8)- V_7^3(11) -q^3U_7^3(10)+q^4U_7^3(13)
.
\end{align*}
By (\ref{PP3Zeta7}) the left hand side of Theorem \ref{TheoremPP37} is 
$\frac{A+(\zeta_7+\zeta_7^6)B+(1+\zeta_7^2+\zeta_7^5)C}{\aqprod{q}{q}{\infty}}$,
so we must show
\begin{align}\label{EqPP3Mod7Eq1}
	\frac{A}{\aqprod{q}{q}{\infty}}
	&=
	\aqprod{q^{49}}{q^{49}}{\infty}
	\left(
		\frac{5}{7}\frac{\jacprod{q^{21}}{q^{49}}}{\jacprod{q^7,q^{14}}{q^{49}}}
		-q^{14}\frac{1}{\jacprod{q^{42},q^{49},q^{56}}{q^{147}}}
		-\frac{2}{7}q\frac{1}{\jacprod{q^7}{q^{49}}}
		+\frac{3}{7}q^2\frac{\jacprod{q^{14}}{q^{49}}}{\jacprod{q^7,q^{21}}{q^{49}}}
		\right.\nonumber\\&\quad\left.		
		+\frac{1}{7}q^3\frac{1}{\jacprod{q^{14}}{q^{49}}}
		-\frac{3}{7}q^4\frac{1}{\jacprod{q^{21}}{q^{49}}}
		-\frac{1}{7}q^6\frac{\jacprod{q^{7}}{q^{49}}}{\jacprod{q^{14},q^{21}}{q^{49}}}
		+q^{6}\frac{1}{\jacprod{q^{14},q^{49},q^{63}}{q^{147}}}
	\right)
	,\\
	\label{EqPP3Mod7Eq2}
	\frac{B}{\aqprod{q}{q}{\infty}}
	&=
	\aqprod{q^{49}}{q^{49}}{\infty}
	\left(
		\frac{3}{7}\frac{\jacprod{q^{21}}{q^{49}}}{\jacprod{q^7,q^{14}}{q^{49}}}
		-q^{14}\frac{1}{\jacprod{q^{42},q^{49},q^{56}}{q^{147}}}
		+\frac{3}{7}q\frac{1}{\jacprod{q^7}{q^{49}}}
		-\frac{1}{7}q^2\frac{\jacprod{q^{14}}{q^{49}}}{\jacprod{q^7,q^{21}}{q^{49}}}
		\right.\nonumber\\&\quad\left.		
		+\frac{2}{7}q^3\frac{1}{\jacprod{q^{14}}{q^{49}}}
		+\frac{1}{7}q^4\frac{1}{\jacprod{q^{21}}{q^{49}}}
		+q^5\frac{\jacprod{q^{35}}{q^{147}}}{\jacprod{q^{21},q^{28},q^{49},q^{49}}{q^{147}}}
		+q^{19}\frac{\jacprod{q^{14}}{q^{147}}}{\jacprod{q^{21},q^{49},q^{49},q^{70}}{q^{147}}}
		\right.\nonumber\\&\quad\left.
		-\frac{2}{7}q^6\frac{\jacprod{q^{7}}{q^{49}}}{\jacprod{q^{14},q^{21}}{q^{49}}}
	\right)
	,\\
	\label{EqPP3Mod7Eq3}
	\frac{C}{\aqprod{q}{q}{\infty}}
	&=
	\aqprod{q^{49}}{q^{49}}{\infty}
	\left(
		\frac{1}{7}\frac{\jacprod{q^{21}}{q^{49}}}{\jacprod{q^7,q^{14}}{q^{49}}}
		+\frac{1}{7}q\frac{1}{\jacprod{q^7}{q^{49}}}
		+\frac{2}{7}q^2\frac{\jacprod{q^{14}}{q^{49}}}{\jacprod{q^7,q^{21}}{q^{49}}}
		-q^{9}\frac{1}{\jacprod{q^{21},q^{49},q^{70}}{q^{147}}}
		\right.\nonumber\\&\quad\left.		
		-\frac{4}{7}q^3\frac{1}{\jacprod{q^{14}}{q^{49}}}
		+\frac{5}{7}q^4\frac{1}{\jacprod{q^{21}}{q^{49}}}
		-\frac{3}{7}q^6\frac{\jacprod{q^{7}}{q^{49}}}{\jacprod{q^{14},q^{21}}{q^{49}}}
		+q^{6}\frac{1}{\jacprod{q^{14},q^{49},q^{63}}{q^{147}}}
	\right)
.
\end{align}

We have
\begin{align*}
	A 
	&=
	\frac{5}{7} + h(q^{133},q^{147}) - h(q^{28},q^{147})
	+q^{5}\frac{\aqprod{q^{147}}{q^{147}}{\infty}^2\jacprod{q^{56}}{q^{147}}}{\jacprod{q^{21},q^{35}}{q^{147}}}
	+q^{16}\frac{\aqprod{q^{147}}{q^{147}}{\infty}^2\jacprod{q^{49}}{q^{147}}}{\jacprod{q^{42},q^{56}}{q^{147}}}
	\\&\quad
	+q^{33}\frac{\aqprod{q^{147}}{q^{147}}{\infty}^2\jacprod{q^{7}}{q^{147}}}{\jacprod{q^{63},q^{70}}{q^{147}}}
	-q^{21}\frac{\aqprod{q^{147}}{q^{147}}{\infty}^2\jacprod{q^{35}}{q^{147}}}{\jacprod{q^{49},q^{63}}{q^{147}}}
	-q^{8}\frac{\aqprod{q^{147}}{q^{147}}{\infty}^2\jacprod{q^{70}}{q^{147}}}{\jacprod{q^{28},q^{42}}{q^{147}}}
	\\&\quad
	-q\frac{\aqprod{q^{147}}{q^{147}}{\infty}^2\jacprod{q^{28}}{q^{147}}}{\jacprod{q^{7},q^{21}}{q^{147}}}
	+q^{6}\frac{\aqprod{q^{147}}{q^{147}}{\infty}^2\jacprod{q^{14}}{q^{147}}}{\jacprod{q^{7},q^{21}}{q^{147}}}
	-q^{4}\frac{\aqprod{q^{147}}{q^{147}}{\infty}^2\jacprod{q^{56}}{q^{147}}}{\jacprod{q^{14},q^{42}}{q^{147}}}
	\\&\quad
	-q^{15}\frac{\aqprod{q^{147}}{q^{147}}{\infty}^2\jacprod{q^{49}}{q^{147}}}{\jacprod{q^{35},q^{63}}{q^{147}}}
	-q^{32}\frac{\aqprod{q^{147}}{q^{147}}{\infty}^2\jacprod{q^{7}}{q^{147}}}{\jacprod{q^{56},q^{63}}{q^{147}}}
	+q^{20}\frac{\aqprod{q^{147}}{q^{147}}{\infty}^2\jacprod{q^{35}}{q^{147}}}{\jacprod{q^{42},q^{70}}{q^{147}}}
	\\&\quad
	+q^{7}\frac{\aqprod{q^{147}}{q^{147}}{\infty}^2\jacprod{q^{70}}{q^{147}}}{\jacprod{q^{21},q^{49}}{q^{147}}}
	,\\
	B
	&=
	\frac{3}{7} + h(q^{112},q^{147}) - h(q^{49},q^{147})
	+q^{8}\frac{\aqprod{q^{147}}{q^{147}}{\infty}^2\jacprod{q^{70}}{q^{147}}}{\jacprod{q^{21},q^{56}}{q^{147}}}
	+q^{22}\frac{\aqprod{q^{147}}{q^{147}}{\infty}^2\jacprod{q^{28}}{q^{147}}}{\jacprod{q^{42},q^{70}}{q^{147}}}
	\\&\quad
	-q^{28}\frac{\aqprod{q^{147}}{q^{147}}{\infty}^2\jacprod{q^{14}}{q^{147}}}{\jacprod{q^{49},q^{63}}{q^{147}}}
	-q^{12}\frac{\aqprod{q^{147}}{q^{147}}{\infty}^2\jacprod{q^{56}}{q^{147}}}{\jacprod{q^{28},q^{63}}{q^{147}}}
	-q^{2}\frac{\aqprod{q^{147}}{q^{147}}{\infty}^2\jacprod{q^{49}}{q^{147}}}{\jacprod{q^{7},q^{42}}{q^{147}}}
	\\&\quad
	+q^{12}\frac{\aqprod{q^{147}}{q^{147}}{\infty}^2\jacprod{q^{7}}{q^{147}}}{\jacprod{q^{14},q^{21}}{q^{147}}}
	-q^{17}\frac{\aqprod{q^{147}}{q^{147}}{\infty}^2\jacprod{q^{7}}{q^{147}}}{\jacprod{q^{21},q^{28}}{q^{147}}}
	+q^{5}\frac{\aqprod{q^{147}}{q^{147}}{\infty}^2\jacprod{q^{35}}{q^{147}}}{\jacprod{q^{7},q^{42}}{q^{147}}}	
	\\&\quad
	-q^{6}\frac{\aqprod{q^{147}}{q^{147}}{\infty}^2\jacprod{q^{70}}{q^{147}}}{\jacprod{q^{14},q^{63}}{q^{147}}}
	-q^{20}\frac{\aqprod{q^{147}}{q^{147}}{\infty}^2\jacprod{q^{28}}{q^{147}}}{\jacprod{q^{35},q^{63}}{q^{147}}}
	+q^{26}\frac{\aqprod{q^{147}}{q^{147}}{\infty}^2\jacprod{q^{14}}{q^{147}}}{\jacprod{q^{42},q^{56}}{q^{147}}}
	\\&\quad
	+q^{10}\frac{\aqprod{q^{147}}{q^{147}}{\infty}^2\jacprod{q^{56}}{q^{147}}}{\jacprod{q^{21},q^{70}}{q^{147}}}
	,\\
	C
	&=
	\frac{1}{7} + h(q^{91},q^{147}) - h(q^{70},q^{147})
	+q^{11}\frac{\aqprod{q^{147}}{q^{147}}{\infty}^2\jacprod{q^{49}}{q^{147}}}{\jacprod{q^{21},q^{70}}{q^{147}}}
	+q^{28}\frac{\aqprod{q^{147}}{q^{147}}{\infty}^2\jacprod{q^{7}}{q^{147}}}{\jacprod{q^{42},q^{49}}{q^{147}}}
	\\&\quad
	-q^{16}\frac{\aqprod{q^{147}}{q^{147}}{\infty}^2\jacprod{q^{35}}{q^{147}}}{\jacprod{q^{28},q^{63}}{q^{147}}}
	-q^{3}\frac{\aqprod{q^{147}}{q^{147}}{\infty}^2\jacprod{q^{70}}{q^{147}}}{\jacprod{q^{7},q^{63}}{q^{147}}}
	+q^{10}\frac{\aqprod{q^{147}}{q^{147}}{\infty}^2\jacprod{q^{28}}{q^{147}}}{\jacprod{q^{14},q^{42}}{q^{147}}}
	\\&\quad
	-q^{16}\frac{\aqprod{q^{147}}{q^{147}}{\infty}^2\jacprod{q^{14}}{q^{147}}}{\jacprod{q^{21},q^{35}}{q^{147}}}
	-q^{14}\frac{\aqprod{q^{147}}{q^{147}}{\infty}^2\jacprod{q^{28}}{q^{147}}}{\jacprod{q^{21},q^{49}}{q^{147}}}
	+q^{20}\frac{\aqprod{q^{147}}{q^{147}}{\infty}^2\jacprod{q^{14}}{q^{147}}}{\jacprod{q^{28},q^{42}}{q^{147}}}
	\\&\quad
	+q^{4}\frac{\aqprod{q^{147}}{q^{147}}{\infty}^2\jacprod{q^{56}}{q^{147}}}{\jacprod{q^{7},q^{63}}{q^{147}}}
	-q^{8}\frac{\aqprod{q^{147}}{q^{147}}{\infty}^2\jacprod{q^{49}}{q^{147}}}{\jacprod{q^{14},q^{63}}{q^{147}}}
	-q^{25}\frac{\aqprod{q^{147}}{q^{147}}{\infty}^2\jacprod{q^{7}}{q^{147}}}{\jacprod{q^{35},q^{42}}{q^{147}}}
	\\&\quad
	+q^{13}\frac{\aqprod{q^{147}}{q^{147}}{\infty}^2\jacprod{q^{35}}{q^{147}}}{\jacprod{q^{21},q^{56}}{q^{147}}}
.
\end{align*}

By expanding
$h(q^{133},q^{147})$ and $h(q^{28},q^{147})$
into products with (\ref{EqHToProducts}) and
then
multiplying both sides of (\ref{EqPP3Mod7Eq1}) by 
$\frac{q^{-5}\jacprod{q^{21},q^{35}}{q^{147}}}{\jacprod{56}{q^{147}}}$
we find it is equivalent to proving
\begin{align*}
	&\quad 1
	+\frac{\GEta{147}{21}{\tau}\GEta{147}{35}{\tau}\GEta{147}{49}{\tau}}
		{\GEta{147}{42}{\tau}\GEta{147}{56}{\tau}^{2}}
	+\frac{\GEta{147}{7}{\tau}\GEta{147}{21}{\tau}\GEta{147}{35}{\tau}}
		{\GEta{147}{56}{\tau}\GEta{147}{63}{\tau}\GEta{147}{70}{\tau}}
	-\frac{\GEta{147}{21}{\tau}\GEta{147}{35}{\tau}^{2}}
		{\GEta{147}{49}{\tau}\GEta{147}{56}{\tau}\GEta{147}{63}{\tau}}
	\\&\quad
	-\frac{\GEta{147}{21}{\tau}\GEta{147}{35}{\tau}\GEta{147}{70}{\tau}}
		{\GEta{147}{28}{\tau}\GEta{147}{42}{\tau}\GEta{147}{56}{\tau}}
	-\frac{\GEta{147}{28}{\tau}\GEta{147}{35}{\tau}}
		{\GEta{147}{7}{\tau}\GEta{147}{56}{\tau}}
	+\frac{\GEta{147}{14}{\tau}\GEta{147}{35}{\tau}}
		{\GEta{147}{7}{\tau}\GEta{147}{56}{\tau}}
	\\&\quad	
	-\frac{\GEta{147}{21}{\tau}\GEta{147}{35}{\tau}}
		{\GEta{147}{14}{\tau}\GEta{147}{42}{\tau}}
	-\frac{\GEta{147}{21}{\tau}\GEta{147}{49}{\tau}}
		{\GEta{147}{56}{\tau}\GEta{147}{63}{\tau}}
	-\frac{\GEta{147}{7}{\tau}\GEta{147}{21}{\tau}\GEta{147}{35}{\tau}}
		{\GEta{147}{56}{\tau}^{2}\GEta{147}{63}{\tau}}
	\\&\quad
	+\frac{\GEta{147}{21}{\tau}\GEta{147}{35}{\tau}^{2}}
		{\GEta{147}{42}{\tau}\GEta{147}{56}{\tau}\GEta{147}{70}{\tau}}
	+\frac{\GEta{147}{35}{\tau}\GEta{147}{70}{\tau}}
		{\GEta{147}{49}{\tau}\GEta{147}{56}{\tau}}
	\\&\quad
	+\frac{11}{42}\frac{\GEta{147}{14}{\tau}\GEta{147}{21}{\tau}\GEta{147}{35}{\tau}}
		{\GEta{147}{28}{\tau}^{2}\GEta{147}{42}{\tau}}
	+\frac{8}{21}\frac{\GEta{147}{21}{\tau}\GEta{147}{28}{\tau}\GEta{147}{35}{\tau}^{2}}
		{\GEta{147}{56}{\tau}^{3}\GEta{147}{63}{\tau}}
	+\frac{4}{21}\frac{\GEta{147}{70}{\tau}}
		{\GEta{147}{35}{\tau}}
	\\&\quad
	-\frac{2}{21}\frac{\GEta{147}{7}{\tau}\GEta{147}{21}{\tau}\GEta{147}{35}{\tau}^{2}}
		{\GEta{147}{42}{\tau}\GEta{147}{56}{\tau}\GEta{147}{70}{\tau}^{2}}
	-\frac{1}{21}\frac{\GEta{147}{14}{\tau}\GEta{147}{21}{\tau}\GEta{147}{35}{\tau}\GEta{147}{70}{\tau}}
		{\GEta{147}{7}{\tau}^{2}\GEta{147}{56}{\tau}\GEta{147}{63}{\tau}}
	\\&\quad
	+\frac{1}{42}\frac{\GEta{147}{7}{\tau}\GEta{147}{28}{\tau}\GEta{147}{35}{\tau}}
		{\GEta{147}{14}{\tau}^{2}\GEta{147}{56}{\tau}}
	+\frac{11}{42}\frac{\GEta{147}{21}{\tau}\GEta{147}{28}{\tau}^{3}\GEta{147}{35}{\tau}}
		{\GEta{147}{14}{\tau}^{3}\GEta{147}{42}{\tau}\GEta{147}{56}{\tau}}
	\\&\quad
	+\frac{8}{21}\frac{\GEta{147}{21}{\tau}\GEta{147}{35}{\tau}\GEta{147}{56}{\tau}^{2}}
		{\GEta{147}{28}{\tau}^{3}\GEta{147}{63}{\tau}}
	-\frac{4}{21}\frac{\GEta{147}{35}{\tau}^{4}}
		{\GEta{147}{56}{\tau}^{4}}
	-\frac{2}{21}\frac{\GEta{147}{21}{\tau}\GEta{147}{70}{\tau}^{3}}
		{\GEta{147}{35}{\tau}^{2}\GEta{147}{42}{\tau}\GEta{147}{56}{\tau}}
	\\&\quad
	+\frac{1}{21}\frac{\GEta{147}{7}{\tau}^{3}\GEta{147}{21}{\tau}\GEta{147}{35}{\tau}}
		{\GEta{147}{56}{\tau}\GEta{147}{63}{\tau}\GEta{147}{70}{\tau}^{3}}
	+\frac{1}{42}\frac{\GEta{147}{14}{\tau}^{3}\GEta{147}{35}{\tau}}
		{\GEta{147}{7}{\tau}^{3}\GEta{147}{56}{\tau}}
	\\
	&=
	\frac{5}{7}
	\frac{\GEta{1}{0}{\tau}^{1/2}\GEta{147}{21}{\tau}\GEta{147}{35}{\tau}\GEta{147}{49}{\tau}\GEta{49}{21}{\tau}}
		{\GEta{147}{0}{\tau}^{1/2}\GEta{147}{56}{\tau}\GEta{49}{7}{\tau}\GEta{49}{14}{\tau}}
	-\frac{\GEta{1}{0}{\tau}^{1/2}\GEta{147}{21}{\tau}\GEta{147}{35}{\tau}}
		{\GEta{147}{0}{\tau}^{1/2}\GEta{147}{56}{\tau}\GEta{147}{42}{\tau}\GEta{147}{56}{\tau}}
	\\&\quad
	-\frac{2}{7}
	\frac{\GEta{1}{0}{\tau}^{1/2}\GEta{147}{21}{\tau}\GEta{147}{35}{\tau}\GEta{147}{49}{\tau}}
		{\GEta{147}{0}{\tau}^{1/2}\GEta{147}{56}{\tau}\GEta{49}{7}{\tau}}
	+\frac{3}{7}
	\frac{\GEta{1}{0}{\tau}^{1/2}\GEta{147}{21}{\tau}\GEta{147}{35}{\tau}\GEta{147}{49}{\tau}\GEta{49}{14}{\tau}}
		{\GEta{147}{0}{\tau}^{1/2}\GEta{147}{56}{\tau}\GEta{49}{7}{\tau}\GEta{49}{21}{\tau}}
	\\&\quad
	+\frac{1}{7}
	\frac{\GEta{1}{0}{\tau}^{1/2}\GEta{147}{21}{\tau}\GEta{147}{35}{\tau}\GEta{147}{49}{\tau}}
		{\GEta{147}{0}{\tau}^{1/2}\GEta{147}{56}{\tau}\GEta{49}{14}{\tau}}
	-\frac{3}{7}
	\frac{\GEta{1}{0}{\tau}^{1/2}\GEta{147}{21}{\tau}\GEta{147}{35}{\tau}\GEta{147}{49}{\tau}}
		{\GEta{147}{0}{\tau}^{1/2}\GEta{147}{56}{\tau}\GEta{49}{21}{\tau}}
	\\&\quad
	-\frac{1}{7}
	\frac{\GEta{1}{0}{\tau}^{1/2}\GEta{147}{21}{\tau}\GEta{147}{35}{\tau}\GEta{147}{49}{\tau}\GEta{49}{7}{\tau}}
		{\GEta{147}{0}{\tau}^{1/2}\GEta{147}{56}{\tau}\GEta{49}{14}{\tau}\GEta{49}{21}{\tau}}
	+\frac{\GEta{1}{0}{\tau}^{1/2}\GEta{147}{21}{\tau}\GEta{147}{35}{\tau}}
		{\GEta{147}{0}{\tau}^{1/2}\GEta{147}{56}{\tau}\GEta{147}{14}{\tau}\GEta{147}{63}{\tau}}
.
\end{align*}
Each term is a modular function with respect to $\Gamma_1(147)$ and by the 
valence formula it suffices to verify the identity in the $q$-expansion past
$q^{779}$.

By expanding
$h(q^{112},q^{147})$ into products and
replacing $h(q^{49},q^{147})$ with $V_{\chi,1}(49\tau)$
and then
multiplying both sides by 
$q^{-8}\frac{\jacprod{q^{21},q^{56}}{q^{147}}}{\aqprod{q^{147}}{q^{147}}{\infty}^2\jacprod{q^{70}}{q^{147}}}$
, we find (\ref{EqPP3Mod7Eq2}) is equivalent to
\begin{align*}
	&\quad 1
	+\frac{\GEta{147}{21}{\tau}\GEta{147}{28}{\tau}\GEta{147}{56}{\tau}}
		{\GEta{147}{42}{\tau}\GEta{147}{70}{\tau}^{2}}
	-\frac{\GEta{147}{14}{\tau}\GEta{147}{21}{\tau}\GEta{147}{56}{\tau}}
		{\GEta{147}{49}{\tau}\GEta{147}{63}{\tau}\GEta{147}{70}{\tau}}
	-\frac{\GEta{147}{21}{\tau}\GEta{147}{56}{\tau}^{2}}
		{\GEta{147}{28}{\tau}\GEta{147}{63}{\tau}\GEta{147}{70}{\tau}}
	\\&\quad
	-\frac{\GEta{147}{21}{\tau}\GEta{147}{49}{\tau}\GEta{147}{56}{\tau}}
		{\GEta{147}{7}{\tau}\GEta{147}{42}{\tau}\GEta{147}{70}{\tau}}
	+\frac{\GEta{147}{7}{\tau}\GEta{147}{56}{\tau}}
		{\GEta{147}{14}{\tau}\GEta{147}{70}{\tau}}
	-\frac{\GEta{147}{7}{\tau}\GEta{147}{56}{\tau}}
		{\GEta{147}{28}{\tau}\GEta{147}{70}{\tau}}
	\\&\quad
	+\frac{\GEta{147}{21}{\tau}\GEta{147}{35}{\tau}\GEta{147}{56}{\tau}}
		{\GEta{147}{7}{\tau}\GEta{147}{42}{\tau}\GEta{147}{70}{\tau}}
	-\frac{\GEta{147}{21}{\tau}\GEta{147}{56}{\tau}}
		{\GEta{147}{14}{\tau}\GEta{147}{63}{\tau}}
	-\frac{\GEta{147}{21}{\tau}\GEta{147}{28}{\tau}\GEta{147}{56}{\tau}}
		{\GEta{147}{35}{\tau}\GEta{147}{63}{\tau}\GEta{147}{70}{\tau}}
	\\&\quad
	+\frac{\GEta{147}{14}{\tau}\GEta{147}{21}{\tau}}
		{\GEta{147}{42}{\tau}\GEta{147}{70}{\tau}}
	+\frac{\GEta{147}{56}{\tau}^{2}}
		{\GEta{147}{70}{\tau}^{2}}
	+\frac{16}{63}\frac{\GEta{147}{7}{\tau}\GEta{147}{21}{\tau}\GEta{147}{35}{\tau}\GEta{147}{56}{\tau}}
		{\GEta{147}{42}{\tau}\GEta{147}{70}{\tau}^{3}}
	\\&\quad
	+\frac{8}{63}\frac{\GEta{147}{14}{\tau}\GEta{147}{21}{\tau}\GEta{147}{56}{\tau}}
		{\GEta{147}{7}{\tau}^{2}\GEta{147}{63}{\tau}}
	-\frac{4}{63}\frac{\GEta{147}{7}{\tau}\GEta{147}{28}{\tau}\GEta{147}{56}{\tau}}
		{\GEta{147}{14}{\tau}^{2}\GEta{147}{70}{\tau}}
	\\&\quad
	-\frac{2}{63}\frac{\GEta{147}{14}{\tau}\GEta{147}{21}{\tau}\GEta{147}{56}{\tau}^{2}}
		{\GEta{147}{28}{\tau}^{2}\GEta{147}{42}{\tau}\GEta{147}{70}{\tau}}
	-\frac{1}{63}\frac{\GEta{147}{21}{\tau}\GEta{147}{28}{\tau}\GEta{147}{35}{\tau}}
		{\GEta{147}{56}{\tau}\GEta{147}{63}{\tau}\GEta{147}{70}{\tau}}
	-\frac{1}{126}\frac{\GEta{147}{56}{\tau}^{2}}
		{\GEta{147}{35}{\tau}^{2}}
	\\&\quad
	+\frac{16}{63}\frac{\GEta{147}{21}{\tau}\GEta{147}{56}{\tau}\GEta{147}{70}{\tau}^{2}}
		{\GEta{147}{35}{\tau}^{3}\GEta{147}{42}{\tau}}
	-\frac{8}{63}\frac{\GEta{147}{7}{\tau}^{3}\GEta{147}{21}{\tau}\GEta{147}{56}{\tau}}
		{\GEta{147}{63}{\tau}\GEta{147}{70}{\tau}^{4}}
	-\frac{4}{63}\frac{\GEta{147}{14}{\tau}^{3}\GEta{147}{56}{\tau}}
		{\GEta{147}{7}{\tau}^{3}\GEta{147}{70}{\tau}}
	\\&\quad
	-\frac{2}{63}\frac{\GEta{147}{21}{\tau}\GEta{147}{28}{\tau}^{3}\GEta{147}{56}{\tau}}
		{\GEta{147}{14}{\tau}^{3}\GEta{147}{42}{\tau}\GEta{147}{70}{\tau}}
	-\frac{1}{63}\frac{\GEta{147}{21}{\tau}\GEta{147}{56}{\tau}^{4}}
		{\GEta{147}{28}{\tau}^{3}\GEta{147}{63}{\tau}\GEta{147}{70}{\tau}}
	+\frac{1}{126}\frac{\GEta{147}{35}{\tau}^{3}}
		{\GEta{147}{56}{\tau}^{2}\GEta{147}{70}{\tau}}
	\\&\quad
	+
	\frac{\GEta{147}{21}{\tau}\GEta{147}{56}{\tau}}{\GEta{147}{0}{\tau}\GEta{147}{70}{\tau}}	
		V_{\chi,1}(49\tau)
	\\
	&=
	\frac{3}{7}
	\frac{\GEta{1}{0}{\tau}^{1/2}\GEta{147}{21}{\tau}\GEta{147}{56}{\tau}\GEta{147}{49}{\tau}\GEta{49}{21}{\tau}}
		{\GEta{147}{0}{\tau}^{1/2}\GEta{147}{70}{\tau}\GEta{49}{7}{\tau}\GEta{49}{14}{\tau}}
	-\frac{\GEta{1}{0}{\tau}^{1/2}\GEta{147}{21}{\tau}\GEta{147}{56}{\tau}}
		{\GEta{147}{0}{\tau}^{1/2}\GEta{147}{70}{\tau}\GEta{147}{42}{\tau}\GEta{147}{56}{\tau}}
	\\&\quad
	+\frac{3}{7}
	\frac{\GEta{1}{0}{\tau}^{1/2}\GEta{147}{21}{\tau}\GEta{147}{56}{\tau}\GEta{147}{49}{\tau}}
		{\GEta{147}{0}{\tau}^{1/2}\GEta{147}{70}{\tau}\GEta{49}{7}{\tau}}
	-\frac{1}{7}
	\frac{\GEta{1}{0}{\tau}^{1/2}\GEta{147}{21}{\tau}\GEta{147}{56}{\tau}\GEta{147}{49}{\tau}\GEta{49}{14}{\tau}}
		{\GEta{147}{0}{\tau}^{1/2}\GEta{147}{70}{\tau}\GEta{49}{7}{\tau}\GEta{49}{21}{\tau}}
	\\&\quad
	+\frac{2}{7}
	\frac{\GEta{1}{0}{\tau}^{1/2}\GEta{147}{21}{\tau}\GEta{147}{56}{\tau}\GEta{147}{49}{\tau}}
		{\GEta{147}{0}{\tau}^{1/2}\GEta{147}{70}{\tau}\GEta{49}{14}{\tau}}
	+\frac{1}{7}
	\frac{\GEta{1}{0}{\tau}^{1/2}\GEta{147}{21}{\tau}\GEta{147}{56}{\tau}\GEta{147}{49}{\tau}}
		{\GEta{147}{0}{\tau}^{1/2}\GEta{147}{70}{\tau}\GEta{49}{21}{\tau}}
	\\&\quad
	+\frac{\GEta{1}{0}{\tau}^{1/2}\GEta{147}{21}{\tau}\GEta{147}{56}{\tau}}
		{\GEta{147}{0}{\tau}^{1/2}\GEta{147}{70}{\tau}\GEta{147}{21}{\tau}\GEta{147}{28}{\tau}\GEta{147}{49}{\tau}}
	+\frac{\GEta{1}{0}{\tau}^{1/2}\GEta{147}{21}{\tau}\GEta{147}{56}{\tau}}
		{\GEta{147}{0}{\tau}^{1/2}\GEta{147}{70}{\tau}^2\GEta{147}{21}{\tau}\GEta{147}{49}{\tau}}
	\\&\quad
	-\frac{2}{7}
	\frac{\GEta{1}{0}{\tau}^{1/2}\GEta{147}{21}{\tau}\GEta{147}{56}{\tau}\GEta{147}{49}{\tau}\GEta{49}{7}{\tau}}
		{\GEta{147}{0}{\tau}^{1/2}\GEta{147}{70}{\tau}\GEta{49}{14}{\tau}\GEta{49}{21}{\tau}}
.
\end{align*}
While it is true that the term involving $V_{\chi,1}(49\tau)$ is a modular function with
respect $\Gamma_1(147)$, it is not immediately apparent as in the 
$5$-dissections. However, we can write
\begin{align*}
	\frac{\GEta{147}{21}{\tau}\GEta{147}{56}{\tau}}
		{\GEta{147}{70}{\tau}\GEta{147}{21}{\tau}\GEta{147}{0}{\tau}}
	&=
	\frac{\GEta{147}{21}{\tau}\GEta{147}{56}{\tau}\GEta{1}{0}{\tau}^{1/2}}
		{\GEta{147}{70}{\tau}\GEta{147}{21}{\tau}\GEta{7}{0}{\tau}^{1/2}}
	\cdot
	\frac{\eta(7\tau)}{\eta(\tau)\eta(147)^2}
,
\end{align*}
to see a modular function and meromorphic modular form of weight $-1$
with respect to $\Gamma_1(447)$. Instead working with this larger congruence
subgroup, we must verify the identity in the $q$-expansion past
$q^{7804}$.

By expanding 
$h(q^{91},q^{147})$ and $h(q^{70},q^{147})$
into products and then
multiplying by 
$\frac{ q^{-11}\jacprod{q^{21},q^{70}}{q^{147}} }
{\aqprod{q^{147}}{q^{147}}{\infty}^2\jacprod{q^{49}}{q^{147}}}$,
we find (\ref{EqPP3Mod7Eq3}) is equivalent to 
\begin{align*}
	&\quad 1
	+\frac{\GEta{147}{7}{\tau}\GEta{147}{21}{\tau}\GEta{147}{70}{\tau}}
		{\GEta{147}{42}{\tau}\GEta{147}{49}{\tau}^{2}}
	-\frac{\GEta{147}{21}{\tau}\GEta{147}{35}{\tau}\GEta{147}{70}{\tau}}
		{\GEta{147}{28}{\tau}\GEta{147}{49}{\tau}\GEta{147}{63}{\tau}}
	-\frac{\GEta{147}{21}{\tau}\GEta{147}{70}{\tau}^{2}}
		{\GEta{147}{7}{\tau}\GEta{147}{49}{\tau}\GEta{147}{63}{\tau}}
	\\&\quad
	+\frac{\GEta{147}{21}{\tau}\GEta{147}{28}{\tau}\GEta{147}{70}{\tau}}
		{\GEta{147}{14}{\tau}\GEta{147}{42}{\tau}\GEta{147}{49}{\tau}}
	-\frac{\GEta{147}{14}{\tau}\GEta{147}{70}{\tau}}
		{\GEta{147}{35}{\tau}\GEta{147}{49}{\tau}}
	-\frac{\GEta{147}{28}{\tau}\GEta{147}{70}{\tau}}
		{\GEta{147}{49}{\tau}^{2}}
	\\&\quad
	+\frac{\GEta{147}{14}{\tau}\GEta{147}{21}{\tau}\GEta{147}{70}{\tau}}
		{\GEta{147}{28}{\tau}\GEta{147}{42}{\tau}\GEta{147}{49}{\tau}}
	+\frac{\GEta{147}{21}{\tau}\GEta{147}{56}{\tau}\GEta{147}{70}{\tau}}
		{\GEta{147}{7}{\tau}\GEta{147}{49}{\tau}\GEta{147}{63}{\tau}}
	-\frac{\GEta{147}{21}{\tau}\GEta{147}{70}{\tau}}
		{\GEta{147}{14}{\tau}\GEta{147}{63}{\tau}}
	\\&\quad
	-\frac{\GEta{147}{7}{\tau}\GEta{147}{21}{\tau}\GEta{147}{70}{\tau}}
		{\GEta{147}{35}{\tau}\GEta{147}{42}{\tau}\GEta{147}{49}{\tau}}
	+\frac{\GEta{147}{35}{\tau}\GEta{147}{70}{\tau}}
		{\GEta{147}{49}{\tau}\GEta{147}{56}{\tau}}
	+\frac{5}{21}\frac{\GEta{147}{56}{\tau}\GEta{147}{70}{\tau}^{2}}
		{\GEta{147}{35}{\tau}^{2}\GEta{147}{49}{\tau}}
	\\&\quad
	-\frac{5}{42}\frac{\GEta{147}{7}{\tau}\GEta{147}{21}{\tau}\GEta{147}{35}{\tau}}
		{\GEta{147}{42}{\tau}\GEta{147}{49}{\tau}\GEta{147}{70}{\tau}}
	+\frac{4}{21}\frac{\GEta{147}{14}{\tau}\GEta{147}{21}{\tau}\GEta{147}{70}{\tau}^{2}}
		{\GEta{147}{7}{\tau}^{2}\GEta{147}{49}{\tau}\GEta{147}{63}{\tau}}
	\\&\quad
	-\frac{2}{21}\frac{\GEta{147}{7}{\tau}\GEta{147}{28}{\tau}\GEta{147}{70}{\tau}}
		{\GEta{147}{14}{\tau}^{2}\GEta{147}{49}{\tau}}
	-\frac{1}{21}\frac{\GEta{147}{14}{\tau}\GEta{147}{21}{\tau}\GEta{147}{56}{\tau}\GEta{147}{70}{\tau}}
		{\GEta{147}{28}{\tau}^{2}\GEta{147}{42}{\tau}\GEta{147}{49}{\tau}}
	\\&\quad
	-\frac{1}{42}\frac{\GEta{147}{21}{\tau}\GEta{147}{28}{\tau}\GEta{147}{35}{\tau}\GEta{147}{70}{\tau}}
		{\GEta{147}{49}{\tau}\GEta{147}{56}{\tau}^{2}\GEta{147}{63}{\tau}}
	-\frac{5}{21}\frac{\GEta{147}{35}{\tau}^{3}\GEta{147}{70}{\tau}}
		{\GEta{147}{49}{\tau}\GEta{147}{56}{\tau}^{3}}
	\\&\quad
	-\frac{5}{42}\frac{\GEta{147}{21}{\tau}\GEta{147}{70}{\tau}^{4}}
		{\GEta{147}{35}{\tau}^{3}\GEta{147}{42}{\tau}\GEta{147}{49}{\tau}}
	-\frac{4}{21}\frac{\GEta{147}{7}{\tau}^{3}\GEta{147}{21}{\tau}}
		{\GEta{147}{49}{\tau}\GEta{147}{63}{\tau}\GEta{147}{70}{\tau}^{2}}
	-\frac{2}{21}\frac{\GEta{147}{14}{\tau}^{3}\GEta{147}{70}{\tau}}
		{\GEta{147}{7}{\tau}^{3}\GEta{147}{49}{\tau}}
	\\&\quad
	-\frac{1}{21}\frac{\GEta{147}{21}{\tau}\GEta{147}{28}{\tau}^{3}\GEta{147}{70}{\tau}}
		{\GEta{147}{14}{\tau}^{3}\GEta{147}{42}{\tau}\GEta{147}{49}{\tau}}
	-\frac{1}{42}\frac{\GEta{147}{21}{\tau}\GEta{147}{56}{\tau}^{3}\GEta{147}{70}{\tau}}
		{\GEta{147}{28}{\tau}^{3}\GEta{147}{49}{\tau}\GEta{147}{63}{\tau}}
	\\
	&=
	\frac{1}{7}
	\frac{\GEta{1}{0}{\tau}^{1/2}\GEta{147}{21}{\tau}\GEta{147}{70}{\tau}\GEta{49}{21}{\tau}}
		{\GEta{147}{0}{\tau}^{1/2}\GEta{49}{7}{\tau}\GEta{49}{14}{\tau}}
	+\frac{1}{7}
	\frac{\GEta{1}{0}{\tau}^{1/2}\GEta{147}{21}{\tau}\GEta{147}{70}{\tau}}
		{\GEta{147}{0}{\tau}^{1/2}\GEta{49}{7}{\tau}}
	\\&\quad
	+\frac{2}{7}
	\frac{\GEta{1}{0}{\tau}^{1/2}\GEta{147}{21}{\tau}\GEta{147}{70}{\tau}{\tau}\GEta{49}{14}{\tau}}
		{\GEta{147}{0}{\tau}^{1/2}\GEta{49}{7}{\tau}\GEta{49}{21}{\tau}}
	-\frac{\GEta{1}{0}{\tau}^{1/2}\GEta{147}{21}{\tau}\GEta{147}{70}{\tau}}
		{\GEta{147}{0}{\tau}^{1/2}\GEta{147}{21}{\tau}\GEta{147}{49}{\tau}\GEta{147}{70}{\tau}}
	\\&\quad
	-\frac{4}{7}
	\frac{\GEta{1}{0}{\tau}^{1/2}\GEta{147}{21}{\tau}\GEta{147}{70}{\tau}}
		{\GEta{147}{0}{\tau}^{1/2}\GEta{49}{14}{\tau}}
	+\frac{5}{7}
	\frac{\GEta{1}{0}{\tau}^{1/2}\GEta{147}{21}{\tau}\GEta{147}{70}{\tau}}
		{\GEta{147}{0}{\tau}^{1/2}\GEta{49}{21}{\tau}}
	\\&\quad
	-\frac{3}{7}
	\frac{\GEta{1}{0}{\tau}^{1/2}\GEta{147}{21}{\tau}\GEta{147}{70}{\tau}\GEta{49}{7}{\tau}}
		{\GEta{147}{0}{\tau}^{1/2}\GEta{49}{14}{\tau}\GEta{49}{21}{\tau}}
	+\frac{\GEta{1}{0}{\tau}^{1/2}\GEta{147}{21}{\tau}\GEta{147}{70}{\tau}}
		{\GEta{147}{0}{\tau}^{1/2}\GEta{147}{14}{\tau}\GEta{147}{49}{\tau}\GEta{147}{63}{\tau}}
.
\end{align*}
However each term is a modular function with respect to $\Gamma_1(147)$ and by
the valence formula it is sufficient to verify the identity in the 
$q$-expansion past $q^{773}$. Thus Theorem \ref{TheoremPP37} holds.

\section{Partition Pair Cranks}

We use $\ell(\pi)$ for the largest part of a partition,
$s(\pi)$ for the smallest part of a partition,
$\#(\pi)$ for the number of parts of a partition,
and $spt(\pi)$ for the number of occurrences of the smallest part of a partition.
For a partition pair $(\pi_1,\pi_2)$
we let $k(\pi_1,\pi_2)$ denote the number of parts of $\pi_1$
that are larger than $s(\pi_1)+\#(\pi_2)$. We note when $\pi_2$ is empty that
$k(\pi_1,\pi_2)$ is the number of parts of $\pi_1$ larger than the smallest part,
so that $spt(\pi_1)+k(\pi_1,\pi_2)=\#(\pi_1)$.
We define the following cranks on partition pairs,
\begin{align*}
	\mbox{paircrank}_1((\pi_1,\pi_2))
	&=	
		 spt(\pi_1) - 1 + k(\pi_1,\pi_2) - \#(\pi_2) 		
	,\\
	\mbox{paircrank}_2((\pi_1,\pi_2))
	&=
		 spt(\pi_1) - 2 +k(\pi_1,\pi_2) - \#(\pi_2) 		
	,\\
	\mbox{paircrank}_3((\pi_1,\pi_2))
	&=
		 spt(\pi_1) - s(\pi_1) - 1 +k(\pi_1,\pi_2) - \#(\pi_2) 		
	,\\
	\mbox{paircrank}_4((\pi_1,\pi_2))
	&=
		 spt(\pi_1) - s(\pi_1) +k(\pi_1,\pi_2) - \#(\pi_2) 		
.
\end{align*}
We let $PP_i$ denote the set of partition pairs counted by $pp_i$. That is,
we first let $PP_1$ denote the set of partition pairs $(\pi_1,\pi_2)$ where
$\pi_1$ is non-empty and if $\pi_2$ is non-empty then $s(\pi_1)\le s(\pi_2)$
and $\ell(\pi_2)\le 2s(\pi_1)$. Next $PP_2$ is the subset of $PP_1$
where $spt(\pi_1)\ge 2$, $PP_3$ is the subset of $PP_1$ where
$spt(\pi_1)\ge s(\pi_1)+1$, and $PP_4$ is the subset of $PP_1$ where
$spt(\pi_1)\ge s(\pi_1)$.
We claim $M_i(m,n)$ is the number of partition pairs of $n$ from $PP_i$ with
$\mbox{paircrank}_i$ equal to m. We prove this for $i=1$ and see the other 
three cases follow similarly. The only rearrangement of $q$-series we need is
the $q$-binomial theorem. This is in a similar fashion to the rearrangements 
for the original crank as in \cite{AndrewsGarvan} and for the
$\overline{\mbox{crank}}$ from Section 3.2 of \cite{GarvanJennings}.

We have
\begin{align*}
	PP_1(z,q)
	&=		
	\sum_{n=1}^\infty \frac{q^n\aqprod{q^{2n+1}}{q}{\infty}}
		{\aqprod{zq^n,z^{-1}q^n}{q}{\infty}}
	\\
	&=
	\sum_{n=1}^\infty \frac{q^n}{\aqprod{zq^n}{q}{\infty}}
	\sum_{k=0}^\infty \frac{\aqprod{zq^{n+1}}{q}{k}z^{-k}q^{nk}}{\aqprod{q}{q}{k}}
	\\
	&=
	\sum_{n=1}^\infty \frac{q^n}{\aqprod{zq^n}{q}{\infty}}
	+
	\sum_{n=1}^\infty\sum_{k=1}^\infty
		\frac{z^{-k}q^{n+kn}}{(1-zq^n)\aqprod{zq^{n+k+1}}{q}{\infty}\aqprod{q}{q}{k}}
	\\
	&=
	\sum_{n=1}^\infty \frac{q^n}{\aqprod{zq^n}{q}{\infty}}
	+
	\sum_{n=1}^\infty\sum_{k=1}^\infty
		\frac{q^{n}}{(1-zq^n)\aqprod{q^{n+1}}{q}{k} \aqprod{zq^{n+k+1}}{q}{\infty}}
		\cdot
		\frac{z^{-k}q^{kn}\aqprod{q}{q}{n+k}}{\aqprod{q}{q}{k}\aqprod{q}{q}{n}}
.
\end{align*}

The first series is the generating function for partitions where the power of 
$q$ counts the number being partitioned and the power of $z$ counts the 
one less than number of parts of the partition. This corresponds to 
$\mbox{paircrank}_1(\pi_1,\pi_2)$ when $\pi_2$ is empty.
For the second series, we interpret the summands as follows.
We have
$\frac{q^{n}}{(1-zq^n)\aqprod{q^{n+1}}{q}{k} \aqprod{zq^{n+k+1}}{q}{\infty}}$
is the generating function for partitions $\pi_1$ with smallest part $n$
with the power of $q$ counting the number being partitioned by $\pi_1$ and
the power of $z$ counting the number of parts of $\pi_1$ that are
either the smallest part (past the first occurrence of the smallest part)
or are at least $n+k+1$ in size. Since
$\frac{\aqprod{q}{q}{n+k}}{\aqprod{q}{q}{k}\aqprod{q}{q}{n}}$ is well known
to be the generating function for partitions into at most $k$ parts with largest
part at most $n$, we have
$\frac{z^{-k}q^{kn}\aqprod{q}{q}{n+k}}{\aqprod{q}{q}{k}\aqprod{q}{q}{n}}$
is the generating function for partitions $\pi_2$ into exactly $k$ parts
with smallest
part at least $n$ and largest part no more than $2n$ with the power of $q$
counting the number being partitioned by $\pi_2$ and the power of $z$ counting
the negative of the number of parts of $\pi_2$. Thus the second series 
corresponds to $\mbox{paircrank}_1(\pi_1,\pi_2)$ when $\pi_2$ is 
non-empty.

\section{Remarks}

We see this method is rather powerful, as from a single Bailey pair we get
a partition type function, congruences for that function, and a 
combinatorial refinement of those congruences. There is still the question of
how these functions behave from a modular perspective. It is not clear
if they also naturally
arise from considering certain modular or weak harmonic mass forms as has been
seen for the spt functions as in \cite{Bringmann} and \cite{BLO1}.  In another 
direction there is also the question of whether or not the series we have
dissected also arise as ranks for some types of partitions.

In a coming paper we find that these functions $PP_i(z,q)$ have representations
as Hecke-Rogers type double series, as was done in \cite{Garvan2} by 
Garvan for other spt cranks. These double series can also be used to prove
most of the congruences of this paper. However, they can not be used to prove
all of the congruences, so the methods of this paper are still of merit. Also,
while the dissections of this paper take some time to develop, the general 
identities can be reused for proving dissections of series of a similar form. 
Also in the coming paper we prove the corresponding results
for the Bailey pairs in groups C and E of Slater \cite{Slater}.
In papers after that, we will handle groups B, F, and G.

\bibliographystyle{abbrv}
\bibliography{BaileyGroupARef}

\end{document}